\documentclass[11pt]{article}    
\usepackage{mathtools}            
\usepackage{amsmath,amssymb,amsthm}
\usepackage{mathptmx}
\usepackage{mathrsfs,bm,bbm}
\usepackage{xcolor}
\usepackage{appendix}
\usepackage{hyperref}
\hypersetup{
	colorlinks,
	linkcolor={red!50!black},
	citecolor={green!50!black},
	urlcolor={blue!80!black}
}

\makeatletter
\newsavebox\myboxA
\newsavebox\myboxB
\newlength\mylenA
\newcommand*\xunderline[2][0.75]{%
	\sbox{\myboxA}{$\m@th#2$}%
	\setbox\myboxB\null
	\ht\myboxB=\ht\myboxA%
	\dp\myboxB=\dp\myboxA%
	\wd\myboxB=#1\wd\myboxA
	\sbox\myboxB{$\m@th\underline{\copy\myboxB}$}
	\setlength\mylenA{\the\wd\myboxA}
	\addtolength\mylenA{-\the\wd\myboxB}%
	\ifdim\wd\myboxB<\wd\myboxA%
	\rlap{\hskip 0.5\mylenA\usebox\myboxB}{\usebox\myboxA}%
	\else
	\hskip -0.5\mylenA\rlap{\usebox\myboxA}{\hskip 0.5\mylenA\usebox\myboxB}%
	\fi}
\makeatother

\newcommand{\mmm}{\mathrel{}\mid\mathrel{}}

\newcommand{\bs}{\bm{\sigma}}
\newcommand{\iii}{i_1,\dots,i_p}
\newcommand{\E}{\mathbb{E}}	
\newcommand{\cQ}{\mathcal{Q}}

\newcommand{\sE}{\mathscr{B}}
\newcommand{\sP}{\mathscr{P}}

\newcommand{\sC}{\mathscr{C}}
\newcommand{\sL}{\mathscr{L}}

\newcommand{\R}{\mathbb{R}}

\newcommand{\tr}{\operatorname{tr}}

\newcommand{\eps}{\epsilon}

\newcommand{\bvs}{\vec{\bm{\sigma}}}

\newcommand{\vs}{\vec{\sigma}}
\newcommand{\vb}{\vec{\beta}}
\newcommand{\vh}{\vec{h}}

\newcommand{\vx}{\vec{x}}
\newcommand{\vex}{\xunderline{x}}
\newcommand{\vbQ}{\xunderline{\bm{Q}}}
\newcommand{\bS}{\mathbb{S}}

\newcommand{\bvb}{\vec{\bm{\beta}}}

\newcommand{\bR}{\bm{R}}

\newcommand{\bQ}{\bm{Q}}
\newcommand{\bA}{\bm{A}}
\newcommand{\bL}{\bm{\Lambda}}

\newcommand{\bd}{\bm{D}}

\newcommand{\bC}{\bm{C}}
\newcommand{\be}{\bm{E}}
\newcommand{\bE}{\overline{\bm{E} }}

\newcommand{\bB}{\bm{B}}

\newcommand{\bI}{\bm{I}}
\newcommand{\bxi}{\bm{\xi}}
\newcommand{\btheta}{\bm{\theta}}
\newcommand{\Sum}{\mathrm{Sum}}
\newcommand{\diag}{\mathrm{diag}}

\newcommand{\trans}{\mathsf{T}}

\renewcommand{\phi}{\varphi}
\renewcommand{\epsilon}{\varepsilon}

\newenvironment{linsys}[2][m]{%
	\setlength{\arraycolsep}{.1111em} 
	\begin{array}[#1]{@{}*{#2}{rc}r@{}} 
	}{%
\end{array}}

\theoremstyle{definition}

\theoremstyle{plain}
\newtheorem{theo}{Theorem}
\newtheorem*{theo*}{Theorem}

\newtheorem{proposition}{Proposition}
\newtheorem{lem}{Lemma}
\theoremstyle{remark}
\newtheorem{rem}{Remark}

\begin{document}

\title{The Crisanti--Sommers Formula \\ for Spherical Spin Glasses with Vector Spins}
\author{Justin Ko\thanks{\textsc{\tiny Department of Mathematics, University of Toronto, jko@math.toronto.edu. Partially supported by NSERC.}
}\\
}
\date{}
\maketitle
\begin{abstract}
We obtain the analogue of the Crisanti--Sommers variational formula for spherical spin glasses with vector spins. This formula is derived from the discrete Parisi variational formula for the limit of the free energy of constrained copies of spherical spin glasses. In vector spin models, the variations of the functional order parameters must preserve the monotonicity of matrix paths which introduces a new challenge in contrast to the derivation of the classical Crisanti--Sommers formula. 
\end{abstract} 
\vspace{0.5cm}

\section{Introduction}

The free energy formula for spherical spin glass models was discovered by Crisanti and Sommers in \cite{crisanti1992sphericalp}. This formula is the analogue of the classical Parisi formula for the Sherrington--Kirkpatrick model \cite{parisi1979infinite,parisi1980sequence} proved in \cite{talagrand2006parisi}. The Parisi formula for the limiting free energy of spherical spin glasses was proven rigorously by Talagrand for even-$p$-spin models in \cite{TSPHERE} and extended to general mixed $p$-spin models by Chen in \cite{CASS}. The equivalence of the Parisi formula and the Crisanti--Sommers formula was proved in \cite{TSPHERE} by showing that both functionals satisfy the same critical point conditions.

In this paper we derive the analogue of the Crisanti--Sommers functional for the spherical vector spin models and show that the limit of the free energy is obtained at the minimum of this functional.  This variational formula for one dimensional vector spins is consistent with the classical Crisanti--Sommers formula.  

Our starting point is the discrete Parisi variational formula for the limit of the free energy of constrained copies of spherical spin glasses proved in \cite{kosphere}. We analyze the critical points of this functional and show a similar reduction can be done in the vector spin case. Unfortunately, the matrix valued functional order parameters in vector spin models may not necessarily have positive definite increments at the minimizer, so the variations can only recover a system of critical point inequalities, which is insufficient to deduce the equivalence of the Crisanti--Sommers and Parisi functionals. To fix this, we will add a barrier function to the functionals that penalizes paths with degenerate increments and study the critical point conditions satisfied by the modified functionals. This approach is explained in more detail in subsection~\ref{subsec:outline}.

The one-dimensional Crisanti--Sommers formula has been studied extensively in the literature. The Parisi and Crisanti--Sommers variational problems were studied in \cite{propertiesparisimeasures,phasediagram,TSPHERE}. The Crisanti--Sommers formula has been applied to derive variational principles for the ground state energy in \cite{parisigroundstate2,parisigroundstate1,dualitygroundstate}. These variational formulas were used to explore related problems such as phase diagrams \cite{rsbspherical,subag2018rsb}, chaos \cite{chenchaos,chen2017temperature} and the geometry of the Gibbs measure \cite{geometrysphere,subag2018free}. The vector spin version of the Crisanti--Sommers formula can be used to study similar questions related to vector spin models.

\subsection{The Limit of the Free Energy and the Parisi Formula}

Multiple copies of mixed even-$p$-spin spherical spin glasses with constrained self overlaps was first studied in \cite{freeenergycostforultrametricity,replicaonandoff}. A rigorous upper bound for the free energy of this model was proved in \cite{PTSPHERE} by Panchenko and Talagrand using the Guerra replica symmetry breaking bound \cite{guerra2003broken}. The sharp lower bound was proved in \cite{kosphere} using the synchronization mechanism described in \cite{panchenko2015free,PVS,PPotts} and the Aizenman--Sims--Starr scheme \cite{AS2} for spherical spin glasses described in \cite{CASS}. These results are a consequence of the ultrametric structure of generalized overlaps that satisfy the Ghirlanda--Guerra identities \cite{GG,guerra1996overlap} which was proved in \cite{PUltra}. Combining the upper and lower bound results in a discrete Parisi variational formula for the free energy of spherical spin glasses with vector spins. 

We start by describing the spherical spin glass model with vector spins and the Parisi formula for the limit of its free energy. Fix integer $n \geq 1$. Let $S_N$ be the sphere in $\R^N$ of radius $\sqrt{N}$. A configuration of $n$ copies of spherical spin glasses can be viewed as vector spins with coordinates restricted to lie on $S_N$,
\begin{equation}
\bvs = (\vs_1, \dots, \vs_N) \in S_N^n \quad \text{where} \quad S_N^n = \big\{ \bvs \in (\R^N)^n \mmm \bs(j) \in S_N \text{ for all } j \leq n \big\}.
\end{equation}
The $j$th coordinate of $\bvs$ is denoted by $\bs(j)$ and the vector entries of $\bvs$ are denoted by
\begin{equation}
\vs_i = \big( \vs_i(1), \dots, \vs_i(n) \big) \in \R^n.
\end{equation}
For $p\geq 2$, the $p$-spin Hamiltonian is denoted by
\begin{equation}
H_{N,p}(\bs(j) ) = \frac{1}{N^{(p-1)/2}} \sum_{1 \leq \iii \leq N} g_{\iii} \vs_{i_1}(j) \cdots \vs_{i_p}(j),
\end{equation}
where $g_{\iii}$ are i.i.d. standard Gaussians for all $p\geq 2$ and indices $(\iii)$. The corresponding mixed $p$-spin Hamiltonian for the $j$th copy at inverse temperatures $(\vb_p)_{p \geq 2}$ is denoted by
\begin{equation}\label{eq:hammixp}
H^j_N(\bvs) = \sum_{p \geq 2} \vb_p(j) H_{N,p}(\bs(j) ).
\end{equation}
We assume that the inverse temperatures satisfy $\sum_{p \geq 2} 2^p \vb_p^{\, 2}(j) < \infty$ for all $j \leq n$, so that \eqref{eq:hammixp} is well-defined, and that $\vb_p = \vec{0}$ for odd $p$. 
The Hamiltonian of $n$ copies of these even mixed $p$-spin models of spherical spin glasses is denoted by
\begin{equation}\label{eq:hamiltonian}
H_N(\bvs) = \sum_{j \leq n} H^j_N(\bvs).
\end{equation}
The overlaps between the vector configurations $\bvs^\ell$ and $\bvs^{\ell'}$ are given by the overlap matrices
\begin{equation}
\bm{R}_{\ell, \ell'} = \bR(\bvs^\ell, \bvs^{\ell'}) = \frac{1}{N}\sum_{i \leq N} \vs_i^{\ell} \otimes \vs_i^{\ell'} \in \bS^n_+
\end{equation}
where $\otimes$ is the outer product on vectors in $\R^n$ and $\bS^n_+$ is the space of $n \times n$ positive semidefinite matrices.

The constraint $\bQ$ is a $n \times n$ symmetric positive definite matrix with off-diagonals $Q^{j,j'} \in [-1,1]$ and diagonals $Q^{j,j} = 1$. Given $\epsilon > 0$, we denote the set of spins with constrained self overlaps by
\begin{align}
Q^\epsilon_N &=  \big\{ \bvs \in S_N^n \mmm \| \bR(\bvs ,\bvs) - \bQ \|_\infty \leq \epsilon \big\},
\end{align} 
where $\| \cdot \|_\infty$ is the infinity norm on $n \times n$ matrices. For an external field $\vh \in \R^n$, we define the free energy as
\begin{equation}\label{eq:freeenergy}
F_N^\epsilon(\bQ) = \frac{1}{N} \E \log \int_{Q_N^\epsilon} \exp \Big( H_N(\bvs) + \sum_{j \leq n} \vh(j) \sum_{i \leq N} \vs_i(j) \Big) \, d \lambda_N^n(\bvs),
\end{equation}
where the reference measure $\lambda_N^n = \lambda_N^{\otimes n}$ is the product of normalized uniform measures $\lambda_N$ on $S_N$.

The limit of $\eqref{eq:freeenergy}$ can be expressed as a Parisi type functional. The Parisi functional is a Lipschitz function of discrete monotone matrix paths encoded by an increasing sequence of real numbers and monotone sequence of $n \times n$ symmetric positive semidefinite matrices,
\begin{equation}\label{eq:xqseq}
\begin{linsys}{8}
 0 & = & x_{0} &\leq &  x_{1} &\leq& \dots&\leq &x_{r - 2}  &\leq &x_{r - 1} &\leq& 1 \\
 \bm 0 &= &\bQ_0 & \leq & \bQ_1  &\leq& \dots&\leq &\bQ_{r - 2} &\leq &\bQ_{r-1} &\leq& \bQ_r &=& \bQ  
\end{linsys}~.
\end{equation}
To lighten notation, we will denote these sequences with $\vex = (x_k)_{k = 0}^{r - 1}$ and $\vbQ = (\bQ_k)_{k = 1}^r$.

For $\bA \in \bS_+^n$, we define the functions
\begin{equation}\label{eq:ximatrix}
\bxi(\bA) = \sum_{p \geq 2} (\vb_p \otimes \vb_p) \odot \bA^{\circ p},
\end{equation}
and
\begin{equation}
\bxi'(\bA) = \sum_{p \geq 2} p (\vb_p \otimes \vb_p) \odot  \bA^{\circ (p - 1)} \quad \text{ and } \quad  \btheta(\bA) = \sum_{p \geq 2} (p - 1) (\vb_p \otimes \vb_p) \odot  \bA^{\circ p},
\end{equation}
where $\odot$ is the Hadamard product on $n \times n$ matrices and $\bA^{\circ p}$ is the $p$th Hadamard power of $\bA$. Since $\vb_p = \vec{0}$ for odd $p$, $\bxi(\cdot)$ is an even convex function in each of its coordinates. The $r$ step discretization of the Parisi functional is defined by
\begin{align}
\sP_r(\bL,\vex, \vbQ) &= \frac{1}{2}\Big[ \langle \vh \vh^\trans, \bL_1^{-1} \rangle  + \langle \bL, \bQ\rangle - n - \log |\bL| +  \sum_{1 \leq k \leq r-1} \frac{1}{x_k} \log \frac{|\bL_{k + 1}|}{|\bL_k|} + \langle \bxi'(\bQ_1), \bL_1^{-1} \rangle \notag \\
&\quad  - \sum_{1 \leq k \leq r-1} x_k \cdot \Sum \big( \btheta (\bQ_{k + 1}) - \btheta (\bQ_{k})\big)\Big] \label{eq:parisi}
\end{align}
where $\langle \bA, \bB \rangle = \tr(\bA \bB )$ is the Frobenius inner product on symmetric matrices, $|\cdot|$ is the determinant and
\begin{equation}\label{def:L}
\bL_{r} = \bL, \quad \bL_{p} = \bL - \sum_{p \leq k \leq r - 1}x_{k} \big( \bxi'( \bQ_{k + 1} ) - \bxi'( \bQ_{k} ) \big) \text{ for $1 \leq p \leq r-1$}.
\end{equation}
The domain of the Parisi functional are all sequences \eqref{eq:xqseq} and Lagrange multipliers $\bL$ such that $|\bL_1| > 0$. This condition also implies that $|\bL_p| > 0$ for all $1 \leq p \leq r- 1$ so \eqref{eq:parisi} is well defined. It was proven in \cite{kosphere} that the limit of the free energy \eqref{eq:freeenergy} is given by minimizing \eqref{eq:parisi}. 
\begin{theo}\label{theo:MAINParisi}\cite[Theorem 2.1]{kosphere}
	The limit of the free energy with self overlaps constrained to $\bQ$ equals
	\begin{equation}
	\lim_{\epsilon \to 0}\lim_{N \to \infty} F^\epsilon_N(\bQ) = \inf_{r,\Lambda,x,Q} \sP_{r}(\bL,\vex, \vbQ).
	\end{equation}
	The infimum is over sequences of the form \eqref{eq:xqseq}, $\bL$ such that $|\bL_1| > 0$, and all $r \geq 1$.
\end{theo}

\subsection{Discrete Form of the Crisanti--Sommers Formula}\label{sec:discCS}

We will show that discrete Parisi formula has a discrete Crisanti--Sommers representation. The discrete form of the Crisanti--Sommers functional is derived by examining the critical points of the discrete Parisi functional \eqref{eq:parisi}. For $r \geq 1$ and the sequence of parameters $\vex$ and $\vbQ$ defined in \eqref{eq:xqseq}, the discrete Crisanti--Sommers representation is given by
\begin{align}
\sC_r(\vex, \vbQ) &= \frac{1}{2}\Big[ \langle\vh \vh^\trans, \bd_1 \rangle + \frac{1}{x_{r - 1}}\log| \bQ - \bQ_{r - 1}| - \sum_{1 \leq k \leq r-2} \frac{1}{x_k} \log \frac{|\bd_{k + 1}|}{|\bd_k|} + \langle \bQ_1, \bd_1^{-1} \rangle \notag \\
&\quad  + \sum_{1 \leq k \leq r-1} x_k \cdot \Sum \big( \bxi (\bQ_{k + 1}) - \bxi (\bQ_{k})\big)\Big]  \label{eq:crisanti},
\end{align}
where,
\begin{equation}\label{def:D}
\bd_p = \sum_{p \leq k \leq r - 1} x_{k} \big( \bQ_{k + 1} - \bQ_{k} \big) \text{  for $1 \leq p \leq r - 1$}.
\end{equation}
We require the additional constraint that $|\bQ - \bQ_{r - 1}| > 0$ otherwise $\sC_r(\vex, \vbQ)$ will be positive infinity. This condition implies that $|\bd_p| > 0$ for all $1 \leq p \leq r-1$ so $\sC_r(\vex, \vbQ)$ is well defined.

We will prove that the representations \eqref{eq:parisi} and \eqref{eq:crisanti} are equal at its minimizers.
\begin{theo}\label{thm:CSformula} For all positive definite constraints $\bQ$, we have
	\[
	\inf_{r,\Lambda,x,Q} \sP_r(\bL,\vex,\vbQ) = \inf_{r,x,Q} \sC_r(\vex,\vbQ),
	\]
	where the first infimum is over sequences \eqref{eq:xqseq} and $\bL \in \bS_+^n$ such such that $|\bL_1| > 0$ and the second infimum is over sequences \eqref{eq:xqseq} such that $|\bd_{r - 1}| > 0$.
\end{theo}

\begin{rem}
	Reassuringly, in the one dimensional case, these formulas agree with the usual discretizations of the Parisi functional and Crisanti--Sommers functional (See \cite[Section~4]{TSPHERE}).
\end{rem}

\subsection{The Integral Form of the Crisanti--Sommers Representation of the Parisi Formula}

The main goal of this paper is to prove that the free energy can be obtained by minimizing a functional closely resembling the Cristanti--Sommers functional for one dimensional spherical spin glasses. The parameters of the functional is the c.d.f. of the trace of the overlap matrix, and the synchronized matrix path identifying the overlap matrix with its trace \cite[Theorem 4]{PVS}. Let 
\begin{equation}\label{eq:csx}
x(t) : [0,n] \to [0,1] \quad \text{such that} \quad x(0) = 0 \quad \text{and} \quad x(n) = 1
\end{equation}
denote a right continuous non-decreasing function and 
\begin{equation}\label{eq:csPhi}
\Phi(t): [0,n] \to \bS^n_+ \quad \text{such that} \quad \tr( \Phi(t)) = t \quad  \text{and} \quad \Phi(0) = \bm 0 \quad \text{and} \quad \Phi(n) = \bQ
\end{equation}
denote a 1-Lipschitz monotone matrix path in the space of $n \times n$ positive semidefinite matrices parametrized by its trace. A monotone matrix path is one with positive semidefinite increments, $\Phi(t_2) - \Phi(t_1) \in \bS_+^n$ for $t_2 \geq t_1$. Since $\Phi$ is 1-Lipschitz in each of its coordinates its coordinate wise derivative $\Phi'$ exists almost everywhere and is bounded by $1$ almost everywhere.

The largest point in the support of the measure associated with the c.d.f. $x(t)$ is denoted by
\[
t_x := x^{-1}(1) = \inf \{ t \in [0,n] \mmm 1 \leq x(t) \}.
\]
Assuming that $|\bQ - \Phi(t_x)| > 0$, we define the quantity
\begin{equation}\label{eq:cs}
\sC(x,\Phi) = \frac{1}{2} \bigg( \int_0^n x(t) \langle \bxi'(\Phi(t)) + \vh \vh^\trans, \Phi'(t) \rangle \, dt + \log| \Phi(n) - \Phi(t_x)| + \int_0^{t_x} \langle \hat \Phi(t)^{-1},  \Phi'(t) \rangle \, dt \bigg),
\end{equation}
where $\hat \Phi(t): [0,n] \to \R^{n \times n}$ is a decreasing matrix path given by
\begin{equation}
\hat \Phi(t) = \int_t^n x(s) \Phi'(s) \, ds.
\end{equation}
Because $\hat \Phi(t) = \bQ - \Phi(t)$ for $t \geq t_x$, the functional does not depend on $t_x$. More precisely, if $\hat t \geq t_x$ and $|\bQ - \Phi(\hat t)| > 0$, then
\begin{align*}
\int_0^{\hat t} \langle \hat \Phi(t)^{-1},  \Phi'(t) \rangle \, dt 
&= \int_0^{t_x} \langle \hat \Phi(t)^{-1},  \Phi'(t) \rangle \, dt + \int_{t_x}^{\hat t} \langle (\bQ - \Phi(t))^{-1},  \Phi'(t) \rangle \, dt
\\&= \int_0^{t_x} \langle \hat \Phi(t)^{-1},  \Phi'(t) \rangle \, dt - \log| \Phi(n) - \Phi(\hat t)| + \log| \Phi(n) - \Phi(t_x)|
\end{align*}
which implies 
\begin{equation}\label{eq:notdependsup}
\log| \Phi(n) - \Phi(\hat t)| + \int_0^{\hat t} \langle \hat \Phi(t)^{-1},  \Phi'(t) \rangle \, dt = \log| \Phi(n) - \Phi(t_x)| + \int_0^{t_x} \langle \hat \Phi(t)^{-1}, \Phi'(t) \rangle \, dt.
\end{equation}
Our main result will be that the limit of the free energy \eqref{eq:freeenergy} is given by minimizing \eqref{eq:cs}. 
\begin{theo}\label{theo:MAINCS} 
	The limit of the free energy with self overlaps constrained to $\bQ$ is
	\begin{equation}
	\lim_{\epsilon \to 0}\lim_{N \to \infty} F^\epsilon_N(\bQ) = \inf_{x,\Phi} \sC(x,\Phi).
	\end{equation}
	The infimum is over $x(t)$ and $\Phi(t)$ defined in \eqref{eq:csx} and \eqref{eq:csPhi} such that $|\bQ - \Phi(t_x)| > 0$. 
\end{theo}

\begin{rem}
	When $n = 1$, \eqref{eq:cs} is identical to the usual 1-dimensional Crisanti--Sommers formula. This is because the only trace parametrization of a one dimensional monotone path is $\Phi(t) = t$. 
\end{rem}

The Crisanti--Sommers form of the functional has some properties that makes it easier to analyze over the Parisi form. First of all, the Lagrange multiplier that appears in \eqref{eq:parisi} is absent in \eqref{eq:cs} and the fixed parameters of the model, $(\vb_p)_{p \geq 1}$ and $\vh$ only appear in the first two terms of \eqref{eq:cs}. We will also prove that $\sC(x,\Phi)$ is a locally Lipschitz (See Lemma~\ref{lem:lip}) with respect to the norm
\[
\| x_1 + x_2 \|_1 + \| \Phi_1 - \Phi_2 \|_\infty := \int_0^n |x_1(t) - x_2(t) | \, dt + \max_{i,j \leq n} \bigg( \sup_{t \in [0,n]} | \Phi_1^{i,j}(t) - \Phi_2^{i,j}(t) | \bigg).
\]
The space of parameters is compact under these norms as a consequence of Prokhorov's theorem and the Arzel\`a--Ascoli theorem. We will use this observation to show that the functional attains its minimum in the interior of its domain, which will allow us to use variational methods to study the minimizers of $\sC(x,\Phi)$. 

\begin{proposition}\label{prop:minimumcompact}
	The Crisanti--Sommers functional attains its minimum on the compact set 
	\begin{equation}
	A_{T,L} = \{ (x,\Phi) \mmm x(t) = 1 \text{ for } t \geq T \text{ and } \| (\bQ - \Phi(t_x) )^{-1} \|_\infty \leq L \}
	\end{equation}
	where the constants
	\begin{equation}\label{eq:cosntantsbound}
	T = n - \frac{1}{\sqrt{n}} e^{-(\langle \vh \vh^\trans + \bxi'(\bQ), \bQ \rangle + n - \log| \bQ| ) }  \quad \text{and} \quad L = \sqrt{n} e^{\langle \vh \vh^\trans + \bxi'(\bQ), \bQ \rangle + n - \log| \bQ| }
	\end{equation}
	only depend on the fixed parameters of the model. 
\end{proposition}

\subsection{Outline of the Paper}\label{subsec:outline}

Following the methodology of the proof in the one dimensional case \cite[Section 4]{TSPHERE}, we will prove that both $\sC_r$ and $\sP_r$ take the same values at the critical points. The minimizers of $\sC_r$ and $\sP_r$ will satisfy the same critical point equations that will allow us to reduce $\sP_r$ to $\sC_r$ and vice versa.

The main difficulty is the minimizer of the Parisi functional in vector spin models may not be an interior point of the domain. In the one dimensional case, we could assume that our the discretization of the paths are strictly monotone, i.e. $q_1 < q_2 < \dots < q_r$. This allowed us to differentiate with respect to $q_k$ to recover the critical point conditions immediately. In the vector spin case, the increments $\bQ_{p} - \bQ_{p- 1}$ may occur on the boundary of the positive definite cone so the directional derivatives are not necessarily equal to $0$ at the critical points. This is our main obstacle, because the system of equations in the critical point conditions becomes a system of inequalities unless we can show that the increments $\bQ_{p} - \bQ_{p- 1}$ are positive definite. 

To fix this, we will introduce a positive definite barrier to the discrete functionals that impose a large penalty if the increments of the matrix path $\vbQ$ are degenerate. This will force the functionals to take a minimum at an interior point, allowing us to use variational calculus to find approximate critical point conditions. This approach will allow us to reduce the Parisi functional into an approximate Crisanti--Sommers form and vice versa. We will use convexity to show that the approximations become exact as the size of the positive definite barrier tends to $0$. 

In Section~\ref{sec:lwbd}, we will show that 
\[
\inf_{r,\Lambda,x,Q} \sP_r (\bL,\vex, \vbQ) \geq \inf_{r,x,Q} \sC_r(\vex, \vbQ),
\]
by using a barrier function to derive approximate critical point conditions for the Parisi functional. We will use the critical point conditions to reduce the Parisi functional into its approximate Crisanti--Sommers form, and we will use convexity to show that the approximate discrete Crisanti--Sommers functional is lower bounded by the usual Crisanti--Sommers functional evaluated at a different point. 

We can use a similar argument with a barrier term to prove the upper bound in Section~\ref{sec:upbd}, 
\[
\inf_{r,\Lambda,x,Q} \sP_r (\bL,\vex, Q) \leq \inf_{r,x,Q} \sC_r(\vex, \vbQ).
\]
This direction of the argument uses the critical point conditions for the Crisanti--Sommers functional. Miraculously, the minimizers of $\sC_r$ satisfies almost exactly the same critical point conditions as the minimizers of $\sP_r$ which allows us to reduce $\sC_r$ to $\sP_r$ in the opposite direction.

In Section~\ref{sec:intform} we will prove that \eqref{eq:cs} is locally Lipschitz to conclude that it is the correct extension of the discrete Crisanti--Sommers formula. Several elementary facts about symmetric matrices and the calculus of matrix valued functionals are included in Appendix~\ref{app:properties}.

\begin{rem}
	If we can show that the minimizers of $\sP_r$ and $\sC_r$ have positive definite increments, then the equality of $\sP_r$ and $\sC_r$ at its critical points can be proved using the same proof as the one-dimensional case without adding the positive definite barrier.
\end{rem}

\section{The Lower Bound of the Parisi Functional}\label{sec:lwbd}

In this section, we will prove that the infimum of the Crisanti--Sommers functional is a lower bound of the Parisi functional:
\begin{lem}\label{lem:lwbd} For any positive definite constraint $\bQ$, we have
	\[
	\inf_{r,\Lambda,x,Q} \sP_r (\bL,\vex, \vbQ) \geq \inf_{r,x,Q} \sC_r(\vex, \vbQ),
	\]
		where the first infimum is over sequences \eqref{eq:xqseq} and $\bL \in \bS_+^n$ such such that $|\bL_1| > 0$ and the second infimum is over sequences \eqref{eq:xqseq} such that $|\bd_{r - 1}| > 0$.
\end{lem}

Without loss of generality, we will assume that $\vb_2 > 0$. This assumption implies that all entries of $\bxi''(\bA)$ are positive for all symmetric matrices $\bA$, so we don't have to worry about dividing by $0$ in the derivation of the critical point conditions. We can make this assumption because both the infimums of $\sC_r$ and $\sP_r$ are uniformly continuous with respect to $(\vb_p)_{p \geq 2}$ [\ref{app:properties}, Proposition~\ref{prop:unifcontC} and Proposition~\ref{prop:unifcontP}] so we can send $\vb_2 \to 0$ to recover the result in the general case.

To simplify notation, we may also fix $x_{r - 1} = 1$. This won't affect the global infimum because the closure of paths satisfying $x_{r - 1} = 1$ is equal to the closure of paths satisfying $x_{r - 1} \leq 1$. 

It remains to prove that for fixed sequences $\vex$ and $r \geq 2$,
\[
\inf_{\Lambda,Q} \sP_r(\bL, \vex, \vbQ) \geq \inf_{r,x,Q} \sC_r(\vex, \vbQ).
\]
We prove this by examining the behavior of $\sP_r$ at its critical point. We will perturb $\sP_r$ by adding a logarithmic penalty at the boundary to force the minimizer of $\sP_r$ to have positive definite increments. The minimizer will satisfy an interior critical point condition that will allow us to reduce the perturbed functional $\sP_r^\epsilon$  into a perturbed $\sC_r^\epsilon$ functional.  These perturbed functionals will converge to $\sP_r$ and $\sC_r$ in the limit as the size of the barrier tends to $0$ using a convexity argument.  

\subsection{Adding a Positive Definite Barrier}
We fix $r \geq 2$ and let $\vbQ = (\bQ_k )_{k = 0}^r$ denote the monotone sequence of matrices such that $\bQ_0 = \bm{0}$ and $\bQ_{r} = \bQ$.  We begin our proof by modifying $\sP_r$ with a logarithmic barrier term that assigns infinitely large penalties if $\vbQ$ is not strictly increasing. Let $\epsilon > 0$, and consider the barrier function
\[
\sE_r( \vbQ ) = - \sum_{0 \leq k \leq r - 1} \log |\bQ_{k + 1} - \bQ_k|.
\]
Since $| \bA | \leq (\frac{\tr(A)}{n})^n$ [\ref{app:properties}, Proposition~\ref{prop:AMGM}] for all $k \leq r - 1$, we have $|\bQ_{k + 1} - \bQ_k| \leq 1$ so $\sE_r \geq 0$. Furthermore, $\sE_r \to +\infty$ if $|\bQ_{k + 1} - \bQ_k| \to 0$ for some $0 \leq k \leq r-1$. 

For a fixed strictly increasing sequence such that
\begin{equation}\label{eq:xseqlwbd}
0 = x_0 < x_1 < \dots < x_{r-1} = 1,
\end{equation}
we define the functional, 
\begin{align}
\sP_r^\epsilon(\bL, \vbQ) &= \frac{1}{2}\Big[ \langle \vh \vh^\trans, \bL_1^{-1} \rangle + \langle \bL,\bQ \rangle - n - \log |\bL| +  \sum_{1 \leq k \leq r-1} \frac{1}{x_k} \log \frac{|\bL_{k + 1}|}{|\bL_k|} + \langle\bxi'(\bQ_1) , \bL_1^{-1} \rangle \notag \\
&\quad  - \sum_{1 \leq k \leq r-1} x_k \cdot \Sum \big( \btheta (\bQ_{k + 1}) - \btheta (\bQ_{k})\big)  - \epsilon \sum_{0 \leq k \leq r - 1} \log |\bQ_{k + 1} - \bQ_k| \Big] \label{eq:pertparisi}.
\end{align}
Notice that $\sP^\epsilon_r = \sP_r + \epsilon \sE_r$ and it decreases pointwise to $\sP(\bL,\vex,\vbQ)$ on its domain as $\epsilon \to 0$, where $\vex$ is the fixed monotone sequence \eqref{eq:xseqlwbd}. The barrier term forces the minimizers to lie in the interior of the positive definite cone, since $\sP_r^\epsilon(\bL, \vbQ ) \to +\infty$ if one of the increments $|\bQ_{k + 1} - \bQ_k| \to 0$. We now examine the behavior of $\sP_r^\epsilon(\bL,\vbQ)$ at its minimizers and recover a system of critical point equations. 

\subsection{Critical Point Equations} We will study the first variation of $\sP_r^\epsilon$ to recover critical point conditions for its minimizer. These critical point conditions will relate the increments of $(\bL_k)_{k = 1}^{r}$ and $(\bd_k)_{k = 1}^{r - 1}$. We want to minimize the function
\[
\sP_r^\epsilon ( \bL, \vbQ ) := \sP_r^\epsilon ( \bL, \bQ_1, \dots, \bQ_{r - 1})
\]
over the parameters
\[
\bL \in \sL := \big\{ \bL \in \bS_+^n \mmm |\bL_1| > 0 \big\} 
\]
and
\[
(\bQ_k)_{k = 1}^{r-1} \in \cQ_r := \big\{ \bQ_1, \dots, \bQ_{r-1} \in \bS_+^n \mmm |\bQ_{k + 1} - \bQ_k| > 0, ~\forall~ 0 \leq k \leq r - 1 \big\}
\]
where $\bS_+^n$ is the space of positive semidefinite $n \times n$ matrices. By compactness, $\sP_r^\epsilon$ attains its minimum at some $\bL \in \sL$ and $(\bQ_k)_{k = 1}^{r-1} \in \cQ_r$. Since $\sE( \vbQ ) = \infty$ if the increments are not positive definite, $(\bQ_k)_{k = 1}^{r-1}$ must have also have positive definite increments,
\[
|\bQ_{k + 1} - \bQ_k | > 0 \quad \forall~ 0 \leq k \leq r - 1.
\]
This implies that symmetric matrices are admissible variations of $\bL$ and $\vbQ$ [\ref{app:properties}, Proposition~\ref{prop:pertposdef}]. In particular, if $\bC$ is a symmetric matrix, then for all $t$ sufficiently small,
\[
\bL + t \bC \in \sL,
\]
and for $1 \leq k \leq r- 1$,
\[
(\bQ_1, \dots, \bQ_k + t \bC, \dots \bQ_{r-1}) \in  \cQ_r.
\]
If $\bL \in \sL$ and $\vbQ \in \cQ_r$ is a minimizer of $\sP^\epsilon$, then for all $1 \leq k \leq r- 1$, 
\[
\frac{d}{dt} \sP_r^\epsilon( \bL + t \bC, \vbQ ) \Big|_{t = 0} = 0 \quad\text{and}\quad \frac{d}{dt} \sP_r^\epsilon( \bL, \bQ_1, \dots, \bQ_k + t \bC, \dots \bQ_{r-1} ) \Big|_{t = 0} = 0.
\]
We can conclude the directional derivatives must be equal to $0$ because both $\bC$ and $-\bC$ are admissible variations. We can compute the first variation of the functionals explicitly by computing the matrix derivatives of $\sP_r^\epsilon$ and derive some critical point conditions on the minimizers:

	\noindent(a)~~~ The directional derivatives of $\sP_r^\epsilon$ with respect to $\bL$ in the symmetric direction $2\bC$ is
	\begin{align}
	\partial_{\bL} \sP_r^\epsilon &= \langle \bQ, \bC \rangle - \langle \bL^{-1} , \bC \rangle - \langle \bL_1^{-1}  \big(\vh \vh^T + \bxi'(\bQ_1) \big) \bL_1^{-1}, \bC \rangle  + \sum_{1 \leq k \leq r - 1} \frac{1}{x_k} \langle\bL_{k + 1}^{-1} - \bL_k^{-1}, \bC \rangle  \label{eq:derivParisiL}.
	\end{align}
	At the minimizer, we require
	\[
	\partial_{\bL} \sP_r^\epsilon := \frac{d}{dt} \sP_r^\epsilon( \bL + 2t \bC, \vbQ ) \Big|_{t = 0} = 0.
	\]
	This equality holds for all symmetric directions $\bC$, so the minimizer must satisfy the equation  [\ref{app:properties}, Proposition~\ref{prop:symmetrictest}]
	\begin{equation}\label{eq:lwbdrcond}
	\bQ = \bL^{-1} + \bL_1^{-1}  \big(\vh \vh^T + \bxi'(\bQ_1) \big) \bL_1^{-1} + \sum_{1 \leq k \leq r - 1} \frac{1}{x_k}  (  \bL_k^{-1} - \bL_{k + 1}^{-1}).
	\end{equation}
	
	\noindent (b)~~~ For $2 \leq p \leq r - 1$, the directional derivatives of $\sP_r^\epsilon$ with respect to $\bQ_p$ in the symmetric direction $2\bC$ is [Appendix~\ref{app:derivativesP}]
	\begin{align}
	\partial_{\bQ_p} \sP_r^\epsilon &= -(x_p - x_{p - 1}) \langle \bL_1^{-1}  \big( \vh \vh^T + \bxi'(\bQ_1) \big) \bL_1^{-1} , \bxi''(\bQ_p) \odot \bC \rangle \notag
	\\&\quad -(x_p - x_{p - 1})  \sum_{1 \leq k \leq p - 1} \frac{1}{x_k} \langle\bL^{-1}_{k} - \bL^{-1}_{k + 1} , \bxi''(\bQ_p) \odot \bC \rangle \notag
	\\&\quad + (x_p - x_{p - 1}) \langle \bQ_{p}, \bxi''(\bQ_p) \odot \bC \rangle  \notag
	\\&\quad + \epsilon \langle (\bQ_{p+ 1 } - \bQ_{p} )^{-1} , \bC \rangle   - \epsilon \langle( \bQ_{p} - \bQ_{p - 1} )^{-1} , \bC \rangle.\label{eq:derivParisiQp}
	\end{align}
	At the minimizer, we require
	\[
	\partial_{\bQ_p} \sP_r^\epsilon :=  \frac{d}{dt} \sP_r^\epsilon( \bL, \bQ_1, \dots, \bQ_p + 2t \bC, \dots, \bQ_{r-1} ) \Big|_{t = 0} = 0.
	\]
	This equality holds for all symmetric directions $\bC$, so the minimizer satisfies the critical point equation
	\begin{align}
	\bQ_{p} &= \bL_1^{-1}  \big( \vh \vh^T + \bxi'(\bQ_1) \big) \bL_1^{-1} + \sum_{1 \leq k \leq p - 1} \frac{1}{x_k} (\bL^{-1}_{k} - \bL^{-1}_{k + 1}) - \epsilon \be_p \label{eq:lwbdpcond}
	\end{align}
	where
	\begin{equation}\label{eq:errorlwbd}
	\be_p := \frac{1}{x_p - x_{p-1}} \Big(  (\bQ_{p + 1} - \bQ_{p} )^{-1} - (\bQ_{p} - \bQ_{p - 1} )^{-1} \Big) \oslash \bxi''(\bQ_p).
	\end{equation}
	The notation $\oslash$ refers to the Hadamard division operation (entry-wise division). $\be_p$ is well defined since the fixed $(x_p)_{p = 1}^r$ in \eqref{eq:xseqlwbd} is strictly monotone and $\vb_2 > 0$ so all entries of $\bxi''(\bA)$ is positive. 
	
	\noindent (c)~~~ For $p = 1$, the directional derivatives of $\sP_r^\epsilon$ with respect to $\bQ_1$ in the symmetric direction $2\bC$ is [Appendix~\ref{app:derivativesP}]
	\begin{align}
	\partial_{\bQ_1} \sP_r^\epsilon &= -x_1 \langle \bL_1^{-1}  \big( \vh \vh^T + \bxi'(\bQ_1) \big) \bL_1^{-1} , \bxi''(\bQ_1) \odot \bC \rangle + x_1 \langle \bQ_{1} ,  \bxi''(\bQ_1) \odot \bC \rangle \notag
	\\&\quad + \epsilon \langle  ( \bQ_{2} - \bQ_{1} )^{-1} , \bC \rangle  - \epsilon \langle \bQ_{1}^{-1} , \bC \rangle \label{eq:derivParisiQ1} .
	\end{align}
	At the minimizer, we require
	\[
	\partial_{\bQ_1} \sP_r^\epsilon :=  \frac{d}{dt} \sP_r^\epsilon( \bL, \bQ_1 + 2t\bC, \dots, \bQ_{r - 1} ) \Big|_{t = 0} = 0.
	\]
	This equality holds for all symmetric directions $\bC$, so the minimizer satisfies the critical point equation
	\begin{equation}\label{eq:lwbd1cond}
	\bQ_{1} = \bL_1^{-1}  \big( \vh \vh^T + \bxi'(\bQ_1) \big) \bL_1^{-1} - \epsilon \be_1,
	\end{equation}
	where $\be_1$ is given by the formula in \eqref{eq:errorlwbd} with $p = 1$. 

For $1 \leq p \leq r- 1$, the critical point equations \eqref{eq:lwbdpcond} and \eqref{eq:lwbd1cond} can be expressed as
\begin{equation}\label{eq:lwbdcriticalpointQ}
\bQ_p = \bL_1^{-1} (\vh \vh^T + \bxi'(\bQ_1) ) \bL_1^{-1} + \sum_{1 \leq k \leq p - 1} \frac{1}{x_k} (\bL^{-1}_{k} - \bL^{-1}_{k + 1}) - \epsilon \be_p
\end{equation}
where $\be_r := \bm{0}$ and $\be_p$ was defined in \eqref{eq:errorlwbd}. These critical point conditions can be used to relate $\bL_k$ in $\sP_r^\epsilon$ with the $\bd_k$ terms in $\sC_r$. Taking differences of the critical point conditions \eqref{eq:lwbdrcond} and \eqref{eq:lwbdcriticalpointQ}, we can conclude that,
\begin{align*}
x_{r-1}(\bQ - \bQ_{r - 1}) &= x_{r-1}\bL^{-1} + (\bL_{r - 1}^{-1} - \bL^{-1})  - \epsilon x_{r-1} ( \be_r - \be_{r-1} )
\\&= \bL_{r - 1}^{-1} - \epsilon x_{r-1} ( \be_r - \be_{r-1} )
\end{align*}
since $x_{r - 1} = 1$, and for $1 \leq p \leq r - 2$, we can conclude that
\begin{align*}
x_{p} (\bQ_{p+1} - \bQ_{p}) &= \bL^{-1}_{p} - \bL_{p+1}^{-1} - \epsilon x_{p} ( \be_{p+1} - \be_{p} ).
\end{align*}
Taking sums of the above, we have for $1 \leq p \leq r-1$, 
\begin{align}
\sum_{p \leq k \leq r-1} x_k (\bQ_{k + 1} - \bQ_k) + \epsilon \sum_{p \leq k \leq r-1} x_k (\be_{k + 1} - \be_k) = \bd_{p} + \epsilon \sum_{p \leq k \leq r-1} x_k (\be_{k + 1} - \be_k) &= \bL_{p}^{-1}.
\end{align}
We will summarize this critical point condition in the following lemma.
\begin{lem}\label{lem:lwbdcritpt} For fixed $r \geq 2$, if $\bL$ and $\vbQ$ is a minimizer of $\sP_r^\epsilon(\bL,\vbQ)$ and $\vb_2 > 0$, then $\bL$ and $\vbQ$ satisfy the following critical point equations
	\begin{equation}\label{eq:critptlwbd}
	\bL_{p}^{-1} = \bd_p(\epsilon) \qquad \text{ for } 1 \leq p \leq r- 1
	\end{equation}
	where
	\begin{equation}\label{eq:lwbderror}
	\bd_p(\epsilon) = \bd_p + \epsilon \bE_p \qquad \text{ and } \qquad \bE_p = \sum_{p \leq k \leq r-1} x_k (\be_{k + 1} - \be_k).
	\end{equation}
\end{lem}

\subsection{Reduction to an approximate Crisanti--Sommers functional} In this subsection, we will reduce $\sP_r^\epsilon$ defined in \eqref{eq:pertparisi} to an approximate Cristanti--Sommers functional. If $\bL$ and $\vbQ$ satisfy the critical point conditions \eqref{eq:critptlwbd}, we will show that $\sP_r^\epsilon$ can be reduced to
\begin{align}
\sC_r^\epsilon(\vbQ) &= \frac{1}{2}\Big[ \langle \vh \vh^\trans, \bd_1(\epsilon) \rangle + \frac{1}{x_{r - 1}}\log| \bd_{r - 1}(\epsilon) | - \sum_{1 \leq k \leq r-2} \frac{1}{x_k} \log \frac{|\bd_{k + 1}(\epsilon)|}{|\bd_k(\epsilon)|} + \langle \bQ_1 , \bd_1^{-1}(\epsilon) \rangle \notag \\
&\quad + \sum_{1 \leq k \leq r-1} x_k \cdot \Sum \big( \bxi (\bQ_{k + 1}) - \bxi (\bQ_{k})\big) \notag \\
&\quad - \epsilon \sum_{1 \leq k \leq r-2} \langle\bE_{k+1} - \bE_k , \bxi' (\bQ_{k + 1}) \rangle - \sum_{1 \leq k \leq r-2} \frac{\epsilon}{x_{k}}  \langle \bd_{k+1}^{-1}(\epsilon) , \bE_{k} -  \bE_{k+1} \rangle \notag \\
&\quad -\epsilon \langle\bd_{r - 1}^{-1}(\epsilon), \bE_{r-1}\rangle + \epsilon \langle \bxi'(\bQ_{r - 1}), \bE_{r-1} \rangle  - \epsilon \sum_{0 \leq k \leq r - 1} \log | \bQ_{k + 1} - \bQ_k |  \Big]. \label{eq:approxcs}
\end{align}
Notice that $\sC_r^\epsilon(\vbQ)$ is of the same form as $\sC_r(\vex,\vbQ)$, but with $\bd_k$ replaced by $\bd_k(\epsilon)$ and some additional error terms. If we set $\epsilon = 0$, then the error terms in the second line all vanish and we are left with the usual $\sC_r(\vex,\vbQ)$ functional. In the next subsection, we will show that we can bound the minimum of $\sC_r^\epsilon(\vbQ)$ with $\sC_r$ evaluated at a different path to remove the error terms. 

\begin{lem}\label{lem:lwbdreduction}
	If $\bL$ and $\vbQ$ satisfy the critical point conditions \eqref{eq:critptlwbd}, then
	\[
	\sP_r^\epsilon( \bL ,\vbQ) = \sC_r^\epsilon(\vbQ).
	\]
\end{lem}

\begin{proof}
	The reduction of $\sP_r^\epsilon$ to $\sC_r^\epsilon$ is a straightforward, but tedious computation. We will show
	\[
	2 (\sP_r^\epsilon( \bL ,\vbQ) - \sC_r^\epsilon(\vbQ) ) = 0.
	\]
	If \eqref{eq:critptlwbd} holds, then
	\begin{alignat}{2}
	\bL_{k + 1} - \bL_{k}  &= x_k (\bxi'(\bQ_{k + 1}) - \bxi'(\bQ_{k})) &&\qquad 1 \leq k \leq r - 1 \label{eq:lwbd1}\\
	\bd_{k} - \bd_{k + 1}  &= x_k (\bQ_{k + 1} - \bQ_{k}) &&\qquad 1 \leq k \leq r - 1\label{eq:lwbd2}\\
	\bL_{k}^{-1} &= \bd_k(\epsilon)  &&\qquad 1 \leq k \leq r - 1\label{eq:lwbd3}
	\end{alignat}
	where $\bd_{r} := \bm 0$. These identities will be used multiple times throughout this proof. 
	
	We begin by observing that the external fields cancel if \eqref{eq:critptlwbd} holds,
	\begin{equation}\label{eq:lwbdfield}
	\langle \vh \vh^\trans , \bL_1^{-1} \rangle \stackrel{\mathclap{\eqref{eq:lwbd3}}}{=} \langle \vh \vh^\trans, \bd_1(\epsilon) \rangle.
	\end{equation}
	Next, we simplify the summation of the logarithm terms in $\sP_r^\epsilon$ using the fact $x_{r-1} = 1$,
	\begin{align}
	- \log |\bL| + \sum_{1 \leq k \leq r-1} \frac{1}{x_k} \log \frac{|\bL_{k + 1}|}{|\bL_k|} &= - \log |\bL| + \frac{1}{x_{r-1}} ( \log|\bL| - \log|\bL_{r-1}| ) + \sum_{1 \leq k \leq r-2} \frac{1}{x_k} \log \frac{|\bL_{k + 1}|}{|\bL_k|} \notag
	\\&\stackrel{\mathclap{\eqref{eq:lwbd3}}}{=} \frac{1}{x_{r - 1}} \log|\bd_{r-1} (\epsilon)| - \sum_{1 \leq k \leq r-2} \frac{1}{x_k} \log \frac{|\bd_{k + 1}(\epsilon)|}{|\bd_k(\epsilon)|} \label{eq:lwbdlog}.
	\end{align}
	Therefore, the log determinant terms in $\sP_r^\epsilon$ and $\sC_r^\epsilon$ also cancel. 
	
	Since $\btheta(\bA) = \bA \odot \bxi'(\bA) - \bxi(\bA)$ and $\Sum(\bA \odot \bB) =  \langle \bA,  \bB\rangle$, the remaining terms in $2(\sP^\epsilon_r - \sC_r^\epsilon)$ are
	\begin{align}
	&\label{eq:lwbdsummation}\quad - \sum_{1 \leq k \leq r-2} x_k \Big( \langle\bQ_{k + 1} , \bxi' (\bQ_{k + 1}) \rangle  - \langle\bQ_{k}, \bxi' (\bQ_{k}) \rangle \Big)  - \Big( \langle \bQ, \bxi' (\bQ) \rangle - \langle \bQ_{r-1}, \bxi' (\bQ_{r-1}) \rangle \Big)\\ 
	& \label{eq:lwbdcan2}\quad + \epsilon \sum_{1 \leq k \leq r-2} \langle\bE_{k+1} - \bE_k , \bxi' (\bQ_{k + 1}) \rangle + \sum_{1 \leq k \leq r-2} \frac{\epsilon}{x_{k}}  \langle \bd_{k+1}^{-1}(\epsilon) , \bE_{k} -  \bE_{k+1} \rangle\\
	&\label{eq:lwbdcan3}\quad + \langle \bL,\bQ \rangle - n + \langle\bxi'(\bQ_1) , \bL_1^{-1} \rangle - \langle \bQ_1 , \bd_1^{-1}(\epsilon) \rangle +\epsilon \langle\bd_{r - 1}^{-1}(\epsilon), \bE_{r-1}\rangle  - \epsilon \langle \bxi'(\bQ_{r - 1}), \bE_{r-1} \rangle .
	\end{align}
	We will show that \eqref{eq:lwbdsummation} will cancel \eqref{eq:lwbdcan2} and \eqref{eq:lwbdcan3} at the critical point. We start by simplifying the summation term in \eqref{eq:lwbdsummation},
	\begin{align}
	&\quad  - \sum_{1 \leq k \leq r-2} x_k \Big( \langle\bQ_{k + 1} , \bxi' (\bQ_{k + 1}) \rangle  - \langle\bQ_{k}, \bxi' (\bQ_{k}) \rangle \Big)  \label{eq:sumfirstr1}
	\\&= -\sum_{1 \leq k \leq r-2} \Big( \langle x_k(\bQ_{k + 1} - \bQ_{k}), \bxi' (\bQ_{k + 1}) \rangle + \langle \bQ_{k}, x_k (\bxi' (\bQ_{k + 1}) - \bxi' (\bQ_{k}) ) \rangle \Big) \notag
	\\&\stackrel{\mathclap{\eqref{eq:lwbd1} \eqref{eq:lwbd2}}}{=} - \sum_{1 \leq k \leq r-2} \Big( \langle \bd_{k} - \bd_{k + 1}, \bxi' (\bQ_{k + 1}) \rangle  + \langle \bQ_{k}, \bL_{k + 1} - \bL_{k} \rangle \Big) \notag
	\\&\stackrel{\mathclap{\eqref{eq:lwbd3}}}{=}- \sum_{1 \leq k \leq r-2} \Big( \langle \bd_{k}(\epsilon) - \bd_{k + 1}(\epsilon), \bxi' (\bQ_{k + 1}) \rangle  + \langle \bQ_{k},  \bL_{k + 1} - \bL_{k} \rangle  \Big) - \epsilon \sum_{1 \leq k \leq r-2}  \langle \bE_{k+1} - \bE_k, \bxi' (\bQ_{k + 1}) \rangle. \label{eq:summationp4}
	\end{align}
	Using summation by parts and \eqref{eq:critptlwbd}, the first summation in \eqref{eq:summationp4} is equal to
	\begin{align}
	&-\sum_{1 \leq k \leq r-2} \Big( \langle \bd_{k}(\epsilon) , \bxi' (\bQ_{k + 1}) - \bxi' (\bQ_{k}) \rangle  + \langle \bd_{k + 1}^{-1}(\epsilon), \bQ_{k} - \bQ_{k + 1} \rangle  \Big) \label{eq:summationp2}
	\\&- \langle \bd_1(\epsilon), \bxi'(\bQ_1) \rangle + \langle \bd_{r - 1}(\epsilon), \bxi'(\bQ_{r - 1}) \rangle - \langle \bd^{-1}_{r - 1}(\epsilon) , \bQ_{r- 1} \rangle + \langle \bd_{1}^{-1}(\epsilon), \bQ_1 \rangle \label{eq:summationp2-2}.
	\end{align}
	The critical point conditions \eqref{eq:critptlwbd} implies
	\begin{align*}
	\bxi'( \bQ_{k+1} ) - \bxi'( \bQ_{k} )  &\stackrel{\mathclap{\eqref{eq:lwbd1}}}{=} \frac{1}{x_{k}}( \bL_{k + 1} -  \bL_{k}) 
	\\&\stackrel{\mathclap{\eqref{eq:lwbd3}}}{=} \frac{1}{x_{k}} \bd_{k}^{-1}(\epsilon)( \bd_{k}(\epsilon) -  \bd_{k+1}(\epsilon))\bd_{k+1}^{-1} (\epsilon)
	\\&\stackrel{\mathclap{\eqref{eq:lwbd2}}}{=} \bd_{k}^{-1}(\epsilon)( \bQ_{k+1} -  \bQ_{k})\bd_{k+1}^{-1}(\epsilon) + \frac{\epsilon}{x_{k}} \bd_{k}^{-1}(\epsilon)( \bE_{k} -  \bE_{k+1})\bd_{k+1}^{-1}(\epsilon),
	\end{align*}
	which combined with the fact $\tr(\bA \bB \bC) = \tr(\bC \bA \bB)$ implies the summation term \eqref{eq:summationp2} simplifies to 
	\begin{equation}\label{eq:summationp3}
	- \sum_{1 \leq k \leq r-2} \frac{\epsilon}{x_{k}} \langle \bd_{k+1}^{-1}(\epsilon) ,  \bE_{k} -  \bE_{k+1} \rangle.
	\end{equation}
	Substituting \eqref{eq:summationp3} into \eqref{eq:summationp4} and adding the boundary terms \eqref{eq:summationp2-2} implies that
	\begin{align}
	\eqref{eq:sumfirstr1} &=  - \epsilon \sum_{1 \leq k \leq r-2} \langle\bE_{k+1} - \bE_k , \bxi' (\bQ_{k + 1}) \rangle -\epsilon \sum_{1 \leq k \leq r-2} \frac{1}{x_{k}} \langle \bd_{k+1}^{-1}(\epsilon), \bE_{k} -  \bE_{k+1} \rangle \notag
	\\&\quad - \langle \bd_1(\epsilon), \bxi'(\bQ_1) \rangle + \langle \bd_{r - 1}(\epsilon), \bxi'(\bQ_{r - 1}) \rangle - \langle\bd^{-1}_{r - 1}(\epsilon), \bQ_{r- 1} \rangle + \langle \bd_{1}^{-1}(\epsilon), \bQ_1 \rangle. \label{eq:lwbdsummationsimplified}
	\end{align}
	
	Substituting \eqref{eq:lwbdsummationsimplified} into \eqref{eq:lwbdsummation} implies
	\begin{align}
	2 (\sP_r^\epsilon - \sC_r^\epsilon ) &= \langle \bL, \bQ\rangle - n + \langle \bxi'(\bQ_1), \bL_1^{-1}  \rangle - \langle \bQ_1, \bd_1^{-1}(\epsilon) \rangle  +\epsilon \langle \bd_{r - 1}^{-1}(\epsilon), \bE_{r-1} \rangle - \epsilon \langle \bxi'(\bQ_{r - 1}) , \bE_{r-1} \rangle \notag
	\\&\quad - \langle \bd_1(\epsilon), \bxi'(\bQ_1)  \rangle + \langle \bd_{r - 1}(\epsilon), \bxi'(\bQ_{r - 1}) \rangle - \langle \bd^{-1}_{r - 1}(\epsilon), \bQ_{r- 1}  \rangle + \langle \bd_{1}^{-1}(\epsilon), \bQ_1 \rangle \notag
	\\&\quad - \langle \bQ, \bxi' (\bQ) \rangle + \langle \bQ_{r-1} , \bxi' (\bQ_{r-1}) \rangle \notag 
	\\& \stackrel{\mathclap{\eqref{eq:lwbd3}}}{=} \langle \bL, \bQ \rangle - \tr(\bI) + \epsilon \langle \bd_{r - 1}^{-1}(\epsilon), \bE_{r-1} \rangle - \epsilon \langle \bxi'(\bQ_{r - 1}) , \bE_{r-1} \rangle  \notag
	\\&\quad + \langle \bd_{r - 1}(\epsilon), \bxi'(\bQ_{r - 1}) \rangle - \langle \bd^{-1}_{r - 1}(\epsilon), \bQ_{r- 1} \rangle - \langle \bQ, \bxi' (\bQ) \rangle + \langle \bQ_{r-1} , \bxi' (\bQ_{r-1}) \rangle \notag
	\\&= \langle \bL, \bQ \rangle - \tr(\bI) +\epsilon \langle\bd_{r - 1}^{-1}(\epsilon), \bE_{r-1} \rangle + \langle \bQ, \bxi'(\bQ_{r - 1}) \rangle - \langle\bd^{-1}_{r - 1}(\epsilon),  \bQ_{r- 1} \rangle  - \langle \bQ , \bxi'(\bQ) \rangle \label{eq:lwbdremainingterms}.
	\end{align}
	since $\bd_{r - 1}(\epsilon) = \bQ - \bQ_{r - 1} + \epsilon \bE_{r-1}$. We will show that the $\langle \bd^{-1}_{r - 1}(\epsilon),\bQ_{r- 1} \rangle$ term cancels all the remaining terms. Using the critical point condition and the definitions of $\bL_{r - 1}$ defined in \eqref{def:L} and $\bd_{r - 1}(\epsilon)$ defined in \eqref{def:D} and \eqref{eq:lwbderror}, we get
	\begin{align*}
	\bd^{-1}_{r - 1}(\epsilon) \bQ_{r- 1} &= \bd_{r  - 1}^{-1}(\epsilon) ( -\bd_{r - 1}(\epsilon) + \bQ - \epsilon \bE_{r - 1} )
	\\&= - \bI +  \bd_{r - 1}^{-1}(\epsilon) \bQ + \epsilon \bd_{r - 1}^{-1}(\epsilon) \bE_{r-1}
	\\&\stackrel{\mathclap{\eqref{eq:lwbd3}}}{=} - \bI +  (\bL - \bxi'(\bQ) + \bxi'(\bQ_{r - 1})) \bQ + \epsilon \bd_{r - 1}^{-1}(\epsilon) \bE_{r-1}.
	\end{align*}
	Taking the trace and using the fact $\tr(\bA \bB) = \tr (\bB \bA)$ implies 
	\begin{equation}\label{eq:tracep}
	\langle \bd^{-1}_{r - 1}, \bQ_{r- 1} \rangle = -\tr(\bI) + \langle \bL, \bQ \rangle - \langle \bQ, \bxi'(\bQ) \rangle + \langle \bQ , \bxi'(\bQ_{r - 1}) \rangle + \epsilon \langle \bd_{r - 1}^{-1}(\epsilon), \bE_{r-1} \rangle.
	\end{equation}
	Substituting \eqref{eq:tracep} into \eqref{eq:lwbdremainingterms} cancels out all remaining terms, so
	\[
	2 \big(\sP_r^\epsilon( \bL ,\vbQ) - \sC_r^\epsilon(\vbQ) \big) = 0.
	\]
\end{proof}

\subsection{Removing the Error Terms}We now bound the minimum of the perturbed functional $\sC_r^\epsilon(\vbQ)$ defined in \eqref{eq:approxcs} with $\sC_r$ evaluated at a different path of matrices.  We can't simply send $\epsilon \to 0$ to remove the error terms, because we do not know that $\epsilon \be_k \to \bm 0$ since $\be_k$ depends on $\epsilon$. Consider the monotone path encoded by the sequences
\begin{equation}\label{eq:xqseqtildelwbd}
\begin{linsys}{8}
x_{-1} &=& 0 &= &x_{0} &< &  x_{1} &<& \dots&< &x_{r - 2}  &< &x_{r - 1} &=& 1 \\
&& \bm 0 &= &\tilde\bQ_0 & < & \tilde\bQ_1  &<& \dots&< &\tilde\bQ_{r - 2} &< &\tilde\bQ_{r-1} &<& \tilde\bQ_r &=& \bQ  
\end{linsys}
\end{equation}
where $\tilde \bQ_p = \bQ_p + \epsilon \be_p$ for $1 \leq p \leq r$. We first note that $(\tilde \bQ_k)_{k = 1}^{r} \in \cQ_r$. By definition, 
\begin{equation}\label{eq:newd}
\tilde \bd_p := \sum_{p \leq k \leq r-1} x_k ( \tilde\bQ_{k + 1} - \tilde\bQ_{k}  ) = \sum_{p \leq k \leq r-1} x_k ( \bQ_{k + 1} - \bQ_{k}  ) + \epsilon \sum_{p \leq k \leq r-1} x_k ( \be_{k + 1} - \be_{k}  )  = \bd_p(\epsilon).
\end{equation}
Since $|\bQ_k - \bQ_{k - 1}| > 0$ implies $|\bxi'(\bQ_k) - \bxi'(\bQ_{k - 1}) | > 0$ [\ref{app:properties}, Proposition~\ref{prop:gapxi}], the critical point condition \eqref{eq:critptlwbd} implies the path $(\tilde \bQ_k)_{k = 1}^{r}$ has positive definite increments for $1 \leq k \leq r-1$,
\[
x_k(\tilde{\bQ}_{k + 1} - \tilde \bQ_k ) = \tilde\bd_{k} - \tilde\bd_{k+1} = \bL_k^{-1} - \bL_{k + 1}^{-1} > 0 .
\]
The boundary conditions are also satisfied since $\be_r = 0$ implies that $\tilde \bQ_{r} = \bQ_{r} + \be_r = \bQ$ and the critical point condition for $\bQ_1$ \eqref{eq:lwbd1cond} implies that
\[
\tilde \bQ_1 = \bQ_1 + \epsilon \be_1 = \bL_1^{-1} ( \vh \vh^\trans + \bxi'(\bQ_1) ) \bL_1^{-1} > 0.
\]

Using convexity, we will prove that the perturbed functional $\sC_r^\epsilon$ can be lower bounded by $\sC_r$ evaluated at the path encoded by \eqref{eq:xqseqtildelwbd},
\begin{equation}\label{eq:lwbdlwbd}
\sC^\epsilon_r((\bQ_k)_{k = 1}^r ) \geq \sC_r( \vex, (\tilde \bQ_k)_{k = 1}^{r})
\end{equation}
provided that $ (\bQ_k)_{k = 1}^r$ satisfies the critical point conditions \eqref{eq:critptlwbd}. Since the sequences of matrices $(\tilde \bQ_k)_{k = 1}^{r}$ is in $\cQ_r$, we get the obvious lower bound,
\[
\sP^\epsilon_r (\bL, (\bQ_k)_{k = 1}^r) = \sC^\epsilon_r(\vex, (\bQ_k)_{k = 1}^r) \geq \sC_r(\vex, (\tilde \bQ_k)_{k = 1}^{r} ) \geq  \inf_{r,x,Q} \sC_r(\vex,\vbQ).
\]
The lower bound does not depend on the discretization $r$, $\epsilon$, nor the fixed sequence \eqref{eq:xseqlwbd}. Therefore, we can minimize the upper bound over sequences \eqref{eq:xseqlwbd}, $r$ and $\epsilon$ to prove the required lower bound,
\[
\inf_{r, \Lambda, x,Q} \sP_r(\bL,\vex,\vbQ) \geq \inf_{r,x,Q} \sC_r(\vex,\vbQ) .
\]

We now prove the lower bound \eqref{eq:lwbdlwbd}. 
\begin{lem}\label{lem:lwbdconcave}
	For all $\epsilon > 0$, if $\vbQ$ satisfies the critical point conditions \eqref{eq:critptlwbd}, then 
	\[
	\sC_r^\epsilon((\bQ_k)_{k = 1}^r) \geq \sC_r( \vex,  (\tilde \bQ_k)_{k = 1}^{r} ). 
	\]
\end{lem}

\begin{proof}
	Since $\tilde \bd_p =  \bd_p(\epsilon)$ and the barrier $\sE_r \geq 0$, it remains to show that
	\begin{align}
	&\langle \bQ_1 , \bd_1^{-1}(\epsilon) \rangle + \sum_{1 \leq k \leq r-1} x_k \cdot \Sum \big( \bxi (\bQ_{k + 1}) - \bxi (\bQ_{k})\big) -\epsilon \langle \bd_{r - 1}^{-1}(\epsilon) , \bE_{r-1} \rangle + \epsilon \langle \bxi'(\bQ_{r - 1}), \bE_{r-1} \rangle \notag\\
	&\quad - \epsilon \sum_{1 \leq k \leq r-2} \langle \bE_{k+1} - \bE_k, \bxi' (\bQ_{k + 1}) \rangle - \sum_{1 \leq k \leq r-2} \frac{\epsilon}{x_{k}} \langle \bd_{k+1}^{-1}(\epsilon), \bE_{k} -  \bE_{k+1} \rangle \label{eq:reduclwbderror}
	\end{align}
	is bounded below by
	\begin{align*}
	\quad& \langle \tilde\bQ_1, \tilde\bd_1^{-1} \rangle + \sum_{1 \leq k \leq r-1} x_k \cdot \Sum \big( \bxi (\tilde \bQ_{k + 1}) - \bxi (\tilde \bQ_{k})\big) 
	\\&= \langle \tilde\bQ_1, \tilde\bd_1^{-1} \rangle  +  \sum_{1 \leq k \leq r-1} (x_{k - 1} - x_{k})  \Sum ( \bxi (\tilde\bQ_{k})) + x_{r - 1} \Sum( \bxi(\tilde\bQ_r) ).
	\end{align*}
	
	We will use convexity of the $\bxi$ terms to absorb the $\epsilon$ error terms in \eqref{eq:reduclwbderror}. The definition of $\bE_k$ in \eqref{eq:lwbderror} implies that
	\begin{equation}\label{eq:lwbd4}
	\bE_k - \bE_{k+1} = x_k (\be_{k + 1} - \be_{k}) \qquad 1 \leq k \leq r- 1. 
	\end{equation}
	Using summation by parts and \eqref{eq:lwbd4} the last four $\epsilon$ terms in \eqref{eq:reduclwbderror} can be simplified to
	\begin{align*}
	&\quad - \epsilon \sum_{1 \leq k \leq r-2} \langle \bE_{k+1} - \bE_k , \bxi' (\bQ_{k + 1}) \rangle - \epsilon \sum_{1 \leq k \leq r-2} \frac{1}{x_{k}}  \langle \bd_{k+1}^{-1}(\epsilon) ,  \bE_{k} -  \bE_{k+1} \rangle \\
	&\quad -\epsilon \langle \bd_{r - 1}^{-1}(\epsilon), \bE_{r-1} \rangle + \epsilon \langle \bxi'(\bQ_{r - 1}), \bE_{r-1} \rangle  
	\\&\stackrel{\mathclap{\eqref{eq:lwbd3} \eqref{eq:lwbd4} }}{=} \epsilon \sum_{1 \leq k \leq r-2} \langle x_k ( \be_{k+1} - \be_k ) ,  \bxi' (\bQ_{k + 1}) \rangle - \epsilon \sum_{1 \leq k \leq r-2} \langle \bL_{k+1} , \be_{k+1} -  \be_{k} \rangle\\
	&\quad -\epsilon \langle  \bL_{r - 1}, \bE_{r-1} \rangle + \epsilon \langle \bxi'(\bQ_{r - 1}), \bE_{r-1} \rangle
	\\&= \epsilon \sum_{1 \leq k \leq r-2} \langle  x_k ( \be_{k+1} - \be_k ), \bxi' (\bQ_{k + 1}) \rangle  +\epsilon \sum_{1 \leq k \leq r-2} \langle \bL_{k+1} - \bL_{k} , \be_{k} \rangle \\
	&\quad- \epsilon \langle \bL_{r- 1} , \be_{r-1} \rangle + \epsilon \langle \bL_1 , \be_1 \rangle -\epsilon \langle \bL_{r - 1} , \bE_{r-1} \rangle + \epsilon \langle  \bxi'(\bQ_{r - 1}), \bE_{r-1} \rangle 
	\\&\stackrel{\mathclap{\eqref{eq:lwbd1}}}{=} \epsilon \sum_{1 \leq k \leq r-2}\langle x_k ( \be_{k+1} - \be_k ) , \bxi' (\bQ_{k + 1}) \rangle +\epsilon \sum_{1 \leq k \leq r-2} \langle x_k(\bxi'(\bQ_{k + 1}) -  \bxi'(\bQ_{k}) ) , \be_{k} \rangle\\
	&\quad- \epsilon \langle \bL_{r- 1} , \be_{r-1} \rangle + \epsilon \langle \bL_1, \be_1 \rangle  -\epsilon \langle \bL_{r - 1} , \bE_{r-1} \rangle + \epsilon \langle  \bxi'(\bQ_{r - 1}) , \bE_{r-1} \rangle
	\\&= \epsilon \sum_{1 \leq k \leq r-1} x_k \big(  \langle  \bxi' (\bQ_{k + 1}) , \be_{k + 1} \rangle - \langle \bxi' (\bQ_{k}),  \be_k \rangle \big) + \epsilon \langle \bL_1, \be_1 \rangle
	\end{align*}
	since $\be_r = 0$ and $\bE_{r - 1} = \be_r - \be_{r - 1} = -\be_{r- 1}$. Therefore, excluding the leftover $\langle \bQ_1, \bd_1^{-1}(\epsilon) \rangle + \epsilon \langle \bL_1, \be_1 \rangle$ term, \eqref{eq:reduclwbderror} is equal to
	\begin{align}
	&\quad\sum_{1 \leq k \leq r-1} x_k \cdot \Sum \big( \bxi (\bQ_{k + 1}) - \bxi (\bQ_{k})\big) + \epsilon \sum_{1 \leq k \leq r-1} x_k \cdot \big( \langle \bxi' (\bQ_{k + 1}) , \be_{k + 1} \rangle - \langle \bxi' (\bQ_{k}) , \be_k  \rangle \big) \notag
	\\&= \sum_{1 \leq k \leq r-1} (x_{k - 1} - x_{k})  \Sum ( \bxi (\bQ_{k}) + \epsilon \bxi' (\bQ_{k}) \odot  \be_{k} ) + x_{r - 1} \Sum( \bxi(\bQ_r + \epsilon \be_r) ). \label{eq:bdconvex}
	\end{align}
	
	Since $\bxi(\bA)$ is convex [\ref{app:properties}, Proposition~\ref{prop:xiconvex}] and $(x_{k-1} - x_{k}) \leq 0$, we also have
	\begin{equation}\label{eq:cvx1}
	(x_{k-1} - x_{k})\Sum( \bxi (\tilde\bQ_{k}) ) = (x_{k-1} - x_{k})\Sum( \bxi (\bQ_{k} + \epsilon \be_k) ) \leq (x_{k-1} - x_{k})\Sum( \bxi (\bQ_{k} ) + \epsilon \bxi'(\bQ_{k}) \odot \be_k ).
	\end{equation}
	Furthermore, the leftover terms satisfy
	\begin{equation}\label{eq:cvx2}
	\langle \bd_1^{-1}(\epsilon) , \bQ_1 \rangle + \epsilon \langle \bL_1 , \be_1 \rangle = \langle \tilde\bd_1^{-1}, \bQ_1 \rangle +  \langle \tilde\bd_1^{-1} , \epsilon \be_1 \rangle = \langle \tilde\bd_1^{-1} , \tilde\bQ_1 \rangle.
	\end{equation}
	Applying \eqref{eq:cvx1} and \eqref{eq:cvx2} to \eqref{eq:bdconvex} and the left over terms implies that \eqref{eq:reduclwbderror} is bounded below by
	\[
	\langle \tilde\bQ_1, \tilde\bd_1^{-1} \rangle  +  \sum_{1 \leq k \leq r-1} (x_{k - 1} - x_{k})  \Sum ( \bxi (\tilde\bQ_{k})) + x_{r - 1} \Sum( \bxi(\tilde\bQ_r) ),
	\]
	which is what we needed to show.
\end{proof}

\subsection{Summary of the Proof} We now summarize the proof of the lower bound.

\begin{proof}[Proof of Lemma~\ref{lem:lwbd}]
	Assuming that $\vb_2 > 0$, for $\epsilon > 0$ and fixed sequence \eqref{eq:xseqlwbd}, the minimizer $\bL^\epsilon$, $\vbQ^\epsilon$ of  $\sP_r^\epsilon$ satisfies the critical point conditions \eqref{eq:critptlwbd} by Lemma~\ref{lem:lwbdcritpt}. From Lemma~\ref{lem:lwbdreduction} and Lemma~\ref{lem:lwbdconcave}, these critical point conditions results in the following chain of inequalities,
	\[
	\inf_{\Lambda,Q} \sP_r^\epsilon(\bL,\vbQ) = \sP^\epsilon_r( \bL^\epsilon,\vbQ^\epsilon) = \sC_r^\epsilon(\vbQ^\epsilon) \geq \inf_{r,x,Q} \sC_r(\vex,\vbQ).
	\]
	Since $\sP^\epsilon_r( \bL,\vbQ)$ is decreasing in $\epsilon$ for fixed $\bL$ and $\vbQ$ and $\sP_r( \bL,\vbQ)$ is continuous, we can interchange the limit with the infimum [\ref{app:properties}, Proposition~\ref{prop:interchangeinf}], so
	\[
	\lim_{\epsilon \to 0} \inf_{\Lambda,Q}  \sP_r^\epsilon( \bL,\vbQ) = \inf_{\Lambda,Q} \lim_{\epsilon \to 0} \sP_r^\epsilon(\bL,\vbQ)  =  \inf_{\Lambda,Q} \sP_r(\bL,\vbQ) \geq \inf_{r,x,Q} \sC_r(\vex,\vbQ).
	\]
	The lower bound does not depends on $r$ nor the sequence \eqref{eq:xseqlwbd}, so we can take the infimum of $\sP_r$ over all sequences of the form \eqref{eq:xseqlwbd} and all discretizations to finish the proof of the lower bound.
	
	This proves the case of the lower bound under the additional assumption that $\beta_2 > 0$. To conclude the general case, suppose that $(\vb_p)_{p \geq 2}$ is a sequence of positive inverse temperature parameters such that $\beta_p = 0$ if $p$ is odd. We can modify the temperature by adding a small positive perturbation to the second term, $(\vb_p^\delta)_{p \geq 2} = (\vb_2 + \delta \vec{1}, \vb_4, \dots)$. Consider $\sP_r^\delta$ and $\sC_r^{\delta}$ defined with respect to $(\vb_p^\delta)_{p \geq 2}$. We have
	\[
	\inf_{r,\Lambda,x,Q} \sP_r^\delta(\bL,\vex,\vbQ) \geq \inf_{r,x,Q} \sC_r^\delta(\vex,\vbQ).
	\] 
	This holds for all $\delta > 0$, so we can use the fact that both $\inf\sP^\delta_r$ and $\inf\sC^\delta_r$ are uniformly continuous functions of the temperature [\ref{app:properties}, Proposition~\ref{prop:unifcontC}] and send $\delta \to 0$ to conclude
	\[
	\inf_{r,\Lambda,x,Q} \sP_r(\bL,\vex,\vbQ) \geq \inf_{r,x,Q} \sC_r(\vex,\vbQ).
	\] 
\end{proof}

\begin{rem}
	The exact formula for the error terms $\be_k$ was not needed in our computations. We are free to choose any barrier $\sE_r$ that assigns infinitely large penalties to degenerate increments to prove the lower bound. The logarithmic barrier was chosen because its derivatives are easy to compute explicitly. 
\end{rem}

\section{The Upper Bound of the Parisi Functional}\label{sec:upbd}

We now use a similar procedure to prove the matching upper bound. To simplify notation, several terms such as $\sP_r^\epsilon$,  $\sC_r^\epsilon$, and $\bE$ that appeared Section~\ref{sec:lwbd} will be redefined in this section. In this section, we will prove that the infimum of the Parisi Functional is a lower bound of the Crisanti--Sommers functional:
\begin{lem}\label{lem:upbd} For any positive definite constraint $\bQ$, we have
	\[
	\inf_{r,\Lambda,x,Q} \sP_r (\bL,\vex, \vbQ) \leq \inf_{r,x,Q} \sC_r(\vex, \vbQ),
	\]
	where the first infimum is over sequences \eqref{eq:xqseq} and $\bL \in \bS_+^n$ such such that $|\bL_1| > 0$ and the second infimum is over sequences \eqref{eq:xqseq} such that $|\bd_{r - 1}| > 0$.
\end{lem}

Like the lower bound, we prove this by examining the behavior of $\sC_r$ at its critical points. We will perturb $\sC_r$ by adding a logarithmic penalty at the boundary to force the minimizer of $\sC_r$ to have positive increments. The minimizers will satisfy an interior critical point condition that will allow us to reduce the perturbed functional $\sC_r^\epsilon$  into a perturbed $\sP_r^\epsilon$ functional.  These perturbed functionals will converge to $\sP_r$ and $\sC_r$ in the limit as the size of the barrier tends to $0$. The main difference is the convexity argument used in the proof of Lemma~\ref{lem:lwbdconcave} does not work in this direction. Instead, we use a concavity argument to absorb the error terms into the Lagrange multiplier term. 

\subsection{Adding a Positive Definite Barrier}We fix $r \geq 2$ and let $\vbQ = (\bQ_k)_{k = 0}^r$ denote the monotone sequence of matrices. We will add a logarithmic barrier to $\sC_r$ that introduces a large penalty when $\vbQ$ is not strictly increasing. Let $\epsilon > 0$ and consider the barrier term
\[
\sE( \vbQ ) := - \epsilon \sum_{0 \leq k \leq r - 1} \log |\bQ_{k + 1} - \bQ_k|.
\]
Since $| \bA | \leq (\frac{\tr(A)}{n})^n$ [\ref{app:properties}, Proposition~\ref{prop:AMGM}] we have $|\bQ_{k + 1} - \bQ_k| \leq 1$ so $\sE_r \geq 0$. Furthermore, $\sE \to + \infty$ if $|\bQ_{k + 1} - \bQ_k| \to 0$ for some $0 \leq k \leq r-1$. 

For a fixed strictly increasing path such that
\begin{equation}\label{eq:xsequpbd}
0 = x_0 < x_1 < \dots < x_{r-1} = 1,
\end{equation}
we define the functional,
\begin{align}
\sC_r^\epsilon(\vbQ) &= \frac{1}{2}\Big[ \log| \bQ - \bQ_{r - 1}| + \langle\vh \vh^\trans, \bd_1  \rangle - \sum_{1 \leq k \leq r-2} \frac{1}{x_k} \log \frac{|\bd_{k + 1}|}{|\bd_k|} + \langle \bd_1^{-1}, \bQ_1 \rangle \notag \\
&\quad  + \sum_{1 \leq k \leq r-1} x_k \cdot \Sum \big( \bxi (\bQ_{k + 1}) - \bxi (\bQ_{k})\big)  - \epsilon \sum_{0 \leq k \leq r - 1}  \log | \bQ_{k + 1} - \bQ_k | \Big] \label{eq:pertcrisanti}.
\end{align}
Notice that $\sC^\epsilon_r = \sC_r + \epsilon \sE_r$ decreases pointwise to $\sC(\vex, \vbQ)$ as $\epsilon \to 0$, where $\vex$ is the monotone sequence \eqref{eq:xsequpbd}. The barrier term forces the minimizer to lie in the interior of the positive definite cone, since $\sC_r^\epsilon( \vbQ ) \to +\infty$ if one of the increments $|\bQ_{k + 1} - \bQ_k| \to 0$. We now examine the behavior of $\sC_r^\epsilon(\vbQ)$ at its minimizers and recover a system of critical point equations. 

\subsection{Critical Point Conditions} We will study the first variation of $\sC_r^\epsilon$ to recover critical point conditions for its minimizer. We want to minimize the function
\[
\sC_r^\epsilon(\vbQ) := \sC_r^\epsilon ( \bQ_1, \dots, \bQ_{r - 1} )
\]
over the parameters
\[
(\bQ_k)_{k = 1}^{r-1} \in \cQ_r := \big\{ \bQ_1, \dots, \bQ_{r-1} \in \bS_+^n \mmm |\bQ_{k + 1} - \bQ_k | > 0, ~\forall~ 0 \leq k \leq r - 1 \big\}.
\]
By compactness, $\sC_r^\epsilon$ attains its minimum at some $(\bQ_k)_{k = 1}^{r-1} \in \cQ_r$. Since $\sE_r(\vbQ) = \infty$ if the increments are not positive definite, $(\bQ_k)_{k = 1}^{r-1}$ must have positive definite increments,
\[
|\bQ_{k + 1} - \bQ_k | > 0 \quad \forall~ 0 \leq k \leq r - 1.
\]
This implies that symmetric matrices are admissible variations of $(\bQ_k)_{k = 1}^{r-1}$ [\ref{app:properties}, Proposition~\ref{prop:pertposdef}]. In particular, if $\bC$ is a symmetric matrix, then for all $t$ sufficiently small,
\[
(\bQ_1, \dots, \bQ_p + t \bC, \dots \bQ_{r - 1}) \in  \cQ_r \text{ for $1 \leq p \leq r- 1$}.
\]
If $(\bQ_k)_{k = 1}^{r-1} \in \cQ_r$ is a minimizer of $\sC_r^\epsilon$, then for all $1 \leq p \leq r- 1$, 
\[
\frac{d}{dt} \sC_r^\epsilon( \bL, \bQ_1, \dots, \bQ_p + t \bC, \dots \bQ_{r - 1} ) \Big|_{t = 0} = 0.
\]
We can compute the first variation of the functionals explicitly by computing the matrix derivatives of $\sC_r^\epsilon$ and derive some critical point conditions on the minimizers [Appendix~\ref{app:derivativesC}]:

	\noindent(a)~~~ For $2 \leq p \leq r - 1$, the directional derivatives of $\sC_r^\epsilon$ with respect to $\bQ_p$ in the symmetric direction $2\bC$ is
	\begin{align}
	\partial_{\bQ_p} \sC_r^\epsilon &= (x_{p - 1} - x_p) \langle \vh \vh^T , \bC  \rangle - (x_{p - 1} - x_p) \langle \bd_1^{-1} \bQ_1 \bd_1^{-1} , \bC \rangle \notag
	\\&\quad - (x_{p - 1} - x_p) \sum_{1 \leq k \leq p-1} \frac{1}{x_k} \langle \bd^{-1}_{k+1} - \bd^{-1}_{k }, \bC \rangle  + (x_{p - 1} - x_p) \langle \bxi'(\bQ_p) , \bC \rangle \notag
	\\&\quad + \epsilon \langle( \bQ_{p+ 1 } - \bQ_{p} )^{-1}, \bC \rangle - \epsilon \langle ( \bQ_{p} - \bQ_{p - 1} )^{-1}, \bC \rangle.\label{eq:derivCrisantiQ_p}
	\end{align}
	At the minimizer, we require
	\[
	\partial_{\bQ_p} \sC_r^\epsilon :=  \frac{d}{dt} \sC_r^\epsilon(\bQ_1, \dots, \bQ_p + 2t \bC, \dots \bQ_{r - 1} ) \Big|_{t = 0} = 0.
	\]
	This equality holds for all symmetric directions $\bC$, so the minimizer satisfies the critical point equation
	\begin{align}
	\bxi'(\bQ_p) &= -\vh \vh^T + \bd_1^{-1} \bQ_1 \bd_1^{-1} + \sum_{1 \leq k \leq p-1} \frac{1}{x_k} (\bd^{-1}_{k+1} - \bd^{-1}_{k }) + \epsilon \be_p, \label{eq:csqp}
	\end{align}
	where
	\begin{equation}\label{eq:upbdeterm}
	\be_p := \frac{1}{x_p- x_{p-1} } \Big( ( \bQ_{p+ 1 } - \bQ_{p} )^{-1} - ( \bQ_{p} - \bQ_{p - 1} )^{-1} \Big).
	\end{equation}
	$\be_p$ is well defined because we fixed a strictly increasing sequence $(x_p)_{p = 1}^r$ in \eqref{eq:xsequpbd}. 
	
	\noindent (b)~~~ For $p = 1$, the directional derivatives of $\sC_r^\epsilon$ with respect to $\bQ_1$ in the symmetric direction $2\bC$ is
	\begin{align}
	\partial_{\bQ_1}\sC_r^\epsilon &= - x_1 \langle \vh \vh^T, \bC \rangle + x_1 \langle \bd_1^{-1} \bQ_1 \bd_1^{-1}, \bC \rangle - x_1 \langle \bxi'(\bQ_1) , \bC \rangle \notag
	\\&\quad + \epsilon \langle ( \bQ_{2} - \bQ_{1} )^{-1} , \bC \rangle - \epsilon \langle \bQ_{1}^{-1} , \bC \rangle. 
	\end{align}
	At the minimizer, we require
	\[
	\partial_{\bQ_1} \sC_r^\epsilon :=  \frac{d}{dt} \sC_r^\epsilon( \bQ_1 + 2t\bC, \dots, \bQ_{r - 1} ) \Big|_{t = 0} = 0.
	\]
	This equality holds for all symmetric directions $\bC$, so the minimizer satisfies the critical point equation
	\begin{equation}\label{eq:csqp2}
	\bxi'(\bQ_{1} ) = -\vh \vh^T + \bd_1^{-1} \bQ_1 \bd_1^{-1} + \epsilon \be_1,
	\end{equation}
	where $\be_1$ is given by the formula in \eqref{eq:upbdeterm} with $p = 1$.

For $1 \leq p \leq r-1$, the critical point equations \eqref{eq:csqp} and \eqref{eq:csqp2} can be expressed as
\[
\bxi'(\bQ_{p} ) = -\vh \vh^T + \bd_1^{-1} \bQ_1 \bd_1^{-1} + \sum_{1 \leq k \leq p-1} \frac{1}{x_k} (\bd^{-1}_{k+1} - \bd^{-1}_{k}) + \epsilon \be_p
\]
where $\be_r := 0$ and $\be_p$ was defined in \eqref{eq:upbdeterm}.
By subtracting these equations, we can conclude for $1 \leq p \leq r - 2$,
\begin{equation}\label{eq:critptupbd}
x_p (\bxi'(\bQ_{p+1} ) - \bxi'(\bQ_{p} )) = \bd^{-1}_{p+1} - \bd^{-1}_{p} + \epsilon x_p ( \be_{p+1} - \be_{p} ).
\end{equation}
Consider $\bL$ given by $\bL := (\bQ - \bQ_{r - 1})^{-1} + \bxi'(\bQ) - \bxi'(\bQ_{r - 1}) - \epsilon (\be_r -  \be_{r - 1})$. For this choice of $\bL$, we have
\[
\bL_{r-1}(\epsilon) =  \bL - (\bxi'(\bQ) - \bxi'(\bQ_{r - 1})) + \epsilon (\be_r -  \be_{r - 1}) = \bd^{-1}_{r-1}.
\]
Subtracting \eqref{eq:critptupbd} from $\bL_{r-1}(\epsilon)$, we conclude that
\begin{equation}
\bL_p + \sum_{p \leq k \leq r - 1} x_k (\be_{k + 1} - \be_k) = \bd_p^{-1} \text{ for $1 \leq p \leq r-1$}.
\end{equation}
The critical point conditions implicitly implies that $\bL_1(\epsilon) > 0$ and $\bd_k^{-1}= \bL_{k}(\epsilon) < \bL_{k + 1}(\epsilon) = \bd_{k + 1}^{-1}$. We summarize the critical point condition in the following lemma.
\begin{lem}\label{lem:upbdcritpt} For fixed $r \geq 2$, if $\vbQ$ is a minimizer of $\sC_r^\epsilon$ and
	\begin{equation}\label{eq:critptpert1}
	\bL := (\bQ - \bQ_{r - 1})^{-1} + \bxi'(\bQ) - \bxi'(\bQ_{r - 1}) - \epsilon (\be_r -  \be_{r - 1})	
	\end{equation}
	then $\vbQ$ satisfies the following critical point equations
	\begin{equation}\label{eq:critptpert2}
	\bd_{p}^{-1} = \bL_p(\epsilon) \qquad \text{ for } 1 \leq p \leq r- 1
	\end{equation}
	where
	\begin{equation}\label{eq:upbderror}
	\bL_p(\epsilon) = \bL_p + \epsilon \bE_p \qquad \text{ and } \qquad \bE_p = \sum_{p \leq k \leq r-1} x_k (\be_{k + 1} - \be_k).
	\end{equation}
\end{lem}

\subsection{Reduction to an approximate Parisi functional}

In this subsection, we will reduce $\sC_r^\epsilon(\vbQ)$ defined in \eqref{eq:pertcrisanti} to an approximate Parisi functional. If $\bL$ equals \eqref{eq:critptpert1} and $\vbQ$ satisfies the critical point conditions \eqref{eq:critptpert2}, then $\sC_r^\epsilon$ can be reduced to
\begin{align}
\sP_r^\epsilon(\bL,\vbQ) &= \frac{1}{2}\Big[ \langle \bL , \bQ \rangle  - n - \log |\bL| +  \sum_{1 \leq k \leq r-1} \frac{1}{x_k} \log \frac{| \bL_{k + 1} (\epsilon)|}{|\bL_k(\epsilon) |} \notag \\
&\quad + \langle\bL_1^{-1}(\epsilon), \vh\vh^\trans + \bxi'(\bQ_1) \rangle   - \sum_{1 \leq k \leq r-1} x_k \cdot \Sum \big( \btheta (\bQ_{k + 1}) - \btheta (\bQ_{k})\big) \notag\\
&\quad - \epsilon \sum_{1 \leq k \leq r-2} \langle \bE_{k+1} - \bE_k ,  \bQ_{k} \rangle  + \sum_{1 \leq k \leq r-2} \frac{\epsilon}{x_{k}} \langle  \bL_{k}^{-1}(\epsilon) ,\bE_{k} -  \bE_{k+1}  \rangle \notag\\
&\quad +\epsilon \langle \bL_{r - 1}^{-1}(\epsilon) , \bE_{r-1} \rangle + \epsilon \langle \bQ_{r - 1} , \bE_{r-1} \rangle - \epsilon \sum_{0 \leq k \leq r - 1}  \log | \bQ_{k + 1} - \bQ_k | \Big] \label{eq:approxparisi}.
\end{align}
Notice that $\sP_r^\epsilon(\bL,\vbQ)$ is of the same form as $\sP_r(\bL,\vex,\vbQ)$, but with $\bL_k$ replaced by $\bL_k(\epsilon)$ and some additional error terms. If we set $\epsilon = 0$, then the error terms in the second line all vanish and we are left with the usual $\sP_r(\bL,\vex,\vbQ)$ functional. In the next subsection, we will show that we can bound the minimum of $\sP_r^\epsilon$ with $\sP_r$ evaluated at a different parameter to remove the error terms. 

\begin{lem}\label{lem:upbdreduction}
	For fixed $r$, if $\bL$ and $ (\bQ_k)_{k = 1}^r$ satisfy the critical point conditions \eqref{eq:critptpert1} and \eqref{eq:critptpert2}, then
	\[
	\sC_r^\epsilon(\vbQ) = \sP_r^\epsilon(\bL,\vbQ).
	\]
\end{lem}

\begin{proof} The proof is a straightforward but tedious computation. The computation is almost identical to the proof of Lemma~\ref{lem:lwbdreduction}. Assuming that \eqref{eq:critptpert1} and \eqref{eq:critptpert2} hold, we will show that 
	\[
	2 (\sP^\epsilon_r(\bL,\vbQ) - \sC_r^\epsilon( \vbQ ) ) = 0.
	\]
	We will use the following identities multiple times throughout the proof,
	\begin{alignat}{2}
	\bL_{k + 1} - \bL_{k}  &= x_k (\bxi'(\bQ_{k + 1}) - \bxi'(\bQ_{k})) &&\qquad 1 \leq k \leq r - 1 \label{eq:upbd1}\\
	\bd_{k} - \bd_{k + 1}  &= x_k (\bQ_{k + 1} - \bQ_{k}) &&\qquad 1 \leq k \leq r - 1\label{eq:upbd2}\\
	\bd_{k}^{-1} &= \bL_k(\epsilon)  &&\qquad 1 \leq k \leq r - 1\label{eq:upbd3}
	\end{alignat}
	where $\bd_{r} := \bm 0$. These identities will allow us to simplify $\sC_r^\epsilon$ into $\sP_r^\epsilon$. 
	
	We start by observing that the external fields cancel if \eqref{eq:critptpert2} holds,
	\begin{equation}\label{eq:upbdexternal}
	\langle\vh\vh^\trans , \bL_1^{-1}(\epsilon) \rangle \stackrel{\mathclap{\eqref{eq:upbd3}}}{=} \langle \vh\vh^\trans , \bd_1(\epsilon)  \rangle .
	\end{equation}
	Next, we simplify the summation fo the logarithm terms in $\sP^\epsilon_r$. Equation \eqref{eq:critptpert2} applied the $r - 1$ term implies that the boundary term in the first summation of \eqref{eq:approxparisi} simplifies to
	\[
	\frac{1}{x_{r - 1}} (\log| \bL_r (\epsilon)| - \log| \bL_{r - 1} (\epsilon)|) \stackrel{\mathclap{\eqref{eq:upbd3}}}{=} \log| \bL| - \log|\bd^{-1}_{r - 1}| = \log | \bL| + \frac{1}{x_{r - 1}}\log| \bQ - \bQ_{r-1}|,
	\]
	since $\bd_{r - 1} = \bQ - \bQ_{r - 1}$ and $x_{r - 1} = 1$. Applying \eqref{eq:critptpert2} again to $\bL_{k + 1} (\epsilon)$ for $1 \leq k \leq r-1$ implies
	\[
	- \log|\bL| + \sum_{1 \leq k \leq r-1} \frac{1}{x_k} \log \frac{|\bL_{k + 1} (\epsilon)|}{|\bL_k(\epsilon)|} \stackrel{\mathclap{\eqref{eq:upbd3}}}{=} \frac{1}{x_{r - 1}} \log| \bQ - \bQ_{r-1}| - \sum_{1 \leq k \leq r-2} \frac{1}{x_k} \log \frac{|\bd_{k + 1}|}{|\bd_k|},
	\]
	Therefore, the log determinant terms in $\sP_r^\epsilon$ and $\sC_r^\epsilon$ also cancel.
	
	Since $\btheta(\bA) = \bA \odot \bxi'(\bA) - \bxi(\bA)$ and $\Sum(\bA \odot \bB) =  \langle\bA , \bB \rangle$, the remaining terms in $2 (\sP^\epsilon_r - \sC_r^\epsilon )$ are
	\begin{align}
	&- \sum_{1 \leq k \leq r-2} x_k \Big( \langle \bQ_{k + 1} , \bxi' (\bQ_{k + 1}) \rangle -\langle \bQ_{k} , \bxi' (\bQ_{k}) \rangle  \Big) - \Big( \langle \bQ , \bxi'(\bQ)  \rangle - \langle \bQ_{r-1} , \bxi'(\bQ_{r-1}) \rangle \Big) \label{eq:reducCtoP1}
	\\&\quad - \epsilon \sum_{1 \leq k \leq r-2} \langle \bE_{k+1} - \bE_k  , \bQ_{k} \rangle   + \sum_{1 \leq k \leq r-2} \frac{\epsilon}{x_{k}} \tr \langle \bL_{k}^{-1}(\epsilon) , \bE_{k} -  \bE_{k+1} \rangle \label{eq:reducCtoP2}
	\\&\quad + \langle \bL , \bQ \rangle - n  + \langle \bxi'(\bQ_1) , \bL_1^{-1}(\epsilon) \rangle + \epsilon \langle \bL_{r - 1}^{-1}(\epsilon) , \bE_{r-1} \rangle + \epsilon \langle \bQ_{r - 1}, \bE_{r-1} \rangle - \langle \bQ_1 , \bd_1^{-1}\rangle  \label{eq:reducCtoP3}.
	\end{align}
	We will show that \eqref{eq:reducCtoP1} will cancel \eqref{eq:reducCtoP2} and \eqref{eq:reducCtoP3} at the critical point. 
	
	We start by simplifying the first summation term in \eqref{eq:reducCtoP1} using \eqref{eq:critptpert2},
	\begin{align}
	&\quad -\sum_{1 \leq k \leq r-2} x_k \Big( \langle \bQ_{k + 1} , \bxi' (\bQ_{k + 1}) \rangle - \langle \bQ_{k} , \bxi' (\bQ_{k}) \rangle  \Big) \label{eq:sumfirstr}
	\\&= -\sum_{1 \leq k \leq r-2} \Big( \langle x_k(\bQ_{k + 1} - \bQ_{k}) ,  \bxi' (\bQ_{k + 1}) \rangle + \langle \bQ_{k} , x_k (\bxi' (\bQ_{k + 1}) - \bxi' (\bQ_{k}) \rangle \Big) \notag
	\\&\stackrel{\mathclap{\eqref{eq:upbd1}\eqref{eq:upbd2}}}{=} -\sum_{1 \leq k \leq r-2} \Big( \langle \bd_{k} - \bd_{k + 1}  ,  \bxi' (\bQ_{k + 1}) \rangle + \langle \bQ_{k} , \bL_{k + 1}(\epsilon) - \bL_{k}(\epsilon) \rangle \Big) + \epsilon \sum_{1 \leq k \leq r-2} \langle \bE_{k+1} - \bE_k ,  \bQ_{k} \rangle .\label{eq:upbdsum}
	\end{align}
	Using summation by parts and \eqref{eq:critptpert2}, the first summation \eqref{eq:upbdsum} is equal to
	\begin{align}
	&-\sum_{1 \leq k \leq r-2} \Big( \langle \bL_{k}^{-1}(\epsilon) , \bxi' (\bQ_{k + 1}) - \bxi' (\bQ_{k}) \rangle + \langle \bL_{k + 1}(\epsilon) , \bQ_{k} - \bQ_{k + 1} \rangle \Big) \label{eq:summation2}
	\\&- \langle \bL^{-1}_1(\epsilon), \bxi'(\bQ_1) \rangle + \langle \bL^{-1}_{r - 1}(\epsilon), \bxi'(\bQ_{r - 1}) \rangle - \langle \bL_{r - 1}(\epsilon) , \bQ_{r- 1} \rangle + \langle \bL_{1}(\epsilon) , \bQ_1\rangle \notag.
	\end{align}
	From the critical point condition \eqref{eq:critptpert2}, we have
	\begin{align*}
	\bQ_{k} - \bQ_{k+1} &\stackrel{\mathclap{\eqref{eq:upbd2}}}{=} \frac{1}{x_{k}}( \bd_{k + 1} -  \bd_{k}) 
	\\&\stackrel{\mathclap{\eqref{eq:upbd3}}}{=} \frac{1}{x_{k}} \bL_{k}^{-1}(\epsilon)( \bL_{k}(\epsilon) -  \bL_{k+1}(\epsilon))\bL_{k+1}^{-1}(\epsilon) 
	\\&\stackrel{\mathclap{\eqref{eq:upbd1}}}{=} \bL_{k}^{-1}(\epsilon)(\bxi'(\bQ_{k})  - \bxi'(\bQ_{k+1}))\bL_{k+1}^{-1}(\epsilon) + \frac{\epsilon}{x_{k}} \bL_{k}^{-1}(\epsilon)( \bE_{k} -  \bE_{k+1})\bL_{k+1}^{-1} (\epsilon),
	\end{align*}
	which combined with the fact $\tr(\bA \bB \bC) = \tr(\bC \bA \bB)$ and equation \eqref{eq:summation2} implies that
	\begin{align}
	\eqref{eq:sumfirstr} =& \epsilon \sum_{1 \leq k \leq r-2} \langle  \bE_{k+1} - \bE_k , \bQ_{k} \rangle - \sum_{1 \leq k \leq r-2} \frac{\epsilon}{x_{k}} \langle \bL_{k}^{-1}(\epsilon) , \bE_{k} -  \bE_{k+1} \rangle \notag
	\\&- \langle\bL^{-1}_1(\epsilon) ,\bxi'(\bQ_1)  \rangle  + \langle \bL^{-1}_{r - 1}(\epsilon), \bxi'(\bQ_{r - 1}) \rangle - \langle \bL_{r - 1}(\epsilon), \bQ_{r- 1} \rangle + \langle \bL_{1}(\epsilon) , \bQ_1 \rangle \label{eq:upbdintermediate}.
	\end{align}
	Substituting \eqref{eq:upbdintermediate} into \eqref{eq:reducCtoP1} implies that
	\begin{align}
	2 (\sP^\epsilon_r - \sC_r^\epsilon ) &=  \langle \bL ,\bQ\rangle - n    + \langle \bxi'(\bQ_1) , \bL_1^{-1}(\epsilon) \rangle + \epsilon \langle \bL_{r - 1}^{-1}(\epsilon) , \bE_{r-1} \rangle + \epsilon \langle \bQ_{r - 1} , \bE_{r-1} \rangle - \langle \bQ_1 , \bd_1^{-1} \rangle \notag
	\\&\quad   - \langle \bL^{-1}_1(\epsilon) , \bxi'(\bQ_1) \rangle + \langle \bL^{-1}_{r - 1}(\epsilon), \bxi'(\bQ_{r - 1}) \rangle - \langle \bL_{r - 1}(\epsilon) , \bQ_{r- 1} \rangle + \langle \bL_{1}(\epsilon) , \bQ_1 \rangle \notag
	\\&\quad -  \langle \bQ, \bxi' (\bQ) \rangle + \langle \bQ_{r-1} , \bxi' (\bQ_{r-1}) \rangle  \notag\\
	&\stackrel{\mathclap{\eqref{eq:upbd3}}}{=}   \langle \bL , \bQ\rangle - n  +\epsilon \langle \bL_{r - 1}^{-1}(\epsilon) , \bE_{r-1} \rangle + \epsilon \langle \bQ_{r - 1} , \bE_{r-1} \rangle \notag
	\\&\quad + \langle \bL^{-1}_{r - 1}(\epsilon), \bxi'(\bQ_{r - 1}) \rangle - \langle \bL_{r - 1}(\epsilon) , \bQ_{r- 1} \rangle -  \langle \bQ , \bxi' (\bQ) \rangle + \langle \bQ_{r-1} , \bxi' (\bQ_{r-1}) \rangle \notag
	\\&= \langle \bL, \bQ\rangle - \tr (\bI)   +\epsilon \langle \bL_{r - 1}^{-1}(\epsilon), \bE_{r-1} \rangle + \langle \bL^{-1}_{r - 1}(\epsilon), \bxi'(\bQ_{r - 1}) \rangle \notag
	\\&\quad- \langle \bL, \bQ_{r-1} \rangle + \langle \bxi'(\bQ), \bQ_{r-1}\rangle -  \langle \bQ , \bxi' (\bQ) \rangle   \label{eq:upbdremaining}.
	\end{align}
	since $\bL_{r - 1}(\epsilon) = \bL - (\bxi'(\bQ) - \bxi'(\bQ_{r - 1})) + \epsilon \bE_{r-1}$. We will show that the $\langle \bL^{-1}_{r - 1}(\epsilon), \bxi'(\bQ_{r- 1})\rangle$ term will cancel all remaining terms. Using the definition of $\bL_{r - 1}(\epsilon)$ defined in \eqref{def:L} and \eqref{eq:upbderror},
	\begin{align*}
	\bL^{-1}_{r - 1}(\epsilon) \bxi'(\bQ_{r- 1}) &= \bL_{r - 1}^{-1}(\epsilon) ( \bL_{r - 1}(\epsilon)- \bL + \bxi'(\bQ) - \epsilon \bE_{r - 1} )
	\\&= \bI - \bL^{-1}_{r - 1}(\epsilon) \bL + \bL^{-1}_{r - 1}(\epsilon) \bxi'(\bQ) - \epsilon \bL^{-1}_{r - 1}(\epsilon)  \bE_{r-1}.
	\end{align*}
	From \eqref{eq:critptpert2} and the fact $\bL_{r-1}^{-1}(\epsilon) = \bd_{r - 1} = \bQ - \bQ_{r - 1}$, we have
	\[
	- \bL^{-1}_{r - 1}(\epsilon) \bL + \bL^{-1}_{r - 1}(\epsilon) \bxi'(\bQ)  = -(\bQ - \bQ_{r-1}) \bL + (\bQ - \bQ_{r-1}) \bxi'(\bQ).
	\]
	Since $\tr(\bA \bB) = \tr(\bB \bA)$, taking the trace implies
	\begin{equation}\label{eq:tracecs}
	\langle \bL^{-1}_{r - 1}(\epsilon) , \bxi'(\bQ_{r- 1}) \rangle = \tr(\bI) - \langle \bL, \bQ\rangle+ \langle \bL, \bQ_{r-1}\rangle + \langle \bQ, \bxi'(\bQ) \rangle - \langle \bxi'(\bQ), \bQ_{r-1} \rangle - \epsilon \langle \bL^{-1}_{r - 1}(\epsilon), \bE_{r-1} \rangle.
	\end{equation}
	Substituting \eqref{eq:tracecs} into \eqref{eq:upbdremaining} cancels out all remaining terms, so
	\[
	2 \big( \sC_r^\epsilon(\vbQ) - \sP_r^\epsilon(\bL,\vbQ) \big) = 0.
	\]
\end{proof}

\subsection{Removing the Error Terms} Like the case of the upper bound, we will use concavity of the terms of $\sP^\epsilon_r(\bL,\vbQ)$ defined in \eqref{eq:approxparisi} to bound the minimizer with $\sP_r$ evaluated at a different Lagrange multiplier parameter. To this end, we define
\[
\tilde\bL = \bL + \epsilon \bE_1
\]
and for $1 \leq p \leq r-1$,
\[
\tilde \bL_p = \tilde \bL - \sum_{p \leq k \leq r-1} x_k ( \bxi'(\bQ_{k + 1}) - \bxi'(\bQ_k) ).
\]
We first note that $\tilde \bL \in \sL$. By definition, we have
\begin{equation}\label{eq:newlag}
\tilde \bL_1 := \tilde \bL - \sum_{1 \leq k \leq r- 1} x_k (\bxi'(\bQ_{k + 1}) - \bxi'(\bQ_k) ) = \bL - \sum_{1 \leq k \leq r- 1} x_k (\bxi'(\bQ_{k + 1}) - \bxi'(\bQ_k) ) + \epsilon \bE_1  =  \bL_1(\epsilon),
\end{equation}
is positive definite at the critical point because \eqref{eq:critptpert1} implies $\bL_1(\epsilon) = \bd_1^{-1} > 0$. By monotonicity, this implies that $\tilde \bL_p > 0$ for all $1 \leq p \leq r - 1$.  

We will use concavity of the log determinants to prove that the original Parisi functional evaluated at $\tilde \bL$ is a lower bound of $\sP_r^\epsilon(\bL,\vbQ)$,
\begin{equation}\label{eq:upbdlowbd}
\sP_r^\epsilon(\bL,\vbQ) \geq \sP_r(\tilde \bL,\vbQ)
\end{equation}
provided that $ (\bQ_k)_{k = 1}^r$ satisfies the critical point conditions \eqref{eq:critptpert2}. Since both $\tilde \bL$ and the path $\vbQ$ are elements in the sets we minimize over, we get the obvious lower bound,
\[
\sC_r^\epsilon(\vbQ) = \sP_r^\epsilon(\bL,\vbQ) \geq \sP_r(\tilde \bL, \vbQ) \geq  \inf_{r,\Lambda,x,Q} \sP_r(\bL,\vex,\vbQ).
\]
The lower bound does not depend on the discretization $r$, $\epsilon$, nor the fixed sequence \eqref{eq:xsequpbd}. In particular, we can minimize $\sC_r^\epsilon$ over sequences \eqref{eq:xsequpbd}, $r$ and $\epsilon$ to prove the required upper bound,
\[
\inf_{r, x,Q} \sC_r(\vex,\vbQ) \geq \inf_{r, \Lambda,x,Q} \sP_r(\bL, \vex, \vbQ) .
\]

We now prove the lower bound \eqref{eq:upbdlowbd}. 

\begin{lem}\label{lem:upbdconcave}
	If $\vex$ is equal to \eqref{eq:xsequpbd}, $\vbQ$ satisfies the critical point conditions \eqref{eq:critptpert2} and $\bL$ equals \eqref{eq:critptpert1}, then
	\[
	\sP_r^\epsilon(\bL,\vbQ) \geq \sP_r(\tilde \bL, \vex, \vbQ). 
	\]
\end{lem}

\begin{proof}
	Since $(\bQ_k)_{k = 1}^r$ is unchanged and $\sE_r \geq 0$, it remains to show that
	\begin{align}
	&\quad \langle \bL, \bQ\rangle - n - \log |\bL| +  \sum_{1 \leq k \leq r-1} \frac{1}{x_k} \log \frac{| \bL_{k + 1} (\epsilon)|}{| \bL_k (\epsilon)|}  + \langle \vh \vh^\trans + \bxi'(\bQ_1),  \bL_1^{-1} (\epsilon) \rangle \notag
	\\&\quad - \epsilon \sum_{1 \leq k \leq r-2} \langle \bE_{k+1} - \bE_k ,  \bQ_{k} \rangle + \sum_{1 \leq k \leq r-2} \frac{\epsilon}{x_{k}} \langle \bL_{k}^{-1}(\epsilon) , \bE_{k} -  \bE_{k+1} \rangle \notag\\
	&\quad +\epsilon \langle \bL_{r - 1}^{-1}(\epsilon) , \bE_{r-1} \rangle + \epsilon \langle \bQ_{r - 1} , \bE_{r-1} \rangle \notag\\
	&=  \langle \bL , \bQ \rangle - n +  \sum_{2 \leq k \leq r-1} \Big( \frac{1}{x_{k - 1}} - \frac{1}{x_{k}} \Big) \log|\bL_{k}(\epsilon)| - \frac{1}{x_1} \log| \bL_1(\epsilon)| + \langle \vh \vh^\trans + \bxi'(\bQ_1) ,  \bL_1^{-1}(\epsilon) \rangle  \notag 
	\\&\quad - \epsilon \sum_{1 \leq k \leq r-2} \langle \bE_{k+1} - \bE_k ,  \bQ_{k} \rangle + \sum_{1 \leq k \leq r-2} \frac{\epsilon}{x_{k}} \langle \bL_{k}^{-1}(\epsilon) ,  \bE_{k} -  \bE_{k+1} \rangle \notag\\
	&\quad +\epsilon \langle \bL_{r - 1}^{-1}(\epsilon) , \bE_{r-1} \rangle + \epsilon \langle \bQ_{r - 1} , \bE_{r-1} \rangle \label{eq:upbdtopterms}
	\end{align}
	is bounded below by
	\begin{align*}
	&\quad \langle \tilde\bL , \bQ \rangle - n - \log |\tilde\bL| +  \sum_{1 \leq k \leq r-1} \frac{1}{x_k} \log \frac{| \tilde\bL_{k + 1} |}{| \tilde\bL_k|} + \langle \vh \vh^\trans, \tilde\bL_1^{-1} \rangle + \langle \bxi'(\bQ_1) , \tilde\bL_1^{-1} \rangle \\
	&= \langle \tilde \bL , \bQ \rangle - n +  \sum_{2 \leq k \leq r-1} \Big( \frac{1}{x_{k - 1}} - \frac{1}{x_{k}} \Big) \log|\tilde\bL_{k}| - \frac{1}{x_1} \log| \tilde\bL_1| + \langle \vh \vh^\trans + \bxi'(\bQ_1) , \tilde\bL_1^{-1} \rangle. \notag 
	\end{align*}
	We will use concavity of the log determinant terms to absorb the error terms in \eqref{eq:upbdtopterms}. We use summation by parts to write the error terms in \eqref{eq:upbdtopterms} as
	\begin{align*}
	&\quad \epsilon \langle \bQ_{r - 1} , \bE_{r-1} \rangle - \epsilon \sum_{1 \leq k \leq r-2} \langle \bE_{k+1} - \bE_k ,  \bQ_{k} \rangle + \sum_{1 \leq k \leq r-2} \frac{\epsilon}{x_{k}} \langle \bL_{k}^{-1}(\epsilon) , \bE_{k} -  \bE_{k+1} \rangle + \epsilon \langle \bL_{r - 1}^{-1}(\epsilon), \bE_{r - 1} \rangle\\
	&= \epsilon \sum_{1 \leq k \leq r-2} \langle \bQ_{k+1} - \bQ_k , \bE_{k + 1} \rangle + \epsilon \langle \bQ_{1} , \bE_{1} \rangle + \sum_{1 \leq k \leq r-2} \frac{\epsilon}{x_{k}} \tr \langle \bL_{k}^{-1}(\epsilon) , \bE_{k} -  \bE_{k+1} \rangle + \epsilon \langle \bL_{r - 1}^{-1}(\epsilon), \bE_{r - 1} \rangle\\
	&\stackrel{\mathclap{\eqref{eq:upbd2}}}{=} \sum_{1 \leq k \leq r-2} \frac{\epsilon }{x_k}  \langle \bd_k - \bd_{k + 1} , \bE_{k + 1} \rangle + \epsilon \langle \bQ_{1} , \bE_{1} \rangle + \sum_{1 \leq k \leq r-2} \frac{\epsilon}{x_{k}} \langle \bL_{k}^{-1}(\epsilon) ,  \bE_{k} -  \bE_{k+1} \rangle + \epsilon \langle \bL_{r - 1}^{-1}(\epsilon), \bE_{r - 1} \rangle\\
	&\stackrel{\mathclap{\eqref{eq:upbd3}}}{=} \sum_{1 \leq k \leq r-2} \frac{\epsilon }{x_k} \langle \bL^{-1}_k(\epsilon) - \bL^{-1}_{k + 1}(\epsilon) , \bE_{k + 1} \rangle + \epsilon \langle \bQ_{1} , \bE_{1} \rangle + \sum_{1 \leq k \leq r-2} \frac{\epsilon}{x_{k}} \langle \bL_{k}^{-1}(\epsilon) , \bE_{k} -  \bE_{k+1} \rangle + \epsilon \langle \bL_{r - 1}^{-1}(\epsilon), \bE_{r - 1} \rangle\\
	&= - \epsilon \sum_{2 \leq k \leq r-1} \Big( \frac{1}{x_{k - 1}} - \frac{1}{x_{k}} \Big) \langle \bL_{k}^{-1}(\epsilon) , \bE_k \rangle + \frac{\epsilon}{x_1} \langle \bL_1^{-1}(\epsilon) , \bE_1 \rangle + \epsilon \langle \bQ_1 , \bE_1 \rangle.
	\end{align*}
	Adding and subtracting $\sum_{2 \leq k \leq r-1} \Big( \frac{1}{x_{k - 1}} - \frac{1}{x_{k}} \Big) \epsilon \langle \bL_k^{-1}(\epsilon), \bE_1 \rangle$ and using the fact that
	\begin{align*}
	&\quad - \sum_{2 \leq k \leq r-1} \Big( \frac{1}{x_{k - 1}} - \frac{1}{x_{k}} \Big) \epsilon \langle \bL_k^{-1}(\epsilon) , \bE_1 \rangle + \frac{\epsilon}{x_1}  \langle \bL_1^{-1}(\epsilon) , \bE_1 \rangle + \epsilon \langle \bQ_1 , \bE_1 \rangle
	\\&= \epsilon \langle \bL^{-1}_{r - 1} (\epsilon) , \bE_1 \rangle + \sum_{1 \leq k \leq r-2} \frac{\epsilon}{x_{k}} \langle \bL_{k}^{-1}(\epsilon) - \bL_{k + 1}^{-1}(\epsilon) , \bE_1 \rangle + \epsilon \langle \bQ_1 , \bE_1 \rangle
	\\&\stackrel{\mathclap{\eqref{eq:upbd3}}}{=}  \epsilon \langle \bd_{r - 1},  \bE_1 \rangle + \sum_{1 \leq k \leq r-2} \frac{\epsilon}{x_{k}} \langle \bd_{k} - \bd_{k + 1} , \bE_1 \rangle + \epsilon \langle \bQ_1 , \bE_1 \rangle\\
	&\stackrel{\mathclap{\eqref{eq:upbd2}}}{=} \epsilon \langle \bQ - \bQ_{r - 1} , \bE_1 \rangle + \sum_{1 \leq k \leq r-2} \epsilon \langle \bQ_{k + 1} - \bQ_{k} , \bE_1 \rangle + \epsilon \langle \bQ_1 , \bE_1 \rangle \\
	&= \epsilon \langle \bQ , \bE_1 \rangle 
	\end{align*}
	if the critical point condition \eqref{eq:critptpert2} holds, we see that \eqref{eq:upbdtopterms} is equal to
	\begin{align}
	& \langle \bL , \bQ \rangle - n +  \sum_{2 \leq k \leq r-1} \Big( \frac{1}{x_{k - 1}} - \frac{1}{x_{k}} \Big) \log|\bL_{k}(\epsilon)| - \frac{1}{x_1} \log| \bL_1(\epsilon)| + \langle \vh \vh^\trans + \bxi'(\bQ_1) , \bL_1^{-1}(\epsilon) \rangle \notag 
	\\&\quad + \sum_{2 \leq k \leq r-1} \Big( \frac{1}{x_{k - 1}} - \frac{1}{x_{k}} \Big) \epsilon \langle \bL_{k}^{-1}(\epsilon) , -\bE_k + \bE_1 \rangle +  \epsilon \langle \bQ , \bE_1 \rangle \label{eq:lowbdperterror}.
	\end{align}
	
	We can now use concavity of the log determinant terms and the first trace term to absorb the error terms. Since $x_{k - 1} < x_k$, the concavity of the log determinant [\ref{app:properties}, Proposition~\ref{prop:logconcave}] implies 
	\begin{equation}\label{eq:concavelog}
	\Big( \frac{1}{x_{k - 1}} - \frac{1}{x_{k}} \Big) \log|\bL_{k}(\epsilon)| + \Big( \frac{1}{x_{k - 1}} - \frac{1}{x_{k}} \Big) \epsilon \langle \bL_k^{-1}, - \bE_k + \bE_1 \rangle \geq \Big( \frac{1}{x_{k - 1}} - \frac{1}{x_{k}} \Big) \log|\tilde\bL_{k}|
	\end{equation}
	because $\bL_{k}(\epsilon) - \epsilon \bE_k + \epsilon \bE_1 = \tilde \bL_k$. The linearity of the trace implies
	\begin{equation}\label{eq:concavetr}
	\langle \bL, \bQ \rangle + \epsilon \langle \bQ , \bE_1 \rangle = \langle \bL + \epsilon \bE_1 , \bQ \rangle =  \langle \tilde \bL , \bQ \rangle.
	\end{equation}
	The inequalities \eqref{eq:concavelog} and \eqref{eq:concavetr} and the fact $\bL_1(\epsilon) = \tilde \bL_1$ shown in \eqref{eq:newlag} implies that \eqref{eq:upbdtopterms} is bounded below by
	\begin{align*}
	\langle \tilde \bL , \bQ \rangle - n +  \sum_{2 \leq k \leq r-1} \Big( \frac{1}{x_{k - 1}} - \frac{1}{x_{k}} \Big) \log|\tilde\bL_{k}| - \frac{1}{x_1} \log| \tilde\bL_1| + \langle \vh \vh^\trans + \bxi'(\bQ_1) , \tilde\bL_1^{-1} \rangle, \notag 
	\end{align*}
	which is what we needed to show.
\end{proof}

\subsection{Summary of the Proof} We now summarize the proof of the upper bound.

\begin{proof}[Proof of Lemma~\ref{lem:upbd}]
	For $\epsilon > 0$ and fixed sequence \eqref{eq:xsequpbd}, if we define $\bL^\epsilon$ to be equal to \eqref{eq:critptpert1}, then the minimizer $\vbQ^\epsilon$ of $\sC_r^\epsilon(\vbQ)$ satisfies the critical point conditions \eqref{eq:critptpert2} by Lemma~\ref{lem:upbdcritpt}. From Lemma~\ref{lem:upbdreduction} and Lemma~\ref{lem:upbdconcave}, these critical point conditions implies the following chain of inequalities,
	\[
	\inf_{Q} \sC_r^\epsilon(\vbQ) = \sC_r^\epsilon(\vbQ^\epsilon) = \sP_r^\epsilon(\bL^\epsilon, \vbQ^\epsilon) \geq \inf_{r,\Lambda,x,Q} \sP_r(\bL, \vex,\vbQ).
	\]
	Since $\sC^\epsilon_r(\vbQ)$ is decreasing in $\epsilon$ for fixed $\vbQ$ and $\sC_r(\vbQ)$ is continuous, we can interchange the limit with the infimum [\ref{app:properties}, Proposition~\ref{prop:interchangeinf}], so
	\[
	\lim_{\epsilon \to 0} \inf_{Q}  \sC_r^\epsilon(\vbQ) = \inf_{Q} \lim_{\epsilon \to 0} \sC_r^\epsilon(\vbQ)  =  \inf_{Q} \sC_r(\vbQ) \geq \inf_{r,\Lambda,x,Q} \sP_r(\bL,\vex,\vbQ).
	\]
	The lower bound does not depends on $r$ nor the sequence \eqref{eq:xsequpbd}, so we can take the infimum of $\sC_r$ over all sequences of the form \eqref{eq:xseqlwbd} and all discretizations to finish the proof of the upper bound. 
\end{proof}

\section{Integral Form of the Crisanti--Sommers functional}\label{sec:intform}

We will derive the integral form for the analogue of the Crisanti--Sommers formula for spherical spin glasses with vector spins. Recall the monotone functions \eqref{eq:csx} and \eqref{eq:csPhi},
\[
x(t) : [0,n] \to [0,1] \quad \text{such that} \quad x(0) = 0 \quad \text{and} \quad x(n) = 1
\]
and
\[
\Phi(t): [0,n] \to \bS^n_+ \quad \text{such that} \quad \tr( \Phi(t)) = t \quad  \text{and} \quad \Phi(0) = \bm 0 \quad \text{and} \quad \Phi(n) = \bQ.
\] 

For $t_x := x^{-1}(1) = \inf \{ t \in [0,n] \mmm 1 \leq x(t) \}$ and paths such that $|\bQ - \Phi(t_x)| > 0$ the analogue of the Crisanti--Sommers functional \eqref{eq:cs} was defined by
\begin{align}
\sC(x,\Phi) &= \frac{1}{2}\bigg( \int_0^n x(t) \langle \bxi'(\Phi(t)) + \vh \vh^\trans, \Phi'(t) \rangle \, dt + \log| \Phi(n) - \Phi(t_x)| + \int_0^{t_x} \langle \hat \Phi(t)^{-1},  \Phi'(t) \rangle \, dt \bigg) \label{eq:intformCS}
\\&= \frac{1}{2} \bigg( \langle \vh \vh^\trans, \hat \Phi(0) \rangle + \int_0^n x(t) \langle \bxi'(\Phi(t)), \Phi'(t) \rangle \, dt + \log| \Phi(n) - \Phi(t_x)| + \int_0^{t_x} \langle \hat \Phi(t)^{-1},  \Phi'(t) \rangle \, dt \bigg) \notag ,
\end{align}
where $\hat \Phi(t): [0,n] \to \R^{n \times n}$ is a decreasing function given by
\begin{equation}
\hat \Phi(t) = \int_t^n x(s) \Phi'(s) \, ds.
\end{equation}

This functional is the continuous Lipschitz extension of the discrete functional \eqref{eq:crisanti} we proved in the last section. We first observe that $\sC(x,\Phi)$ agrees with the discrete formula when $x(t)$ corresponds to a discrete probability measure on the trace.
\begin{lem}\label{lem:discreization}
	Let $\Phi(t)$ be a fixed monotone matrix path. Let $x(t)$ be a step function with $r - 1$ steps,
	\[
	x(t) = x_k \quad \text{for} \quad t_{k} \leq t < t_{k + 1}
	\]
	for $1 \leq k \leq r-1$ with boundary terms
	\[
	\quad x(t) = 0 \quad \text{for} \quad 0 \leq t < t_{1} \quad \text{and}\quad x(t) = 1 \quad \text{for} \quad t_x := t_{r - 1}  \leq t \leq 1  .
	\]
	If we define $\bQ_k := \Phi(t_k)$, then
	\[
	\sC(x,\Phi) = \sC_r(\vex,\vbQ).
	\]
\end{lem}

\begin{proof} We first observe for $t_p \leq t < t_{p + 1}$ that,
	\begin{align}
	\hat \Phi(t) = \int^n_{t} x(t) \Phi'(t) \, dt  &= \sum_{k = p + 1}^{r-2} x_k \int_{t_{k}}^{t_{k + 1}} \Phi'(t) \, dt +  x_p \int_{t}^{t_{p + 1}} \Phi'(t) \, dt \notag
	\\&= \sum_{k = p + 1}^{r-2} x_k ( \Phi(t_{k + 1}) -  \Phi(t_{k}) ) + x_p ( \Phi(t_{p + 1}) -  \Phi(t) )  \notag
	\\&= \bd_{p + 1} + x_p ( \Phi(t_{p + 1}) -  \Phi(t) ) \label{eq:Dint}.
	\end{align}
	We now compute each of the terms in $\sC(x,\Phi)$ when $x(t)$ is piecewise constant. 
	
		\noindent (a)~~~ The identity \eqref{eq:Dint} implies $\hat \Phi(0) = \bd_1$ since $x_0 = 0$, so
		\begin{equation*}\label{eq:conttodiscstart}
		\langle \vh \vh^\trans, \hat \Phi(0) \rangle = \langle \vh \vh^\trans, \bd_1 \rangle.
		\end{equation*}
		\noindent (b)~~~ Since $x(t) = x_k$ for $t_{k} \leq t < t_{k + 1}$, [\ref{app:properties}, Proposition~\ref{prop:derivtheta}]  implies the second term in \eqref{eq:intformCS} simplifies to
		\begin{alignat*}{2}
		\int_0^n x(t) \langle \bxi'(\Phi(t)), \Phi'(t) \rangle \, dt = \sum_{k = 0}^{r - 1} x_k \int_{t_k}^{t_{k + 1}} \langle \bxi'(\Phi(t)), \Phi'(t) \rangle \, dt = \sum_{1 \leq k \leq r-1} x_k\cdot \Sum  \big( \bxi (\bQ_{k + 1}) - \bxi (\bQ_{k})\big). 
		\end{alignat*}
		\noindent (c)~~~  Since $\Phi(t_x) = \bQ_{r - 1}$ by definition,
		\begin{equation*}
		\log| \Phi(n) - \Phi(t_x)| = \log| \bQ - \bQ_{r - 1}|.
		\end{equation*}
		\noindent (d)~~~ For almost every $t_p < t < t_{p + 1}$ the identity \eqref{eq:Dint} and Proposition~\ref{prop:derivlogdet} implies
		\[
		\frac{d}{dt} \bigg( -\frac{1}{x_p} \log | \hat \Phi(t) | \bigg) = \frac{d}{dt} \bigg( -\frac{1}{x_p} \log | \bd_{p + 1} + x_p ( \Phi(t_{p + 1}) -  \Phi(t) )  | \bigg) = \langle \hat \Phi(t)^{-1},  \Phi'(t) \rangle
		\]
		so the fundamental theorem of calculus implies that
		\begin{equation*}
		\int_{t_1}^{t_x} \langle \hat \Phi(t)^{-1},  \Phi'(t) \rangle \, dt = \sum_{k = 1}^{r - 2} \int_{t_k}^{t_{k + 1}} \langle \hat \Phi(t)^{-1},  \Phi'(t) \rangle \, dt = - \sum_{1 \leq k \leq r-2} \frac{1}{x_k} \log \frac{|\bd_{k + 1}|}{|\bd_k|}
		\end{equation*}
		and since $x(t) = 0$ for $0 \leq t < t_1$, $\hat\Phi(t) = \bd_1$ the boundary term is
		\begin{equation*}\label{eq:conttodiscend}
		\int_{0}^{t_1} \langle \hat \Phi(t)^{-1},  \Phi'(t) \rangle \, dt = \langle  \bd_1^{-1}, \Phi(t_1) \rangle -  \langle  \bd_1^{-1}, \Phi(0) \rangle = \langle  \bd_1^{-1}, \bQ_1 \rangle.
		\end{equation*}
		\noindent
	Substituting the formulas derived in (a) to (d) into $\sC(x,\Phi)$ finishes the proof.
\end{proof}
Lemma~\ref{lem:discreization} implies that $\sC(x,\Phi)$ evaluated at a piecewise constant c.d.f. corresponds to $\sC_r(\vex, \vbQ)$ evaluated at some sequence of the form \eqref{eq:xqseqcs}. To see that every discrete path encoded by $(\vex, \vbQ)$ corresponds to some $(x,\Phi)$, consider the sequences
\begin{equation}\label{eq:xqseqcs}
\begin{linsys}{8}
x_{-1} &=& 0 &\leq &x_{0} &\leq &  x_{1} &\leq& \dots&\leq &x_{r - 2}  &\leq &x_{r - 1} &=& 1 \\
&& \bm 0 &= &\bQ_0 & \leq & \bQ_1  &\leq& \dots&\leq &\bQ_{r - 2} &\leq &\bQ_{r-1} &<& \bQ_r &=& \bQ
\end{linsys}~.
\end{equation}
Taking $t_k := \tr(\bQ_k)$ we define a Lipschitz path $\Phi$ by taking $\Phi(t_k) = \bQ_k$ at each point $t_k$ and interpolate linearly,
\[
\Phi(t_k) = \bQ_k, \qquad \Phi(t) =  \frac{t_{k + 1} - t}{ t_{k + 1} - t_k}\Phi(t_k)  + \frac{t - t_k}{ t_{k + 1} - t_k} \Phi(t_{k + 1}) \qquad \text{ for } t_{k} \leq t < t_{k + 1},
\]
and a piecewise constant c.d.f. $x(t) = x_k$ for $t_{k} \leq x_k < t_{k + 1}$. Applying Lemma~\ref{lem:discreization} implies that $\sC_r(\vex, \vbQ)$ evaluated at any sequence of the form \eqref{eq:xqseqcs} corresponds to $\sC(x,\Phi)$ for some $(x,\Phi)$. This implies that
\begin{equation}\label{eq:upbdintform}
\inf_{r,x,Q} \sC_r(\vex,\vbQ) \geq \inf_{x,\Phi} \sC(x,\Phi),
\end{equation}
since infimum on the right is over all c.d.f.s and not necessarily piecewise constant ones. 

The opposite inequality is a bit trickier to show. We first show that $\sC(x,\Phi)$ is locally Lipschitz, which will imply that the integral form of the functional is the Lipschitz extension of the functional evaluated on discrete paths. The functional is not well defined when $|\bQ - \Phi(t_x)| = 0$ because of the log determinant term, so we will show that the functional is Lipschitz if we restrict the domain to matrix paths such that $|\bQ - \Phi(t_x)|$ is uniformly bounded away from $0$.

Let $T \in [0,1)$ and $L > 0$. Consider the compact set
\[
A_{T,L} = \{ (x,\Phi) \mmm x(t) = 1 \text{ for } t \geq T, \| (\bQ - \Phi(t_x) )^{-1} \|_\infty \leq L \}.
\]
This set is closed because any convergent sequence $(x_n, \Phi_n)$ must satisfy the uniform bounds,
\[
x_n(t) = 1 \text{ for } t \geq T \quad \text{and} \quad  \| (\bQ - \Phi_n (x_n^{-1}(1)) )^{-1} \|_\infty \leq L,
\]
for all $n$, so its limit point must as well. Furthermore, the product of the space of c.d.f.s on $[0,n]$ equipped with the $\|\cdot\|_1$ norm and the space of Lipschitz paths with fixed endpoints equipped with the $\| \cdot \|_\infty$ norm is compact by Prokhorov's theorem and the Arzel\`a--Ascoli theorem. Since $A_{T,L}$ is a closed subset of a compact set it is compact. We will show that $\sC(x,\Phi)$ is Lipschitz on the compact set $A_{T,L}$. 

\begin{lem}\label{lem:lip}
	Let $(x_1, \Phi_1), (x_2, \Phi_2) \in A_{T,L}$. There exists a constant $C_L$ that only depends on the fixed parameters of the model and the uniform bound $L$ on $\| (\bQ - \Phi_n (x_n^{-1}(1)) )^{-1} \|_\infty$ such that
	\[
	| \sC( x_1, \Phi_1) - \sC( x_2, \Phi_2) | \leq C_L (\| x_1 - x_2\|_1 + \| \Phi_1 - \Phi_2 \|_\infty),
	\]
	where
	\[
	\| x_1 - x_2\|_1 = \int_{0}^{n} | x_1(t) - x_2(t) | \, dt \quad \text{and} \quad  \| \Phi_1 - \Phi_2 \|_\infty = \max_{i,j \leq n} \biggr( \sup_{t \in [0,n]} | \Phi^{i,j}_1(t) - \Phi^{i,j}_2(t)| \biggr).
	\]
\end{lem}

\begin{proof}
	Without loss of generality, suppose that $x_1^{-1}(1) \leq x_2^{-1}(1)$. If this is the case, then $\sC(x_1, \Phi_2)$ is also well defined since $\bQ - \Phi_2( x_1^{-1}(1) ) \geq \bQ - \Phi_2( x_2^{-1}(1) )$ by monotonicity, so $(x_1, \Phi_2) \in A_{T,L}$. Therefore,
	\[
	| \sC( x_1, \Phi_1) - \sC( x_2, \Phi_2) | \leq | \sC( x_1, \Phi_1) - \sC( x_1, \Phi_2) | + | \sC( x_1, \Phi_2) - \sC( x_2, \Phi_2) |.
	\]
	Therefore, it suffices to show that the functional is Lipschitz in each of its coordinates,
	\[
	| \sC( x, \Phi_1) - \sC( x, \Phi_2) | \leq C_L \| \Phi_1 - \Phi_2 \|_\infty \quad \text{and} \quad | \sC( x_1, \Phi) - \sC( x_2, \Phi) | \leq C_L \| x_1 - x_2 \|_1.
	\]
	We start by showing the first inequality. The computation to show the functional is Lipschitz in $x$ for fixed $\Phi$ follows the similar computations. 
	\\\\
	\textit{Lipschitz in $\Phi$:} Fix $x(t)$ and consider $( x, \Phi_1), (x,\Phi_2) \in A_{T,L}$. We first show that the functional is Lipschitz with respect to the infinity norm on matrix paths,
	\begin{equation}\label{eq:lipphi}
	| \sC( x, \Phi_1) - \sC( x, \Phi_2) | \leq C_L \| \Phi_1 - \Phi_2 \|_\infty.
	\end{equation}
	We will show that each term in $\sC(x,\Phi)$ is Lipschitz in $\Phi$ for fixed $x$. 
	
		\noindent (a)~~~ The integrand consists of functions of bounded variation, so we can integrate by parts to conclude
		\begin{align*}
		\langle \vh \vh^\trans, \hat \Phi(0)  \rangle &= \int_0^n x(t) \cdot \frac{d}{dt} \langle \vh \vh^\trans, \Phi(t) \rangle \, dt
		\\&=   \langle \vh \vh^\trans, \bQ \rangle - \int_0^n \langle \vh \vh^\trans, \Phi(t) \rangle dx(t).
		\end{align*}
		Since $x(0) = 0$ and $x(n) = 1$, we have
		\begin{align*}
		|\langle \vh \vh^\trans, \hat \Phi_1(0) \rangle - \langle \vh \vh^\trans, \hat \Phi_2(0) \rangle| &\leq \int_0^n \Big|\langle \vh \vh^\trans, \Phi_1(t) - \Phi_2(t) \rangle  \Big| \, dx(t) 
		\\&\leq n^2 \| \vh \vh^\trans \|_\infty \| \Phi_1 - \Phi_2 \|_\infty.
		\end{align*}
		\noindent (b)~~~ The second term can be bounded in a similar manner using integration by parts,	
		\begin{align*}
		\int_0^n x(t) \langle \bxi'(\Phi(t)), \Phi'(t) \rangle \, dt &= \int_0^n x(t) \cdot \frac{d}{dt} \Sum (\bxi(\Phi(t))) \, dt 
		\\&=  \Sum (\bxi(\bQ))  - \int_0^n \Sum (\bxi(\Phi(t))) \, dx(t).
		\end{align*}
		Since $\bxi(t)$ is a power series and $\Phi$ is bounded, we can conclude
		\begin{align*}
		\bigg| \int_0^n x(t) \langle \bxi'(\Phi_1(t)), \Phi_1'(t) \rangle \, dt - \int_0^n x(t) \langle \bxi'(\Phi_2(t)), \Phi_2'(t) \rangle \, dt \bigg| &\leq n^2 \|\bxi'(1)\|_\infty \| \Phi_1 - \Phi_2 \|_\infty.
		\end{align*}
		\noindent(c)~~~ The condition $\| (\bQ - \Phi(t_x))^{-1} \|_\infty \leq L$ and equivalence of the infinity norm and operator norm on $\R^{n \times n}$ implies that 
		\[
		\lambda_{min} ( \bQ - \Phi(t_x) ) = \frac{1}{\lambda_{max} ( ( \bQ - \Phi(t_x) )^{-1} )} \geq \frac{1}{\sqrt{n} \| (\bQ - \Phi(t_x))^{-1} \|_\infty}  \geq \frac{1}{\sqrt{n} L}.
		\]
		The determinant is the product of eigenvalues so $| \bQ - \Phi(t_x) | \geq  (\sqrt{n} L )^{-n} > 0$. Furthermore, $\log (t)$ is Lipschitz on $[ (\sqrt{n} L )^{-n},\infty)$ and $| \bA |$ is a polynomial of the entires of $\bA$, so there exists universal constants $C_1,C_2$ that depends only on $L$ and the dimension $n$ such that
		\[
		| \log| \bQ - \Phi_1(t_x) | - \log| \bQ - \Phi_2(t_x) | | \leq C_1\big| | \bQ - \Phi_1(t_x) | - | \bQ - \Phi_2(t_x) | \big| \leq C_2 \| \Phi_1 - \Phi_2 \|_\infty.
		\]
		\noindent (d)~~~ To show the last term is Lipschitz, we will show that all of its unit directional derivatives are uniformly bounded and apply the mean value theorem to conclude Lipschitz continuity. Let $\Phi,\Psi$ be arbitrary matrices such that $(x,\Phi), (x,\Psi) \in A_{T,L}$. By monotonicity $(1 -\eps)\Phi +  \eps\Psi$ is also a Lipschitz monotone path and since the matrix inverse is convex [\ref{app:properties}, Proposition~\ref{prop:inverseconvex}], 
		\[
		\bigr( \bQ - ( \eps\Phi(t_x) + (1 -\eps) \Psi(t_x) ) \bigl)^{-1} \leq \eps ( \bQ - \Phi(t_x) ) )^{-1}  + (1 - \eps) ( \bQ - \Phi(t_x) ) )^{-1}
		\]
		so $(x,(1 -\eps)\Phi + \eps \Psi ) \in A_{T,L}$ for all $\eps \in [0,1]$.
		
		If we set $\Theta(t) = \frac{ \Psi(t) - \Phi(t) }{ \| \Psi - \Phi\|_\infty}$, then for all $\epsilon \in [0, \| \Psi - \Phi\|_\infty]$,
		\[
		(x,\Phi + \epsilon \Theta) = (x,(1 -\eps)\Phi + \eps \Psi) \in A_{t,L},
		\]
		Consider the function
		\[
		f(\Phi) = \int_0^{t_x} \langle \hat \Phi(t)^{-1},  \Phi'(t) \rangle \, dt.
		\]
		We will show that the directional derivatives of $f$ in the admissible unit direction $\Theta$ is uniformly bounded by some constant $C$ that only depends on the fixed parameters of the model and the bound $L$, 
		\[
		\bigg| \frac{d}{d\epsilon} f(\Phi + \epsilon \Theta) \Big|_{\epsilon = 0} \bigg| \leq C.
		\]
		Using [\ref{app:properties}, Proposition~\ref{prop:partialinverse}] to compute the derivative of the inverse,
		\begin{align*}
		\frac{d}{d\epsilon}f(\Phi + \epsilon \Theta) \bigg|_{\epsilon = 0} &= \frac{d}{d\epsilon}\int_0^{t_x} \langle (\hat \Phi(t) + \epsilon \hat \Theta)^{-1},  \Phi'(t) + \epsilon \Theta'(t) \rangle \, dt \bigg|_{\epsilon = 0}
		\\&= -\int_0^{t_x} \langle \hat\Phi(t)^{-1} \hat\Theta(t) \hat\Phi(t)^{-1},  \Phi'(t) \rangle \, dt + \int_0^{t_x} \langle \hat \Phi(t)^{-1},  \Theta'(t) \rangle \, dt
		\\&= -\int_0^{t_x} \int^n_t x(s) \langle \hat\Phi(t)^{-1} \Theta'(s) \hat\Phi(t)^{-1},  \Phi'(t) \rangle \, ds dt + \int_0^{t_x} \langle \hat \Phi(t)^{-1},  \Theta'(t) \rangle \, dt.
		\end{align*}
		Since $\tr(\bA \bB \bC) = \tr( \bC \bA \bB)$, we integrate by parts to conclude
		\begin{align*}
		&-\int_0^{t_x} \int^n_t x(s) \langle \hat\Phi(t)^{-1} \Phi'(t) \hat\Phi(t)^{-1},  \Theta'(s) \rangle \, ds dt 
		\\&= \int_0^{t_x} x(t) \langle \hat\Phi(t)^{-1} \Phi'(t) \hat\Phi(t)^{-1},  \Theta(t) \rangle \, dt  + \int_0^{t_x} \int^n_t \langle \hat\Phi(t)^{-1} \Phi'(t) \hat\Phi(t)^{-1},  \Theta(s) \rangle \, dx(s) dt
		\end{align*}
		and
		\begin{align*}
		\int_0^{t_x} \langle \hat \Phi(t)^{-1},  \Theta'(t) \rangle \, dt = -\langle \hat \Phi(t_x)^{-1},  \Theta(t_x)  \rangle  -\int_0^{t_x} x(t) \langle \hat\Phi(t)^{-1}  \Phi'(t) \hat\Phi(t)^{-1},  \Theta(t) \rangle \, dt.
		\end{align*}
		Since $\| \Theta \|_\infty = 1$, $\|\Phi'\|_\infty \leq 1$, and $\| \hat \Phi^{-1}(t) \|_\infty \leq \| \bQ - \Phi(t_x) \|_\infty \leq L$, we can replace each entry of the matrices in the integrand with its highest possible value  to get the crude upper bound
		\[
		\bigg| \frac{d}{d\epsilon} f(\Phi + \epsilon \Theta) \Big|_{\epsilon = 0} \bigg| \leq 4n^4 L^2.
		\]
		
		This upper bound holds for any starting point $\Phi$ and all admissible directions $\Theta$. We will now show that this implies that our functional is Lipschitz. Given $\Phi_1$ monotone paths $\Phi_2$, we have
		\[
		\Phi_2 = \Phi_1 + \|\Phi_1 - \Phi_2\|_\infty \Theta,
		\]
		where $\Theta = \frac{\Phi_2 - \Phi_1}{ \|\Phi_1 - \Phi_2\|_\infty  }$. Consider the function $g: [0, \|\Phi_1 - \Phi_2\|_\infty] \to \R$, 
		\[
		g(t) =  f(\Phi_1 + t \Theta).
		\]
		First notice that $\Phi_1 + t \Theta$ is a monotone path and $(x, \Phi_1 + t \Theta) \in A_{T,L}$ for $t \in [0, \|\Phi_1 - \Phi_2\|_\infty]$. For $t \in (0, \|\Phi_1 - \Phi_2\|_\infty)$ the uniform bound on the directional derivative implies
		\[
		|g'(t)| = \bigg| \frac{d}{d\epsilon}f(\Phi_1 + t \Theta + \epsilon \Theta) \Big|_{\epsilon = 0} \bigg| \leq 4n^4 L^2.
		\]
		Since $g(t)$ is continuous, the mean value theorem implies that
		\[
		\bigg| \int_0^{t_x} \langle \hat \Phi_1(t)^{-1},  \Phi_1'(t) \rangle \, dt - \int_0^{t_x} \langle \hat \Phi_2(t)^{-1},  \Phi_2'(t) \rangle \, dt \bigg| = |g(0) - g( \|\Phi_2 - \Phi_1 \|_\infty ) | \leq 4n^4 L^2 \|\Phi_1 - \Phi_2\|_\infty,
		\]
		which proves that $f(\Phi)$ is Lipschitz. 
	\\\\
	\noindent Combining the bounds proved in (a) to (d) completes the proof for \eqref{eq:lipphi}. 
	\\\\
	\textit{Lipschitz in $x$:} Showing the functional is Lipschitz in $x$ follows from a similar computation. Fix $\Phi(t)$ and consider $( x_1, \Phi), ( x_2, \Phi) \in A_{T,L}$. We now show that the functional is Lipschitz with respect to the $L^1$ norm on monotone functions,
	\begin{equation}\label{eq:lipx}
	| \sC( x_1, \Phi) - \sC( x_2, \Phi) | \leq C_L \| x_1 - x_2 \|_1.
	\end{equation}
	We will show that each term in $\sC(x,\Phi)$ is Lipschitz in $x$ for fixed $\Phi$. To make the dependence of $\hat \Phi$ on $x_1$ and $x_2$ explicit, we define
	\[
	\hat \Phi_{x_1}(t) :=  \int_t^n x_1(s) \Phi'(s) \, ds \quad \text{and} \quad \hat \Phi_{x_2}(t) :=  \int_t^n x_2(s) \Phi'(s) \, ds.
	\]
	
		\noindent (a)~~~ The matrix path satisfies $\|\Phi\|_\infty \leq 1$, so
		\begin{align*}
		|\langle \vh \vh^\trans, \hat \Phi_{x_1}(0) \rangle - \langle \vh \vh^\trans, \hat \Phi_{x_2}(0) \rangle| &\leq \int_0^n |x_1 - x_2| \cdot |\langle \vh \vh^\trans, \Phi'(t) \rangle  | \, dt 
		\\&\leq n^2 \| \vh \vh^\trans \|_\infty \| x_1 - x_2 \|_1.
		\end{align*}
		\noindent (b)~~~ The matrix path satisfies $\|\bxi'(\Phi) \|_\infty \leq \|\bxi'(1) \|_\infty$ and $\| \Phi' \|_\infty \leq 1$, so
		\begin{align*}
		\bigg| \int_0^n x_1(t) \langle \bxi'(\Phi(t)), \Phi'(t) \rangle \, dt - \int_0^n x(t) \langle \bxi'(\Phi(t)), \Phi'(t) \rangle \, dt \bigg| \leq n^2 \|\bxi'(1) \|_\infty \| x_1 - x_2 \|_1.
		\end{align*}
		\noindent(c)~~~ By observation \eqref{eq:notdependsup}, we can replace the bound with $\hat t = \sup( \{ t \leq T \mmm \| (\bQ - \Phi(t))^{-1} \|_\infty \leq L \} )$. We need to show
		\[
		f(x) := \log| \Phi(n) - \Phi(\hat t)| + \int_0^{\hat t} \langle \hat \Phi_{x}(t)^{-1},  \Phi'(t) \rangle \, dt,
		\]
		is Lipschitz in $x$. The log determinant term is independent of $x$, so we only need to show that the integral is Lipschitz in $x$. We will show that all directional derivatives of $x$ with respect to an admissible unit direction is bounded. Let $y(t)$ be another monotone function such that $y(t) = 1$ for all $t \geq \hat t$. It is easy to see that $( (1 - \epsilon) x(t) + \epsilon y(t), \Phi) \in A_{T,L}$, so $z(t) = \frac{y(t) - x(t)}{ \| x - y\|_\infty}$ is an admissible unit direction. Since $z(t) = 0$ for $t \geq \hat t$, Fubini's theorem implies that
		\begin{align*}
		\frac{d}{d\eps} f(x(t) + \epsilon z(t)) \bigg|_{\eps = 0} &= - \int_0^{\hat t} \int_t^{\hat t} z(s) \langle \hat \Phi_{x}(t)^{-1} \Phi'(s) \hat \Phi_{x}(t)^{-1} ,  \Phi'(t) \rangle \, ds dt
		\\&= - \int_0^{\hat t} \int_0^{s} z(s) \langle \hat \Phi_{x}(t)^{-1} \Phi'(t) \hat \Phi_{x}(t)^{-1} ,  \Phi'(s) \rangle \, dt ds.
		\end{align*}
		Since $\|z\|_1 = 1$, $\|\Phi'\|_\infty \leq 1$, and $\| \hat \Phi_x^{-1}(t) \|_\infty \leq \| \bQ - \Phi(\hat t) \|_\infty \leq L$ we can replace each entry of the matrices in the integrand with its highest possible value to get the crude upper bound
		\[
		\bigg| \frac{d}{d\epsilon} f(x + \epsilon z(t) ) \Big|_{\epsilon = 0} \bigg| \leq n^4 L^2.
		\]	
		The mean value theorem implies that
		\[
		| f(x_1(t)) - f(x_2(t) )| \leq  n^4 L^2 \| x_1 - x_2\|_1.
		\]
		\noindent Combining the bounds proved in (a) to (c) completes the proof for \eqref{eq:lipx}. 
\end{proof}

Since $\sC(x,\Phi)$ restricted to $A_{T,L}$ is Lipschitz continuous by Lemma~\ref{lem:lip}, the extreme value theorem implies that $\sC$ attains its minimum at some $(x_{T,L},\Phi_{T,L}) \in A_{T,L}$. We will show the global minimizer of $\sC(x,\Phi)$ over its domain lies in $A_{\hat T, \hat L}$ for some $\hat T$ and $\hat L$ that only depends on the fixed parameters of the model,
\[
\inf\{ \sC(x,\Phi) \mmm (x,\Phi) \in A_{T,L} \text{ for some $T \in [0,1)$ and $L > 0$}   \} = \inf\{ \sC(x,\Phi) \mmm (x,\Phi) \in A_{\hat T, \hat L} \}.
\]
This fact is enough to conclude that
\begin{equation}\label{eq:lwbdintform}
\inf_{x,\Phi} \sC(x, \Phi) = \inf_{A_{\hat T, \hat L}} \sC(x, \Phi) = \inf_{r,A_{\hat T, \hat L}} \sC_r(\vex,\vbQ) \geq \inf_{r,x,Q} \sC_r(\vex,\vbQ) .
\end{equation}
In the second inequality, we used the fact that $\sC(x, \Phi)$ is Lipchitz on $A_{\hat T, \hat L}$, so its value agrees with the limit points of discrete c.d.f. The bounds \eqref{eq:upbdintform} and \eqref{eq:lwbdintform} implies that
\[
\inf_{x,\Phi} \sC(x, \Phi) = \inf_{r,x,Q} \sC_r(\vex, \vbQ).
\]

It remains to find an explicit formula for $\hat T$ and $\hat L$. There quantities are derived from the necessary conditions satisfied by the minimizer of $\sC(x,\Phi)$ obtained by perturbing the critical points of $\sC(x,\Phi)$. 
\begin{lem}\label{lem:supportmin}
	Let $T \in [0,1)$ and $L > 0$. Let
	\[
	(x,\Phi) = \mathrm{argmin}_{x,\Phi \in A_{T,L} } \sC(x,\Phi)
	\] 
	If $\mu$ is the probability measure on $[0,n]$ associated with $x$, 
	\[
	x(t) = \mu( [0,t) ),
	\]
	and 
	\[
	\hat T = \sup \{ t \leq T \mmm \| (\bQ - \Phi(t) )^{-1} \|_\infty \leq L \}
	\]
	is the largest feasible point in the support of $\mu$, then 
	\[
	\mu \big( \big\{ t \leq \hat T \mmm \langle \vh \vh^\trans + \bxi'(\bQ), \bQ \rangle + 1 - \log| \bQ|  + \log | \bQ - \Phi(t)|  \geq 0 \big\} \big) = 1.
	\]
\end{lem}

\begin{proof} Let $(x,\Phi)$ be the minimizer of $\sC(x,\Phi)$ on $A_{T,L}$. The proof involves examining the critical point condition of $\sC(x,\Phi)$ by perturbing the c.d.f. We define 
	\[
	\hat T = \sup \{ t \leq T \mmm \| (\bQ - \Phi(t) )^{-1} \|_\infty \leq L \},
	\]
	to be the largest feasible point in the support the measure $\mu$ corresponding to $x$. If $y(t)$ is another c.d.f. such that $y(t) =1$ for $t \geq \hat T$, then $(1 - \epsilon)x(t) + \epsilon y(t) =  x(t) + \epsilon ( y(t) - x(t) )$ satisfies the condition $(1 - \epsilon)x(t) + \epsilon y(t) = 1$ for $t \geq T$, so
	\[
	((1 - \epsilon)x(t) +  \epsilon y(t), \Phi) \in A_{T,L} \quad \text{for all} \quad \epsilon \in [0,1].
	\]
	
	In particular, if we define $z(t) =  y(t) - x(t)$, then the right derivative
	\[
	\frac{d}{d\epsilon}  \sC( x(t) + \epsilon  z(t) ) \Big|_{\epsilon = 0} \geq 0
	\]
	since a perturbation of the minimizer in an admissible direction must be non-negative. Taking the directional derivative and using the independence of $\sC$ on $t_x$ explained in \eqref{eq:notdependsup}, we see that
	\begin{align*}
	&\frac{d}{d\epsilon} \sC( x(t) + \epsilon  z(t) ) \Big|_{\epsilon = 0} 
	\\&= \int_0^n z(t) \langle \vec{h}\vec{h}^\trans + \bxi'(\Phi(t)), \Phi'(t) \rangle \, dt - \int_0^{\hat T} \bigg\langle \hat \Phi(t)^{-1} \bigg( \int_t^n z(s) \Phi'(s) \, ds \bigg) \hat \Phi(t)^{-1}, \Phi'(t) \bigg\rangle \, dt.
	\end{align*}
	Since $z(t) = 0$ for $t \geq \hat T$, the second integral can be simplified using Fubini's theorem,
	\begin{alignat*}{2}
	&\int_0^{\hat T} \bigg\langle \hat \Phi(t)^{-1} \bigg( \int_t^n z(s) \Phi'(s) \, ds \bigg) \hat \Phi(t)^{-1}, \Phi'(t) \bigg\rangle \, dt 
	\\&= \int_0^{\hat T}\int_t^{\hat T} z(s) \langle \hat \Phi(t)^{-1} \Phi'(t) \hat \Phi(t)^{-1},  \Phi'(s) \rangle \, ds dt  &&\qquad \color{gray} \tr(\bA \bB \bC) = \tr(\bC\bA\bB)
	\\&= \int_0^{\hat T}\int_0^s z(s) \langle \hat \Phi(t)^{-1} \Phi'(t) \hat \Phi(t)^{-1},  \Phi'(s) \rangle \, dt ds  &&\qquad \color{gray} 0 \leq t \leq s \leq \hat T
	\\&= \int_0^n\int_0^t z(t) \langle \hat \Phi(s)^{-1} \Phi'(s) \hat \Phi(s)^{-1},  \Phi'(t) \rangle \, ds dt.  &&\qquad \color{gray} \text{relabel $s$ and $t$}
	\end{alignat*}
	If we define the matrix,
	\[
	\Psi(t) =  \vec{h} \vec{h}^\trans + \bxi'( \Phi(t) ) - \int_0^t \hat \Phi(s)^{-1} \Phi'(s) \hat \Phi(s)^{-1} \, ds
	\]
	then our computations above implies that
	\[
	\frac{d}{d\epsilon} \sC( x(t) + \epsilon  z(t) ) \Big|_{\epsilon = 0}  = \int_0^n z(t) \langle \Psi(t),  \Phi'(t) \rangle \, dt.
	\]
	Since $z(t) = y(t) - x(t)$, the critical point condition $\frac{d}{d\epsilon}  \sC( x(t) + \epsilon  z(t) ) \big|_{\epsilon = 0} \geq 0$ implies that
	\[
	\int_0^n y(t) \langle \Phi(t),  \Phi'(t) \rangle \, dt \geq  \int_0^n x(t) \langle \Phi(t),  \Phi'(t) \rangle \, dt
	\]
	for all functions $y(t)$. From this critical point condition, we are able to recover the support of $\mu(t)$, the measure corresponding to $x(t)$. In particular, if we define
	\[
	y(t) = \nu ([0,t] ) = \int_0^t d \nu(s),  
	\]
	then Fubini's theorem implies that
	\[
	\int_0^n y(t) \langle \Psi(t),  \Phi'(t) \rangle \, dt =  \int_0^n \int_0^t \langle \Psi(t),  \Phi'(t) \rangle \, d \nu(s) dt =  \int_0^n \int_s^n \langle \Psi(t),  \Phi'(t) \rangle \, dt d \nu(s)  .
	\]
	The critical point condition implies that
	\begin{equation}\label{eq:critptcdf}
	\int_0^n \int_s^n \langle \Psi(t),  \Phi'(t) \rangle \, dt d \nu(s)  \geq \int_0^n \int_s^n \langle \Psi(t),  \Phi'(t) \rangle \, dt d \mu(s)  .
	\end{equation}
	Since $\int_0^n  d \nu(s) = 1$ and $\int_0^n  d \mu(s) = 1$, 
	\[
	\int_0^n \int_0^n \langle \Psi(t),  \Phi'(t) \rangle \, dt d \nu(s)  = \int_0^n \int_0^n \langle \Psi(t),  \Phi'(t) \rangle \, dt d \mu(s)  ,
	\]
	we can subtract \eqref{eq:critptcdf} to conclude
	\[
	\int_0^n \int_0^s \langle \Psi(t),  \Phi'(t) \rangle \, dt d \nu(s)  \leq \int_0^n \int_0^s \langle \Psi(t),  \Phi'(t) \rangle \, dt d \mu(s)
	\]
	for all measures $\nu$ such that $\nu([0,t]) = 1$ for all $t \geq \hat T$. In particular, $\mu$ must be supported on points less than or equal to $\hat T$ that maximize the function
	\[
	f(s) := \int_0^s \langle \Psi(t),  \Phi'(t) \rangle \, dt .
	\]
	
	Since $f(0) = 0$, this means that the support of $\mu$ cannot contain points such that $f(s) < 0$. By monotonicity [\ref{app:properties}, Proposition~\ref{prop:monotone}], $\hat \Phi(t) \leq \bQ - \Phi(t)$ and therefore,
	\begin{align*}
	\int_0^s \langle \Psi(t), \Phi'(t) \rangle \, dt &= \int_0^s \langle \vh \vh^\trans + \bxi'(\bQ), \Phi'(t) \rangle \, dt - \int_0^s \int_0^t \langle \hat  \Phi(s)^{-1} \Phi'(s) \hat \Phi(s)^{-1}, \Phi'(t) \rangle \, ds  \, dt
	\\&\leq \int_0^s \langle \vh \vh^\trans + \bxi'(\bQ), \Phi'(t) \rangle \, dt - \int_0^s \int_0^t \langle (\bQ - \Phi(s))^{-1} \Phi'(s) (\bQ - \Phi(s))^{-1}, \Phi'(t) \rangle \, ds  \, dt
	\\&= \int_0^s \langle \vh \vh^\trans + \bxi'(\bQ), \Phi'(t) \rangle \, dt - \int_0^s \langle (\bQ - \Phi(t))^{-1}, \Phi'(t) \rangle  \, dt + \int_0^s \langle \bQ^{-1}, \Phi'(t) \rangle  \, dt
	\\&= \langle \vh \vh^\trans + \bxi'(\bQ) + \bQ^{-1}, \Phi(s) \rangle - \log |\bQ| + \log| \bQ - \Phi(s)|
	\\&\leq  \langle \vh \vh^\trans + \bxi'(\bQ) + \bQ^{-1}, \bQ \rangle - \log |\bQ| + \log| \bQ - \Phi(s)|.
	\end{align*}
	In particular, we have
	\[
	\mu\big( \{ s \leq \hat T \mmm \langle \vh \vh^\trans + \bxi'(\bQ) + \bQ^{-1}, \bQ \rangle - \log |\bQ| + \log| \bQ - \Phi(s)| < 0 \} \big) = 0.
	\]
\end{proof}

Lemma~\ref{lem:supportmin} implies that given $\Phi$, the support of $\mu$ cannot take values close to $n$, since $|\bQ - \Phi(t)| \to 0$ as $t \to n$. Proposition~\ref{prop:minimumcompact} follows immediately. 
\begin{proof}[Proof of Proposition~\ref{prop:minimumcompact}]
	
	If $T$ is the largest point in the support of $\mu$, then Lemma~\ref{lem:supportmin} implies
	\[
	\langle \vh \vh^\trans + \bxi'(\bQ), \bQ \rangle + 1 - \log| \bQ|  + \log | \bQ - \Phi(T)| \geq 0 \implies \lambda_{min} (\bQ - \Phi(T) ) \geq e^{-(\langle \vh \vh^\trans + \bxi'(\bQ), \bQ \rangle + n - \log| \bQ| ) },
	\]
	because $| \bQ - \Phi(T)| \leq \lambda_{min} ( \bQ - \Phi(T) )$ since the eigenvalues of $( \bQ - \Phi(T))$ are less than $1$.
	Since $\lambda_{min} (\bQ - \Phi(T) ) = (\lambda_{max} ( (\bQ - \Phi(T))^{-1}) )^{-1}$ and the infinity norm and operator norm on symmetric real valued square matrices are equivalent, i.e. $\| \cdot \|_\infty \leq \sqrt{n} \| \cdot \|_2$, we have that 
	\[
	\| (\bQ - \Phi(T) )^{-1}\|_\infty \leq \sqrt{n} \lambda_{max} ( (\bQ - \Phi(T))^{-1} ) \leq \sqrt{n} e^{(\langle \vh \vh^\trans + \bxi'(\bQ), \bQ \rangle + n - \log| \bQ| ) } =: \hat L.
	\]
	The universal upper bound $\hat L$ only depends on the fixed parameters of the model.
	
	Furthermore, since the matrix paths are parametrized by the trace, we have 
	\[
	\lambda_{min}( \bQ - \Phi(T) ) \leq \tr( \bQ - \Phi(T) ) = n  - T \implies \lambda_{max} ( ( \bQ - \Phi(T) )^{-1} ) \geq \frac{1}{n - T},
	\]
	so we must have
	\[
	T \leq n - \frac{1}{\sqrt{n}} e^{-(\langle \vh \vh^\trans + \bxi'(\bQ), \bQ \rangle + n - \log| \bQ| ) } < n.
	\]
	If we define $\hat T = n - C^{-1} e^{-(\langle \vh \vh^\trans + \bxi'(\bQ), \bQ \rangle + n - \log| \bQ| ) }$ then we have just shown that the largest point $T$ in the support of any minimizer of $\sC$ must satisfy
	\[
	T \leq n - \frac{1}{\sqrt{n}} e^{-(\langle \vh \vh^\trans + \bxi'(\bQ), \bQ \rangle + n - \log| \bQ| ) } =: \hat T < n.
	\]
	This gives the explicit formulas for the constants \eqref{eq:cosntantsbound} in Proposition~\ref{prop:minimumcompact}.
\end{proof}

\bibliographystyle{abbrv}

\newpage

\begin{appendices}
	
	\section{Appendix --- Elementary Facts About Symmetric Matrices}\label{app:properties}
	
	In this section, we state several facts about matrices that was used in the proof of the Crisanti--Sommers formula. Let $\bS_{+}^n$ be the space of symmetric positive semidefinite real valued $n \times n$ matrices.

	\subsection{Matrix Directional Derivatives}\label{app:matrixderivs} Let $\bC$ be an arbitrary symmetric matrix we use the following notation to denote the matrix derivatives of various functions in the direction $\bC$. Let $f : \bS_{+}^n \to \R$, we define
	\[
	\frac{d}{d \bA} f(\bA) := \frac{d}{dt} f(\bA + t \bC) \Big|_{t = 0} \in \R.
	\]
	
	We summarize several matrix derivatives that are used to compute the partial derivative of the functional in this paper. Let $\langle \bA, \bB \rangle = \tr(\bA \bB)$ denote the Frobenius inner product. We have
	\begin{enumerate}
		\item 
		\begin{equation}
		\frac{\partial}{\partial \bA} \langle \bB, \bA \rangle = \frac{d}{dt} \tr(\bB (\bA + t \bC) ) \Big|_{t = 0} = \langle \bB, \bC \rangle .
		\end{equation}
		\item
		\begin{equation}
		\frac{\partial}{\partial \bA} \Sum( \bxi( \bA ) ) = \frac{d}{dt} \Sum( \bxi(\bA + t \bC) ) \Big|_{t = 0} = \langle \bxi'( \bA), \bC\rangle.
		\end{equation}
		\item 
		\begin{equation}
		\frac{\partial}{\partial \bA} \log | \bA | = \frac{d}{dt} \log | \bA + t\bC | \Big|_{t = 0} = \langle \bA^{-1} , \bC\rangle .
		\end{equation}
		\item 
		\begin{equation}
		\frac{\partial}{\partial \bA} \langle \bA^{-1}, \bB \rangle = \frac{d}{dt} \tr(\bB (\bA + t \bC)^{-1} ) \Big|_{t = 0} = -\langle \bA^{-1} \bB \bA^{-1} , \bC\rangle.
		\end{equation}
	\end{enumerate}
	
	Before proving these derivatives, we write several basic operations in terms of the Frobenius inner product for symmetric matrices,
	
	\begin{enumerate}
		\item Let $\bA$ be a $n \times n$ matrix, and let $\vh \in \R^n$, we have
		\begin{equation}
		( \vh , \bA \vh ) = \vh^\trans \bA \vh = \tr( \vh \vh^\trans \bA ) = \langle \vh \vh^\trans, \bA \rangle.
		\end{equation}
		\item Let $\bA$, $\bB$ and $\bC$ be $n \times n$ matrices, we have
		\begin{equation}
		\langle \bA \odot \bB , \bC \rangle = \tr( (\bA \odot \bB) \times \bC ) = \tr( \bA \times (\bB \odot \bC) ) = \langle \bA,\bB  \odot \bC \rangle.
		\end{equation}
		\item Let $\bm{1}$ be the $n \times n$ matrix with all $1$'s, as a consequence of the above fact, we have
		\begin{equation}
		\Sum( \bA \odot \bB ) = \tr( \bm{1} \times (\bA \odot \bB) ) = \tr( (\bm{1} \odot \bA) \times \bB) ) = \tr( \bA \bB) = \langle \bA, \bB \rangle.
		\end{equation}
	\end{enumerate}
	
	We now compute the directional derivatives.
	
	\begin{proposition}[Derivative of the Trace]\label{prop:partialtrace}
		For any matrix $\bB$, the directional derivative of the trace in direction $\bC$ is given by
		\[
		\frac{\partial}{\partial \bA} \tr(\bB \bA) = \frac{d}{dt} \tr(\bB \times (\bA + t \bC) ) \Big|_{t = 0} = \tr(\bB\bC) .
		\]
	\end{proposition}
	\begin{proof}
		By linearity, we have
		\[
		\frac{d}{dt} \tr(\bB \times (\bA + t \bC) ) \Big|_{t = 0}  = \frac{d}{dt} \tr(\bB \bA)  + t \tr ( \bB \bC) \Big|_{t = 0} =\tr ( \bB \bC).  
		\]
	\end{proof}
	This immediately implies the directional derivative of a quadratic form.
	\begin{proposition}[Derivative of quadratic form]
		For $\vh \in \R^n$, the directional derivative of the quadratic form in direction $\bC$ is given by
		\[
		\frac{\partial}{\partial \bA} (\vh, \bA \vh)  = \frac{d}{dt} (\vh, (\bA + t\bC) \vh) \Big|_{t = 0} = \tr( \vh \vh^\trans \bC) .
		\]
	\end{proposition}
	\begin{proof}
		We can write the quadratic form as
		\[
		(\vh, \bA \vh) = \tr(\vh \vh^\trans \bA).
		\]
		The property follows immediately from Proposition~\ref{prop:partialtrace}.
	\end{proof}
	
	Next, we compute the matrix derivative with respect to a smooth function $\bm{G}(\bA)$, where the each coordinate of $\bm{G}_{ij}(\bA)$ is a smooth function single variable function of $\bA_{ij}$.
	
	\begin{proposition}[Derivatives of Matrix Valued Functions]\label{prop:derivtheta}
		Suppose each coordinate of $G(\bA)$ only depends on the corresponding coordinate of $\bA$. If $\bA$ is positive semidefinite, then
		\[
		\frac{\partial}{\partial \bA} \Sum( \bm{G}( \bA ) ) = \frac{d}{dt} \Sum( \bm{G}(\bA + t \bC) ) \Big|_{t = 0} = \tr( \bm{G}'( \bA) \bC) .
		\]
	\end{proposition}
	
	\begin{proof}
		We first write 
		\[
		\Sum( \bm{G}( \bA ) ) = \tr( \bm{1} \times \bm{G}(\bA) ).
		\]
		Therefore, by Proposition~\ref{prop:partialtrace},
		\[
		\frac{d}{dt} \tr(\bm{1} \times \bm{G}(\bA + t\bC) ) \Big|_{t = 0} = \tr(\bm{1} \times \frac{d}{dt} \bm{G}(\bA + t\bC) ) \Big|_{t = 0}  = \tr(\bm{1} \times \bm{G}'(\bA) \odot \bC ) = \tr(\bm{G}'(\bA) \bC ).
		\]
		We used the fact that each coordinate of $\bm{G}(\bA)$ only depends on the corresponding coordinate of $\bA$, so the matrices can be differentiated term by term.
	\end{proof}
	
	We now compute the derivative of the log determinant.
	\begin{proposition}[Derivative of Log Determinant]\label{prop:derivlogdet}
		For any matrix $\bA > 0$, the directional derivative of the log determinant in direction $\bC$ is given by
		\[
		\frac{\partial}{\partial \bA} \log | \bA | = \frac{d}{dt} \log | \bA + t\bC | \Big|_{t = 0} = \tr(\bA^{-1} \bC) .
		\]
	\end{proposition}
	\begin{proof} 
		By the chain rule, we have
		\[
		\frac{d}{dt} \log | \bA + t\bC | \Big|_{t = 0} =   \frac{1}{| \bA + t\bC |} \frac{d}{dt} | \bA + t\bC | \Big|_{t = 0}.
		\]
		By Jacobi's formula for invertible matrices, we have
		\[
		\frac{d}{dt} | \bA + t\bC |  = \tr( | \bA + t\bC | (\bA + t\bC)^{-1} \bC) = | \bA + t\bC |  \tr( (\bA + t\bC)^{-1} \bC),
		\]
		and therefore
		\[
		\frac{d}{dt} \log | \bA + t\bC | \Big|_{t = 0} = \frac{| \bA + t\bC |}{| \bA + t\bC |} \tr(  (\bA + t\bC)^{-1} \bC) \Big|_{t = 0} = \tr(A^{-1} \bC).
		\]
	\end{proof}
	
	Lastly, we compute the matrix derivative of the inverse
	
	\begin{proposition}[Derivative of Inverse]\label{prop:partialinverse}
		For any matrix $\bB$ and $\bA > 0$, the directional derivative of the inverse of $\bA$ in direction $\bC$ is given by
		\[
		\frac{\partial}{\partial \bA} \tr(\bB \bA^{-1}) = \frac{d}{dt} \tr(\bB \times (\bA + t \bC)^{-1} ) \Big|_{t = 0} = -\tr(\bA^{-1} \bB \bA^{-1} \bC) = -\tr(\bA^{-1} \bC \bA^{-1} \bB) .
		\]
	\end{proposition}
	
	\begin{proof}
		By definition, we have
		\[
		\frac{d}{dt} \tr(\bB \times (\bA + t \bC)^{-1} ) \Big|_{t = 0} = \lim_{t \to 0} \frac{\tr(\bB (\bA + t \bC)^{-1} ) - \tr(\bB \bA^{-1})  }{t} = \lim_{t \to 0}\tr \bigg( \frac{\bB (\bA + t \bC)^{-1} - \bB \bA^{-1} }{t} \bigg)
		\]
		which simplifies to
		\[
		\lim_{t \to 0}\tr \bigg( \bB (\bA + t \bC)^{-1} \frac{ \bA -  (\bA + t \bC)}{t} \bA^{-1} \bigg) = -\tr ( \bB \bA^{-1} \bC \bA^{-1} ) = -\tr ( \bA^{-1} \bB \bA^{-1} \bC  ) .
		\]
		Since $\tr(\bA \bB \bC) = \tr( \bC \bA \bB)$, the derivative is also equal to $-\tr ( \bA^{-1} \bC\bA^{-1} \bB   ) $
	\end{proof}
	
	\begin{proposition}[Quadratic Form of Inverse]
		For any vectors $\vh \in \R^n$ and $\bA > 0$, the directional derivative of the trace in direction $\bC$ is given by
		\[
		\frac{\partial}{\partial \bA} (\vh, \bA^{-1} \vh) = \frac{d}{dt} (\vh, (\bA + t \bC)^{-1} \vh) \Big|_{t = 0} = \tr(\bA^{-1} \vh\vh^\trans \bA^{-1} \bC) .
		\]
	\end{proposition}
	
	\begin{proof}
		We can write the quadratic form as
		\[
		(\vh, \bA \vh) = \tr(\vh \vh^\trans \bA).
		\]
		The property follows immediately from Proposition~\ref{prop:partialinverse}.
	\end{proof}
	
	\subsection{Critical Point Conditions} If $\bA$ is an interior minimizer of $f$, then
	\[
	\frac{d}{d t} f(\bA + t\bC) \Big|_{t = 0} = 0
	\]
	for all directions $\bC$. This is because both $\bC$ and $-\bC$ are admissible variations at an interior point.	The following results will be used to derive the matrix equalities in the critical point conditions.
	
	\begin{proposition}\label{prop:direcchar}
		If $\tr(\bA \bC) \leq \tr(\bB \bC)$ for all symmetric matrices $\bC$ then $\bA \leq \bB$. 
	\end{proposition}
	
	\begin{proof}
		Suppose that $\tr(\bA \bC) \leq \tr(\bB \bC)$ for all symmetric matrices $\bC$ but $\bB - \bA$ is not positive semidefinite. That is, there exists a vector $v$ such that
		\[
		v^\trans(\bB - \bA)v < 0.
		\]
		Consider the symmetric matrix $\bC = vv^\trans$. Therefore, our assumption implies,
		\[
		0 \leq \tr( (\bB - \bA) \bC) = \tr( (\bB - \bA) v v^\trans) = \tr(  v^T(\bB - \bA) v) < 0 
		\]
		which is a contradiction. Therefore, $\bB - \bA$ must be positive semidefinite.	
	\end{proof}
	
	This implies the following two sided version of the result,
	
	\begin{proposition}\label{prop:symmetrictest}
		If $\tr(\bA \bC) = \tr(\bB \bC)$ for all symmetric matrices $\bC$, then $\bA = \bB$. 
	\end{proposition}
	
	\begin{proof}
		If $\tr(\bA \bC) = \tr(\bB \bC)$, then by Proposition~\ref{prop:direcchar},
		\[
		\tr(\bA \bC) \leq \tr(\bB \bC) \implies \bA \leq \bB \qquad \text{ and } \qquad \tr(\bB \bC) \leq \tr(\bA \bC) \implies \bB \leq \bA.
		\]
		Therefore, $\bA - \bB$ is both positive semidefinite and negative semidefinite, so all of its eigenvalues are $0$. This implies that $\bA = \bB$.
	\end{proof}
	
	\subsection{Properties of Positive Semidefinite Matrices}
	
	\begin{proposition}[Log determinants are concave]\label{prop:logconcave}
		Let $\bA > 0$ and suppose $\bC$ is a symmetric matrix. The function
		\[
		f(x) = \log| \bA + x\bC |
		\]
		is concave in its domain. In particular, we have
		\[
		\log|\bA| + \tr(\bA^{-1} \bC) \geq \log| \bA + \bC|
		\]
		for all $\bC$ such that $ |\bA + \bC | > 0$. 
	\end{proposition}
	
	\begin{proof}
		It suffices to show $f''(x) \leq 0$ whenever $\bA + x\bC > 0$. These derivatives are the first and second directional derivatives of $\log| \bA + x\bC |$ in the direction $\bC$,
		\[
		f'(x) = \tr ( (\bA + x\bC)^{-1} \bC ) \qquad f''(x) = -\tr ( (\bA + x\bC)^{-1} \bC(\bA + x\bC)^{-1} \bC ).
		\]
		Using the eigendecomposition of $(\bA + x\bC)^{-1}$, we can express it as $(\bA + x\bC)^{-1} = \bB \bB^\trans$ for some matrix $\bB$. Therefore,
		\begin{align*}
		\tr ( (\bA + x\bC)^{-1} \bC(\bA + x\bC)^{-1} \bC ) &= \tr (  \bC (\bA + x\bC)^{-1} \bC(\bA + x\bC)^{-1}) 
		\\&=  \tr (  \bC (\bA + x\bC)^{-1} \bC\bB \bB^\trans) 
		\\&= \tr (  (\bC \bB )^\trans (\bA + x\bC)^{-1} \bC\bB ).
		\end{align*}
		It is easy to see that $ (\bC \bB )^\trans (\bA + x\bC)^{-1} \bC\bB$ is a positive definite matrix because
		\[
		v^\trans (\bC \bB )^\trans (\bA + x\bC)^{-1} \bC\bB v = (\bC \bB v )^\trans (\bA + x\bC)^{-1} \bC\bB v > 0
		\]
		for all $v \in \R^n$ since $(\bA + x\bC)^{-1}$ is positive definite. Therefore, the sum of its eigenvalues are positive, so
		\[
		\tr ( (\bA + x\bC)^{-1} \bC(\bA + x\bC)^{-1} \bC ) = \tr (  (\bC \bB )^\trans (\bA + x\bC)^{-1} \bC\bB ) > 0. 
		\]
		
		Since $f(x)$ is concave, it lies below its tangent lines so
		\[
		f(t) \leq f(0) + f'(0)t \text{ for all $t$ in the domain of $f(t)$}.
		\]
		If $|\bA + \bC | > 0$ then we can take $t = 1$ to conclude
		\[
		\log|\bA| + \tr(\bA^{-1} \bC) \geq \log| \bA + \bC|.
		\]
	\end{proof}
	
	\begin{proposition}[$\Sum(\bxi(\bA) )$ is convex]\label{prop:xiconvex}
		Suppose $\bC$ is a symmetric matrix. The function
		\[
		f(x) = \Sum ( \bxi(\bA + x \bC) )
		\]
		is convex. In particular, we have
		\[
		\Sum ( \bxi(\bA) ) + \tr(  \bxi'(\bA) \bC) \leq  \Sum ( \bxi(\bA + \bC) )
		\]
		for all $\bC$. 
	\end{proposition}
	
	\begin{proof}
		Since $\beta_p$ are positive and $\beta_p = 0$ for all odd $p$, $\bxi(\bA)$ is a convex function in each of its coordinates. Since the finite sum of convex functions are convex, 
		\[
		f(x) = \Sum ( \bxi(\bA + x \bC) )
		\]
		is convex. Since each entry of $\bxi$ is a convex function, we have
		\[
		\bxi(\bA)_{ij} + \bxi'(\bA)_{ij} \bC_{ij} \leq \bxi(\bA + \bC)_{ij}.
		\]
		Summing over $i,j \leq n$ implies
		\[
		\Sum ( \bxi(\bA) ) + \tr(  \bxi'(\bA) \bC) \leq  \Sum ( \bxi(\bA + \bC) ).
		\]
	\end{proof}
	
	\begin{proposition}[Inverse Matrices are Convex]\label{prop:inverseconvex}
		Let $\bA$ be a positive definite matrix, and suppose that $\bB, \bC \in \bS_+^n$ also satisfy $(\bA - \bB)^{-1} > 0$ and $(\bA - \bC)^{-1} > 0$. Then inverting matrices are convex,
		\[
		(\bA - ( \epsilon \bB + (1 - \epsilon) \bC ))^{-1} \leq \epsilon (\bA - \bB)^{-1} + (1 - \epsilon) (\bA - \bC)^{-1}.
		\]
	\end{proposition}
	
	\begin{proof}
		Let $v \in \R^n$ and $\bA, \bB \in \bS_+^n$. It suffices to show that
		\[
		f(t) = \tr( (\bA - t \bB)^{-1} v v^\trans)
		\]
		satisfies $f''(t) \geq 0$ for all $t$ in the domain. By the chain rule and Proposition~\ref{prop:partialinverse},
		\[
		f'(t) = \tr( (\bA - t \bB)^{-1} \bB (\bA - t \bB)^{-1} v v^\trans)
		\]
		and 
		\[
		f''(t) = 2 \tr( (\bA - t \bB)^{-1} \bB (\bA - t \bB)^{-1} \bB  (\bA - t \bB)^{-1} v v^\trans).
		\]
		Since $(\bA - t \bB)^{-1}$ is positive definite, we have
		\[
		f''(t) = 2 \tr( (\bB(\bA - t \bB)^{-1}v)^\trans (\bA - t \bB)^{-1} (\bB  (\bA - t \bB)^{-1} v) ) > 0. 
		\]
		
		The result in the proposition follows immediately. Let $v \in \R^n$, $\bA$ be a positive definite matrix, and suppose that $\bB, \bC \in \bS_+^n$ also satisfy $(\bA - \bB)^{-1} > 0$ and $(\bA - \bC)^{-1} > 0$. We have shown that
		\[
		g(t) = \tr( (\bA - \bC - t (\bB - \bC) )^{-1} v v^\trans) = v^\trans (\bA - \bC - t (\bB - \bC) )^{-1} v 
		\]
		is a convex function in for $t \in [0,1]$ so
		\[
		v^\trans(\bA - ( \epsilon \bB + (1 - \epsilon) \bC ))^{-1} v = g(\epsilon) \leq (1 - \epsilon) g(0) + \epsilon g(1) = \epsilon x^\trans (\bA - \bB)^{-1} v + (1 - \epsilon) v^\trans (\bA - \bC)^{-1} v.
		\]
		This holds for all $v \in \R^n$, so
		\[
		(\bA - ( \epsilon \bB + (1 - \epsilon) \bC ))^{-1} \leq \epsilon (\bA - \bB)^{-1} + (1 - \epsilon) (\bA - \bC)^{-1}.
		\]
	\end{proof}
	
	\begin{proposition}[Upper Bound on the Determinant]\label{prop:AMGM}
		If $\bA$ is a positive definite, then
		\[
		|\bA| \leq \Big( \frac{ \tr(\bA) }{n} \Big)^n.
		\]
	\end{proposition}
	
	\begin{proof}
		Since $\bA$ is positive definite, its eigenvalues $\lambda_1, \dots, \lambda_n$ are positive. Therefore, by the AM--GM inequality,
		\[
		|\bA|^{1/n} = \Big( \prod_{i = 1}^n \lambda_j \Big)^{1/n} \leq \frac{\sum_{i = 1}^n \lambda_i}{n} = \frac{ \tr(\bA) }{n}.
		\]
	\end{proof}
	
	\begin{proposition}[Admissible Perturbations of Positive Definite Matrices]\label{prop:pertposdef}
		If $\bA$ is a positive definite, then for all symmetric matrices $\bC$, there exists a $\epsilon^*$ such that
		\[
		\bA + \epsilon \bC
		\]
		is also positive definite for all $\epsilon < \epsilon^*$.
	\end{proposition}
	
	\begin{proof}
		We will show that $v^\trans (\bA + \epsilon \bC) v > 0$ for all $v \in \R^n$ and all $\epsilon$ sufficiently small. Since $\bA$ is positive definite, we have
		\[
		v^\trans (\bA + \epsilon \bC) v \geq \lambda_{min}(\bA) \|v\|^2 - \epsilon \| \bC\|_{\infty} \|v\|^2 = (\lambda_{min}(\bA)  - \epsilon \| \bC\|_{\infty} ) \|v\|^2 .
		\]
		Since $\lambda_{min}(\bA) > 0$, setting $\epsilon < \frac{\lambda_{min}(\bA)}{  \| \bC\|_{\infty} }$ guarantees $v^\trans (\bA + \epsilon \bC) v  > 0$. 
	\end{proof}
	
	\begin{proposition}[Hadamard Product of Positive Definite Matrices]\label{prop:gapxi}
		If for each $j \leq n$, there exists a $p \geq 2$ such that $\beta_p(j) \neq 0$, and both $|\bQ_\ell| > 0$ and $|\bQ_{\ell} - \bQ_{\ell - 1}| > 0$, then $|\bxi'(\bQ_{\ell}) - \bxi'(\bQ_{\ell - 1}) | > 0$.
	\end{proposition}
	
	\begin{proof}
		First recall the Hadamard product representation of $\bxi'(\bQ)$,
		\[
		\bxi'(\bQ) = \sum_{p \geq 2} p (\beta_p \beta_p^\trans) \odot (\bQ)^{\odot (p - 1)}.
		\]
		Using the difference of powers formula to factor term by term, we have
		\begin{align*}
		\bxi'(\bQ_\ell) - \bxi'(\bQ_{\ell - 1})&= \sum_{p \geq 2} p (\beta_p \beta_p^\trans) \odot \Big(  \bQ_\ell^{\odot (p - 1)} - \bQ_{\ell - 1}^{\odot (p - 1)} \Big)
		\\&=  (\bQ_\ell - \bQ_{\ell - 1}  ) \odot \sum_{p \geq 2} p (\beta_p \beta_p^\trans) \odot \sum_{0 \leq k \leq p - 2} \bQ_\ell^{\odot (p - 2) - k} \odot \bQ_{\ell-1}^{\odot k}.
		\end{align*}
		By the Schur product theorem, the above is the Hadamard product of positive semidefinite matrices, so it must be positive semidefinite. By our assumption on $\beta_p$, there exists a $M$ sufficiently large such that all entries of $\sum_{2 \leq p \leq M} p (\beta_p \beta_p^\trans)$ are positive. Therefore,
		\begin{align}
		\bxi'(\bQ_\ell) - \bxi'(\bQ_{\ell - 1}) &= (\bQ_\ell - \bQ_{\ell - 1}  ) \odot \sum_{2 \leq p \leq M} p (\beta_p \beta_p^\trans) \odot \sum_{0 \leq k \leq p - 2} \bQ_\ell^{\odot (p - 2) - k} \odot \bQ_{\ell-1}^{\odot k} \label{eq:hadamardfirst}
		\\&\quad + (\bQ_\ell - \bQ_{\ell - 1}  ) \odot \sum_{p > M} p (\beta_p \beta_p^\trans) \odot \sum_{0 \leq k \leq p - 2} \bQ_\ell^{\odot (p - 2) - k} \odot \bQ_{\ell-1}^{\odot k}  \label{eq:hadamardsecond}.
		\end{align}
		It suffices to show that the first matrix term \eqref{eq:hadamardfirst} is positive definite, because the second matrix term \eqref{eq:hadamardsecond} is positive semidefinite (the product of positive semidefinite matrices), and the sum of a positive definite matrix and a positive semidefinite matrix is positive definite. To prove this fact, we recall Oppenheim's inequality 
		, which states for all positive semidefinite matrices $\bA$ and $\bB$,
		\[
		\det(\bA \odot \bB) \geq \det(\bA) \prod_{j \leq n} \bB_{jj}.
		\]
		Therefore, we have the determinant of \eqref{eq:hadamardfirst} is bounded below by
		\begin{equation}\label{eq:hadamardlowbd}
		|\bQ_\ell - \bQ_{\ell - 1}  | \cdot \prod_{i \leq n} \bigg( \sum_{2 \leq p \leq M} p (\beta_p \beta_p^\trans) \odot \sum_{0 \leq k \leq p - 2} \bQ_\ell^{\odot (p - 2) - k} \odot \bQ_{\ell-1}^{\odot k} \bigg)_{ii}.
		\end{equation}
		We claim that each diagonal element appearing above is strictly positive. Since $\bQ_{\ell} > 0$, for each $p \geq 2$, we have
		\[
		\bA_p := \sum_{0 \leq k \leq p - 2} \bQ_\ell^{\odot (p - 2) - k} \odot \bQ_{\ell-1}^{\odot k}  = \bQ_{\ell}^{\odot (p - 2)} + \sum_{1 \leq k \leq p - 2} \bQ_\ell^{\odot (p - 2) - k} \odot \bQ_{\ell-1}^{\odot k}  > 0
		\]
		since the first term is the Hadamard power of a positive definite matrix and hence positive definite by the Schur product theorem. Since the diagonal elements of a positive definite matrix are all strictly positive, we have for all $i \leq n$,
		\begin{align*}
		\bigg( \sum_{2 \leq p \leq M} p (\beta_p \beta_p^\trans) \odot \bA_p \bigg)_{ii} \geq  \bigg( \sum_{2 \leq p \leq M} p (\beta_p \beta_p^\trans) \bigg)_{ii} \cdot \min_{2 \leq p \leq M} \bigl( \bA_p\bigr)_{ii} > 0.
		\end{align*} 
		Substituting this fact into \eqref{eq:hadamardlowbd} and using the fact $|\bQ_\ell - \bQ_{\ell - 1}  | > 0$ implies \eqref{eq:hadamardfirst} is positive definite, so the superadditivity of the determinant for positive semidefinite matrices implies
		\[
		| \bxi'(\bQ_\ell) - \bxi'(\bQ_{\ell - 1} ) | > 0
		\]
		as required. 
	\end{proof}

	\begin{proposition}[Monotonicity of Products]\label{prop:monotone}
		If $\bA, \bC \in \bS_+^n$, then
		\[
		\tr( \bA \bC ) \geq 0.  
		\]
		In particular, if $\bB \geq \bA$, then
		\[
		\tr( \bB \bC ) \geq \tr( \bA\bC ). 
		\]
	\end{proposition}

	\begin{proof}
		Consider the eigendecomposition $\bC = \bR \bL \bR^\trans$, where $\bL = \diag(\lambda_1, \dots, \lambda_n)$ is the diagonal matrix of eigenvalues and $\bR$ is an orthogonal matrix. Since $\bL$ is diagonal, we can write it as a sum of real valued vectors $v_1, \dots, v_n$,
		\[
		\bL = \sum_{i = 1}^n v_i v_i^\trans,
		\]
		where $v_i = \sqrt{\lambda_i} e_i$. Since $\tr(\bA \bB \bC) = \tr( \bC\bA \bB)$, the rank 1 decomposition above implies that
		\[
		\tr( \bA \bC ) = \tr( \bR^\trans \bA \bR \bL  ) = \sum_{i = 1}^n \tr( \bR^\trans \bA \bR v_i v_i^\trans ) = \sum_{i = 1}^n \tr( (\bR v_i)^\trans \bA (\bR v_i) ) \geq 0,
		\]
		since $\bA$ is positive semidefinite. If $\bB \geq \bA$, then $\bB - \bA$ is positive semidefinite, so
		\[
		\tr( (\bB - \bA) \bC ) \geq 0 \implies \tr( \bB \bC ) \geq \tr( \bA\bC ).
		\]
	\end{proof}

	\subsection{Calculus Results}

	\begin{proposition}\label{prop:interchangeinf} Recall the functions $\sP_r^\epsilon( \bL,\vbQ)$ defined in \eqref{eq:pertparisi} and $\sC_r^\epsilon(\vbQ)$ defined in \eqref{eq:pertcrisanti}. We have
	\[
	\lim_{\epsilon \to 0} \inf_{\Lambda,Q}  \sP_r^\epsilon( \bL,\vbQ) = \inf_{\Lambda,Q} \lim_{\epsilon \to 0} \sP_r^\epsilon(\bL,\vbQ)  =  \inf_{\Lambda,Q} \sP_r(\bL,\vbQ)
	\]
	and
	\[
	\lim_{\epsilon \to 0} \inf_{Q}  \sC_r^\epsilon(\vbQ) = \inf_{Q} \lim_{\epsilon \to 0} \sC_r^\epsilon(\vbQ)  =  \inf_{Q} \sC_r(\vbQ) .
	\]
	\end{proposition}
	
	\begin{proof}
		Since $\sP_r^\eps = \sP + \epsilon \sE_r$ and $\sE_r \geq 0$, $\sP_r^\eps$ is decreasing in $\epsilon$. The functional $\sP_r^\epsilon( \bL,\vbQ)$ is only defined for strictly monotone sequences $\vbQ$. If we restrict $\sP_r$ to strictly monotone sequences, and take the infimums only over $\vbQ$ with strictly increasing increments then
		\[
		\lim_{\epsilon \to 0} \inf_{\Lambda,Q}  \sP_r^\epsilon( \bL,\vbQ) = \inf_{\epsilon, \Lambda,Q}  \sP_r^\epsilon( \bL,\vbQ)  = \inf_{\Lambda,Q} \lim_{\epsilon \to 0} \sP_r^\epsilon(\bL,\vbQ) = \inf_{\Lambda,Q}  \sP_r( \bL,\vbQ)
		\]
		since $\inf_{\Lambda,Q}  \sP_r^\epsilon( \bL,\vbQ)$ is also decreasing in $\epsilon$. Furthermore, since $\sP_r$ is continuous and the space of $\vbQ$ with increasing increments is the closure of the paths with strictly increasing elements we can take the infimum over all paths of the form \eqref{eq:xqseq} without changing the value of $\inf_{\Lambda,Q}  \sP_r( \bL,\vbQ)$. 
		
		The proof for $\sC_r^\epsilon$ is identical. 
	\end{proof}
	
	\begin{proposition}[Uniform Continuity of $\sC$ with respect to Temperature]\label{prop:unifcontC}
		Let $\sC_{\bvb_1}(\vex,\vbQ)$ and $\sC_{\bvb_2}(\vex,\vbQ)$ denote the Crisanti--Sommers functional \eqref{eq:crisanti} with respect to $\bvb_1$ and $\bvb_2$. If
		\begin{equation}\label{eq:betaassump}
		\sum_{p \geq 2}  \| \vb^1_p \otimes \vb^1_p - \vb^2_p \otimes \vb^2_p \|_1 \leq \delta,
		\end{equation}
		then for any $\vx, \vbQ$,
		\[
		|\sC_{\bvb_1}(\vex,\vbQ) - \sC_{\bvb_2}(\vex,\vbQ)| \leq 2\delta.
		\]
	\end{proposition}
	
	\begin{proof}
		Since $\sC_r(\vex,\vbQ)$ only depends on temperature through $\bxi$, we only have to find a bound for
		\begin{equation}\label{eq:temperaturebound}
		\frac{1}{2} \bigg| \sum_{1 \leq k \leq r-1} x_k \cdot \Sum \big( \bxi_{\bvb_1} (\bQ_{k + 1}) - \bxi_{\bvb_1} (\bQ_{k})\big) -  \sum_{1 \leq k \leq r-1} x_k \cdot \Sum \big( \bxi_{\bvb_2} (\bQ_{k + 1}) - \bxi_{\bvb_2} (\bQ_{k})\big)   \bigg|,
		\end{equation}
		where 
		\[
		\bxi_{\bvb_1}(\bQ) =  \sum_{p \geq 2} (\vb^1_p \otimes \vb^1_p) \odot \bQ^{\circ p} \qquad \text{ and } \qquad \bxi_{\bvb_2}(\bQ) =  \sum_{p \geq 2} (\vb^2_p \otimes \vb^2_p) \odot \bQ^{\circ p}.
		\]
		Since $(\bQ_p)_{ij} \leq 1$ for any matrix $\bm 0 \leq \bQ_p \leq \bQ$, the assumption \eqref{eq:betaassump} implies
		\[
		\Big| \Sum\big( \bxi_{\bvb_1}(\bQ_p) - \bxi_{\bvb_2}(\bQ_p) \big) \Big| \leq \sum_{p \geq 2} \sum_{i,j \leq n} \Big| (\vb^1_p(i)\vb^1_p(j) - \vb^2_p(i)\vb^2_p(j)) \Big| \leq \delta.
		\]
		Using summation by parts, we see that 
		\[
		\sum_{1 \leq k \leq r-1} x_k \cdot \Sum \big( \bxi (\bQ_{k + 1}) - \bxi (\bQ_{k})\big) = -\sum_{1 \leq k \leq r-1} (x_{k} - x_{k-1}) \Sum \big( \bxi (\bQ_{k}) ) + x_{r - 1} \Sum( \bxi(\bQ_r) ),
		\]
		so \eqref{eq:temperaturebound} is bounded by
		\[
		\sum_{1 \leq k \leq r-1} (x_{k} - x_{k-1}) \Big| \Sum \big( \bxi_{\bvb_1} (\bQ_{k}) ) - \Sum \big( \bxi_{\bvb_2} (\bQ_{k}) ) \Big| + x_{r - 1} | \Sum( \bxi_{\bvb_1}(\bQ_r) ) - \Sum( \bxi_{\bvb_2}(\bQ_r) ) | \leq 2\delta.
		\]
	\end{proof}
	
	\begin{proposition}[Uniform Continuity of $\sP$ with respect to Temperature]\label{prop:unifcontP}
		Let $\sP_{\bvb_1}(\bL,\vex,\vbQ)$ and $\sP_{\bvb_2}(\bL,\vex,\vbQ)$ denote the Parisi functional \eqref{eq:parisi} with respect to $\bvb_1$ and $\bvb_2$. If
		\begin{equation}\label{eq:betassump2}
		\sum_{p \geq 2}  \| \vb^1_p \otimes \vb^1_p - \vb^2_p \otimes \vb^2_p \|_1 \leq \delta,
		\end{equation}
		then
		\[
		\Big|\inf_{r,\Lambda,x,Q} \sP_{\bvb_1}(\bL,\vex,\vbQ) - \inf_{r,\Lambda,x,Q} \sP_{\bvb_2}(\bL,\vex,\vbQ) \Big| \leq \delta.
		\]
	\end{proposition}
	
	\begin{proof}
		Recall that the Parisi functional is a the limit of the free energy,
		\[
		\lim_{N \to \infty} \frac{1}{N} \E \log \int_{\bQ} \exp H^{\bvb}_N(\sigma) \, d \lambda_N = \inf_{r,\Lambda,x,Q} \sP_{\bvb}(\bL,\vex,\vbQ) .
		\]
		We can use Gaussian interpolation to prove uniform continuity. Consider the Hamiltonian,
		\[
		H_t(\sigma) = \sqrt{t} H^{\bvb_1}_N(\sigma) + \sqrt{1 - t} H^{\bvb_2}_N(\sigma),
		\]
		and the interpolating free energy,
		\[
		\varphi( t ) = \frac{1}{N} \E \log \int_{\bQ} \exp H_t(\sigma) \, d \lambda_N .
		\]
		Differentiating with respect to $t$ and integrating by parts, we see that
		\[
		\varphi'( t ) =  \frac{1}{N} \Big\langle \frac{d}{dt} H_t(\sigma) \Big\rangle_t = \frac{1}{2} \Big\langle \Sum(\bxi_{\bvb_1}(\bR_{1,1}) - \bxi_{\bvb_2}(\bR_{1,1}) ) -\Sum ( \bxi_{\bvb_1}(\bR_{1,2}) - \bxi_{\bvb_2}(\bR_{1,2})) )  \Big\rangle_t
		\]
		where $\langle \cdot \rangle_t$ is the Gibbs average proportional to $e^{H_t(\sigma)}$. Since $\|\bR\|_\infty \leq 1$, the assumption \eqref{eq:betassump2} implies
		\[
		\Big| \Sum\big( \bxi_{\bvb_1}(\bR) - \bxi_{\bvb_2}(\bR) \big) \Big| \leq \sum_{p \geq 2} \sum_{i,j \leq n} \Big| (\vb^1_p(i)\vb^1_p(j) - \vb^2_p(i)\vb^2_p(j)) \Big| \leq \delta.
		\]
		Therefore, $|\varphi'( t )| \leq \delta$, so
		\[
		\Big| \frac{1}{N} \E \log \int_{\bQ} \exp H^{\bvb_1}_N(\sigma) \, d \lambda_N - \frac{1}{N} \E \log \int_{\bQ} \exp H^{\bvb_2}_N(\sigma) \, d \lambda_N \Big| \leq \delta,
		\]
		and taking limits implies
		\[
		\Big|\inf_{r,\Lambda,x,Q} \sP_{\bvb_1}(\bL,\vex,\vbQ) - \inf_{r,\Lambda,x,Q} \sP_{\bvb_2}(\bL,\vex,\vbQ) \Big| \leq \delta.
		\]
	\end{proof}
	
	\section{Derivatives of $\sP_r$ and $\sC_r$}
	
	\subsection{Derivatives of $\sP_r$ with respect to $Q$}\label{app:derivativesP}
	We use summation by parts, to write $\sP_r(\bL,\vex, \vbQ)$ as
	\begin{align}
	\sP_r(\bL,\vex, \vbQ) &= \frac{1}{2}\Big[ \tr(\bL \bQ) - n + \tr( (\vh\vh^{\trans} + \bxi'(\bQ_1)) \bL_1^{-1} )  - \sum_{2 \leq k \leq r} \Big( \frac{1}{x_k} - \frac{1}{x_{k - 1}} \Big) \log | \bL_k |  - \frac{1}{x_1} \log | \bL_1| \notag \\
	&\quad  - \sum_{1 \leq k \leq r-1} x_k \cdot \tr \big( \bm{1} \times (\btheta (\bQ_{k + 1}) - \btheta (\bQ_{k})) \big)\Big] \label{eq:derivp}.
	\end{align}
	Since $(\bL_k)_{k = 1}^r$ is a function of $(\bQ_k)_{k = 1}^{r}$, for $1 \leq p \leq r - 1$ we have
	\begin{enumerate}
		\item [(1)] If $p < \ell$
		\[
		\frac{d \bL_\ell}{d \bQ_{p}} = \bm{0}
		\]
		\item [(2)] If $p > \ell$
		\[
		\frac{d \bL_\ell}{d \bQ_{p}} = (x_p - x_{p - 1}) (\bxi''(\bQ_p) \odot \bC)
		\]
		\item [(3)] If $p = \ell$
		\[
		\frac{d \bL_\ell}{d \bQ_{p}} = x_p (\bxi''(\bQ_p) \odot \bC).
		\]
	\end{enumerate}
	Using the formulas in [Appendix~\ref{app:matrixderivs}] and the chain rule on \eqref{eq:derivp}, the derivatives in direction $2\bC$ (the constant $2$ is to cancel the constant factor of $\frac{1}{2}$ in front of $\sP_r$)  for $2 \leq p \leq r-1$ are given by
	\begin{align}
	\partial_{\bQ_p} \sP_r &= -(x_p - x_{p - 1}) \tr \Big( \bL_1^{-1}  \big( \vh \vh^T + \bxi'(\bQ_1) \big) \bL_1^{-1} \times (\bxi''(\bQ_p) \odot \bC) \Big)   \notag
	\\&\quad -(x_p - x_{p - 1}) \sum_{1 \leq \ell < p} \Big( \frac{1}{x_\ell} - \frac{1}{x_{\ell - 1}} \Big) \tr\big( \bL^{-1}_\ell (\bxi''(\bQ_p) \odot \bC)) \big) \notag
	\\&\quad - x_p \Big( \frac{1}{x_p} - \frac{1}{x_{p - 1}} \Big) \tr\big( \bL^{-1}_p (\bxi''(\bQ_p) \odot \bC) \big) - \frac{(x_p - x_{p - 1})}{x_1} \tr( \bL_1^{-1} (\bxi''(\bQ_p) \odot \bC) ) \notag
	\\&\quad + (x_p - x_{p - 1}) \tr( \bQ_{p} (\bxi''(\bQ_p) \odot \bC) ) \notag
	\\&= -(x_p - x_{p - 1}) \tr \Big( \bL_1^{-1}  \big( \vh \vh^T + \bxi'(\bQ_1) \big) \bL_1^{-1} \times (\bxi''(\bQ_p) \odot \bC) \Big)   \notag
	\\&\quad -(x_p - x_{p - 1})  \sum_{1 \leq k \leq p - 1} \frac{1}{x_k} \Big( \tr( (\bL^{-1}_{k} - \bL^{-1}_{k + 1}) \times (\bxi''(\bQ_p) \odot \bC))\Big)  \notag
	\\&\quad + (x_p - x_{p - 1}) \tr( \bQ_{p} \times (\bxi''(\bQ_p) \odot \bC) ), \notag
	\end{align}
	and the derivative for $p = 1$ is given by
	\begin{align}
	\partial_{\bQ_1} \sP_r &= -x_1 \tr \Big( \bL_1^{-1}  \big( \vh \vh^T + \bxi'(\bQ_1) \big) \bL_1^{-1} \times (\bxi''(\bQ_p) \odot \bC) \Big)  + x_1 \tr( \bQ_{1} \times (\bxi''(\bQ_p) \odot \bC) ) \notag.
	\end{align}
	
	\subsection{Derivative of $\sC_r$ with respect to $Q$}\label{app:derivativesC}
	
	We use summation by parts to write $\sC_r(\vex,\vbQ)$ as
	\begin{align}
	\sC_r(\vex,\vbQ) &= \frac{1}{2}\Big[ \tr( \vh \vh^\trans \bd_1) + \tr( \bQ_1 \bd_1^{-1} )  + \frac{1}{x_1} \log |\bd_1| + \sum_{2 \leq k \leq r-1} \log|\bd_{k}| \Big( \frac{1}{x_{k}} - \frac{1}{x_{k - 1}} \Big) \notag \\
	&\quad  + \sum_{1 \leq k \leq r-1} x_k \cdot \tr \Big(\bm{1} \times \big( \bxi (\bQ_{k + 1}) - \bxi (\bQ_{k})\big) \Big) \Big] \label{eq:derivc}.
	\end{align}
	Since $(\bd_k)_{k =  1}^r$ is a function of $(\bQ_k)_{k = 1}^r$, for $1 \leq p \leq r - 1$ we have
	\begin{enumerate}
		\item [(1)] If $p < \ell$
		\[
		\frac{d \bd_\ell}{d \bQ_{p}} = 0
		\]
		\item [(2)] If $p > \ell$
		\[
		\frac{d \bd_\ell}{d \bQ_{p}} = (x_{p - 1} - x_p) \bC
		\]
		\item [(3)] If $p = \ell$
		\[
		\frac{d \bd_\ell}{d \bQ_{p}} = -x_p \bC.
		\]
	\end{enumerate}
	Using the formulas in [Appendix~\ref{app:matrixderivs}] and the chain rule on \eqref{eq:derivc}, the derivatives in direction $2\bC$ (the constant $2$ is to cancel the constant factor of $\frac{1}{2}$ in front of $\sC_r$) for $2 \leq p \leq r-1$ are given by
	\begin{align}
	\partial_{\bQ_p}\sC_r &= (x_{p - 1} - x_p) \tr( \vh \vh^T \bC ) - (x_{p - 1} - x_p) \tr( \bd_1^{-1} \bQ_1 \bd_1^{-1} \bC ) + \frac{1}{x_1} (x_{p - 1} - x_p) \tr(\bd_1^{-1} \bC) \notag
	\\&\quad + \sum_{2 \leq k < p} (x_{p - 1} - x_p) \Big( \frac{1}{x_{k}} - \frac{1}{x_{k - 1}} \Big) \tr(\bd_{k}^{-1} \bC  ) - x_p \Big( \frac{1}{x_{p}} - \frac{1}{x_{p - 1}} \Big) \tr(\bd^{-1}_{p} \bC) \notag
	\\&\quad + (x_{p - 1} - x_p) \tr( \bxi'(\bQ_p) \bC ) \notag
	\\&= (x_{p - 1} - x_p) \tr( \vh \vh^T \bC ) - (x_{p - 1} - x_p) \tr( \bd_1^{-1} \bQ_1 \bd_1^{-1} \bC ) \notag
	\\&\quad - (x_{p - 1} - x_p) \sum_{1 \leq k \leq p-1} \frac{1}{x_k} \tr( (\bd^{-1}_{k+1} - \bd^{-1}_{k }) \bC) + (x_{p - 1} - x_p) \tr( \bxi'(\bQ_p) \bC ), \notag
	\end{align}
	and the derivative for $p = 1$ is given by
	\begin{align}
	\partial_{\bQ_1}\sC_r &= - x_1 \tr( \vh \vh^T \bC ) + x_1 \tr( \bd_1^{-1} \bQ_1 \bd_1^{-1} \bC )  - x_1 \tr( \bxi'(\bQ_1) \bC) \notag.
	\end{align}
	
\end{appendices}

\end{document}